\documentclass[11pt]{article}

\usepackage[hidelinks]{hyperref}
\hypersetup{
  colorlinks   = true, 
  urlcolor     = blue, 
  linkcolor    = blue, 
  citecolor   = red 
}
\usepackage{amsmath,amsthm,amssymb}
\usepackage{graphicx}
\usepackage{bbm}
\newcommand{\sy}{\boldsymbol{\Psi}}
\newcommand{\py}{\boldsymbol{\Phi}}

\usepackage{xcolor}
\usepackage[margin=1in]{geometry} 
\usepackage{enumerate,mathtools,mathrsfs,bm,graphicx}
\usepackage[numbers, super]{natbib}
\usepackage[hidelinks]{hyperref}
\newcommand{\N}{\mathbb{N}}									

\newcommand{\R}{\mathbb{R}}

\newcommand{\vertiii}[1]{{\left\vert\kern-0.25ex\left\vert\kern-0.25ex\left\vert #1 
    \right\vert\kern-0.25ex\right\vert\kern-0.25ex\right\vert}}
    
\newcommand{\overbar}[1]{\mkern 1.5mu\overline{\mkern-1.5mu#1\mkern-1.5mu}\mkern 1.5mu}
\newcommand{\inner}[2]{\langle #1, #2 \rangle}
\DeclarePairedDelimiter\abs{\lvert}{\rvert}					
\DeclarePairedDelimiter\norm{\lVert}{\rVert}				

\newtheorem{theorem}{Theorem}[section]
\newtheorem{corollary}{Corollary}[theorem]

\newtheorem{lemma}[theorem]{Lemma}
\newtheorem{proposition}[theorem]{Proposition}
\newtheorem*{remark}{Remark}
\newtheorem{definition}[theorem]{Definition}

\begin{document}
	\title{The Zero Viscosity Limit of Stochastic Navier-Stokes Flows}
	\author{Daniel Goodair \qquad \qquad Dan Crisan}
	\date{\today} 
	\maketitle
\setcitestyle{numbers}	
	\begin{abstract}
	We introduce an analogue to Kato's Criterion regarding the inviscid convergence of stochastic Navier-Stokes flows to the strong solution of the deterministic Euler equation. Our assumptions cover additive, multiplicative and transport type noise models. This is achieved firstly for the typical noise scaling of $\nu^\frac{1}{2}$, before considering a new parameter which approaches
zero with viscosity but at a potentially different rate. We determine the implications of this for our
criterion and clarify a sense in which the scaling by $\nu^\frac{1}{2}$ is optimal. To enable the analysis we prove the existence of probabilistically weak, analytically weak solutions to a general stochastic Navier-Stokes Equation on a bounded domain with no-slip boundary condition in three spatial dimensions, as well as the existence and uniqueness of probabilistically strong, analytically weak solutions in two dimensions. The criterion applies for these solutions in both two and three dimensions, with some technical simplifications in the 2D case.  
     
	\end{abstract}

\tableofcontents
\thispagestyle{empty}
\newpage

\setcounter{page}{1}
\addcontentsline{toc}{section}{Introduction}

\section*{Introduction}

The effect of viscosity in the presence of a boundary is an extensively studied and physically meaningful phenomenon, which is well summarised in [\cite{maekawa2016inviscid}] and has seen developments across [\cite{arakeri2000ludwig}, \cite{cebeci1977momentum}, \cite{clauser1956turbulent}, \cite{gie2012boundary}, \cite{gorbushin2018asymptotic}, \cite{iftimie2011viscous}, \cite{keller1978numerical}, \cite{lee2013effect}, \cite{malik1990numerical}, \cite{metivier2004small}, \cite{moore1963boundary}, \cite{morris1984boundary}, \cite{schetz2011boundary}, \cite{shaing2008collisional}, \cite{stevenson2002incipient}, \cite{temam2002boundary}, \cite{weinan2000boundary}] to list only a few contributions in the theory, observation and numerics of this analysis. As first proposed by Prandtl [\cite{prandtl1904flussigkeitsbewegung}] one may consider a thin layer around the boundary where the effects of viscosity remain significant, separate from the internal fluid which has inviscid behaviour. The width of the boundary layer formally scales with the square root of viscosity, fitting for the parabolic Navier-Stokes equations, and is described by the Prandtl Equations in the most commonly studied setting of a no-slip boundary condition (under which the fluid velocity is zero on the boundary) for Navier-Stokes. Indeed this is considered to be the most physically reasonable boundary condition for viscous flow as well discussed in [\cite{day1990no}, \cite{richardson1973no}, \cite{ruckenstein1983no}], though mathematically it is highly problematic when studying the vanishing viscosity limit. Whilst Prandtl's equations arise formally through an asymptotic analysis, it was Kato in his paper [\cite{kato1984remarks}] who rigorously underpinned this theory with results towards the inviscid limit of Navier-Stokes on a bounded domain. Kato's work shows that under sufficient smoothness of the initial condition, weak solutions of the Navier-Stokes equations with no-slip boundary condition converge to strong solutions of the Euler equation with impermeable boundary condition (the normal component of the fluid velocity is zero on the boundary) if and only if $$\lim_{\nu \rightarrow 0} \nu \int_0^T \norm{\nabla u^{\nu}_s}^2_{L^2(\Gamma_{c\nu})}ds = 0 $$
where $u^{\nu}$ is a weak solution of the Navier-Stokes equation with viscosity $\nu$ and $\Gamma_{c \nu}$ is a boundary strip of width $c \nu$  for $c>0$ fixed but arbitrary. This mathematically reflects the observed production of vorticity through large gradients of velocity at the boundary ([\cite{lyman1990vorticity}, \cite{morton1984generation}, \cite{serpelloni2013vertical}, \cite{tani1962production}, \cite{wallace2010measurement}]) and instability of the boundary layer ([\cite{dovgal1994laminar}, \cite{sanjose2019modal}, \cite{simpson1989turbulent}, \cite{timoshin1997instabilities}]). Forty years later, Kato's criterion remains the most fundamental result available in this area having seen only minor extensions such as [\cite{kelliher2007kato}] and [\cite{wang2001kato}]. Moreover we still have little to no understanding regarding the validity of this criterion, and whether general flows converge in the zero viscosity limit is one of the fundamental open problems of fluid mechanics.\\

Meanwhile there has been significant development in the theoretical analysis of stochastic fluid equations, and in particular those perturbed by a transport type noise (where the stochastic integral depends on the gradient of the solution). The paper of Brze\'{z}niak, Capinski and Flandoli [\cite{brzezniak1992stochastic}] in 1992 brought attention to the significance of fluid dynamics equations with transport noise, since generating much interest with the potential regularising effects as seen in [\cite{attanasio2011renormalized}, \cite{flandoli2013topics}, \cite{flandoli2015open}, \cite{flandoli2021delayed},  \cite{flandoli2021high},   \cite{bethencourt2022transport}]. Much more recently special consideration has been given to transport type stochastic perturbations due to their physical relevance, for example in the seminal works [\cite{holm2015variational}] and [\cite{memin2014fluid}]. In these papers Holm and M\'{e}min establish a new class of stochastic equations driven by transport type noise which serve as fluid dynamics models by adding uncertainty in the transport of fluid parcels to reflect the unresolved scales. Indeed the significance of such equations in modelling, numerical schemes and data assimilation continues to be well documented, see [\cite{alonso2020modelling}, \cite{cotter2020data}, \cite{cotter2018modelling}, \cite{cotter2019numerically}, \cite{crisan2021theoretical}, \cite{dufee2022stochastic}, \cite{flandoli20212d}, \cite{flandoli2022additive}, \cite{holm2020stochastic}, \cite{holm2019stochastic}, \cite{lang2022pathwise}, \cite{van2021bayesian}, \cite{street2021semi}]. This presents two key motivations for demonstrating Kato's Criterion for stochastic Navier-Stokes Equations:
\begin{enumerate}
    \item To extend Kato's meaningful boundary layer theory to the physically pertinent stochastic setting;
    \item To initiate considerations for a specific choice of (regularising) noise under which the criterion is satisfied. This would constitute a prodigious development in the fundamental open problem of resolving the boundary layer criterion. 
\end{enumerate}
We shall work with a Navier-Stokes Equation perturbed by either an It\^{o} type noise
\begin{equation} \label{number2equation}
    u_t - u_0 + \int_0^t\mathcal{L}_{u_s}u_s\,ds - \nu\int_0^t \Delta u_s\, ds + \nu^{\frac{1}{2}}\int_0^t \mathcal{G}(u_s) d\mathcal{W}_s + \nabla \rho_t = 0
\end{equation}
or a Stratonovich one
\begin{equation} \label{number3equation}
    u_t - u_0 + \int_0^t\mathcal{L}_{u_s}u_s\,ds - \nu\int_0^t \Delta u_s\, ds + \nu^{\frac{1}{2}}\int_0^t \mathcal{G}(u_s) \circ d\mathcal{W}_s + \nabla \rho_t = 0.
\end{equation}
Here $u$ represents the fluid velocity, $\rho$ the pressure and $\mathcal{L}$ the nonlinear term, $\mathcal{W}$ is a cylindrical Brownian motion and $\mathcal{G}$ is an operator valued mapping satisfying assumptions to be given in Subsection \ref{subsection assumptions}. The precise functional framework for the equation is given in Subsection \ref{functional framework subsection}, and notions of solution are defined in Subsection \ref{sub def}. As in Kato's original paper one must work with analytically weak solutions of Navier-Stokes, as strong solutions are only known to exist locally up to a time approaching zero with viscosity. The first main contribution of the paper is in showing the existence of such solutions, probabilistically weak in three spatial dimensions but strong in two dimensions. This is a typical reflection of the uniqueness in two dimensions which is unavailable in 3D, as seen in related works on stochastic Navier-Stokes and in more general SPDE theory, see [\cite{debussche2023consistent}, \cite{mikulevicius2005global},  \cite{rockner2022well}]). These works allow for a transport type noise, though do not extend to our result; in [\cite{rockner2022well}] the authors consider an abstract variational framework which in large resembles our setting, though the noise operator must have only a small dependency on first order derivatives (relative to viscosity in our context). Such an assumption is unavoidable in the It\^{o} case (\ref{number2equation}), however we alleviate it in a treatment of (\ref{number3equation}). This is critical as we are guided by an application to Stochastic Advection by Lie Transport (SALT), one of the aforementioned physical stochastic perturbation principles [\cite{holm2015variational}] which does not satisfy such a restriction. Debussche, Hug and M\'{e}min show the corresponding results in [\cite{debussche2023consistent}] for the Location Uncertainty scheme [\cite{memin2014fluid}], where the noise is specifically chosen so as to conserve energy through a backscattering term which helps the analysis. Moreover the atypical nature of this perturbation renders it difficult to recover the results for more traditional noise of the form (\ref{number2equation}), (\ref{number3equation}). In addition the paper [\cite{mikulevicius2005global}] of Mikulevicius and Rozovskii deals with a Stratonovich transport noise, though it imposes a stringent coercivity condition and considers only the whole space without boundary.\\

The theoretical analysis of fluid equations with a transport type noise on a bounded domain has proven of great challenge. For analytically strong solutions, prior to the authors' works of [\cite{goodair2022stochastic}, \cite{goodair2022existence1}] the only successful existence result of which we are aware was given in [\cite{brzezniak2021well}] where the authors assume again that the first order dependency is small (which is necessary in the It\^{o} case), but crucially that the noise terms are traceless under Leray Projection. This assumption is designed to circumvent the technical difficulties of a first order noise operator on a bounded domain, which is well elucidated in [\cite{goodair2022navier}] and the failure of this assumption for stochastic Lie transport is precisely why we could only show the (local) existence of analytically strong solutions to the SALT Navier-Stokes Equation on a bounded domain in vorticity form (the velocity form of the SALT Navier-Stokes Equation on a bounded domain remains an open problem). We are, however, successful in showing the (global) existence of analytically weak solutions for the velocity form in the present paper. The difference lies in the energy norm for the solutions: in the strong case this is produced from a $W^{1,2}$ inner product, a space in which the Leray Projector is not an orthogonal projection and indeed does not commute with the derivatives entering into consideration from this inner product. The presence of this Leray Projector prevents us from using the usual cancellation type arguments for transport noise, so without assuming sufficient smallness of the derivative dependency as discussed then we cannot achieve the necessary energy estimates. In contrast weak solutions exist in an energy space generated by the $L^2$ inner product, in which the Leray Projector is self-adjoint and the preceding ideas can be applied.\\

The second and titular contribution of the paper is in characterising the zero viscosity limit of these solutions, establishing a stochastic counterpart to Kato's Criterion. For this we scale the stochastic integral of (\ref{number2equation}), (\ref{number3equation}) with a parameter that must go to zero with viscosity, which is traditionally taken to be $\nu^\frac{1}{2}$ having been motivated in [\cite{kuksin2004eulerian}] as the only noise scaling which leads to non-trivial limiting measures (in the limit $t \rightarrow \infty$ and $\nu \rightarrow 0$) for an additive noise in two dimensions in the absence of a boundary. The significance of this scaling for energy balance is further underlined in [\cite{kuksin2005family}, \cite{kuksin2006remarks}, \cite{kuksin2008distribution}] and has been used to study the inviscid limit problem in [\cite{glatt2015inviscid}, \cite{luongo2021inviscid}]. We firstly establish our stochastic Kato's Criterion for this choice of scaling, demonstrating the equivalence of these conditions taken in expectation, again for a general stochastic term which can have arbitrarily high first order dependency. Furthermore a new criterion is presented dependent on an abstract scaling choice, and is specifically analysed in the case where this parameter is some exponent of viscosity. In particular we show that with the choice $\nu^{\beta}$, then for $\frac{1}{4} < \beta < \frac{1}{2}$ one requires the condition
$$ \lim_{\nu \rightarrow 0} \nu^{4\beta - 1}\mathbbm{E} \int_0^T \norm{\nabla u^{\nu}_s}^2_{L^2(\Gamma_{c\nu})}ds = 0$$ in order to deduce the convergence to Euler. For $\beta \geq \frac{1}{2}$ then the anticipated condition $$ \lim_{\nu \rightarrow 0} \nu\mathbbm{E} \int_0^T \norm{\nabla u^{\nu}_s}^2_{L^2(\Gamma_{c\nu})}ds = 0$$ is both necessary and sufficient for the convergence.\\

Overall the zero viscosity limit of solutions of the stochastic Navier-Stokes equation has so far received little treatment; the works of [\cite{bessaih2013inviscid}, \cite{glatt2015inviscid}] determine measure theoretic results for the problem posed in two dimensions with periodic boundary conditions and an additive noise. For Navier boundary conditions the convergence of stochastic Navier-Stokes to stochastic Euler in 2D has been proven in [\cite{cipriano2015inviscid}], matching the deterministic result as presented in [\cite{kelliher2006navier}], again for additive noise. In the classical case of a no-slip boundary condition, the only results of which we are aware are given in [\cite{luongo2021inviscid}, \cite{wang2023kato}]. A stochastic Kato type result is proven in each paper, for the limit to determinstic Euler in [\cite{luongo2021inviscid}] and stochastic Euler in [\cite{wang2023kato}], though once more only in 2D and with additive noise. Our result for a general first order noise in 2D and 3D thus represents a distinct addition to the literature.\\ 

The structure of the paper is as follows:
\begin{itemize}
    \item Section \ref{section preliminaries} is devoted to the setup of the problem in terms of notation, the functional framework of solutions and the assumptions we impose on the noise. Examples of noise satisfying these requirements are considered in Subsection \ref{section applications}, in addition to an explicit illustration that the SALT Navier-Stokes equation satisfies the assumptions. We also define our notions of solutions to the stochastic Navier-Stokes equation in Subsection \ref{sub def}, stating the key results concerning the existence and uniqueness of these solutions.
    \item In Section \ref{section zero viscous} we state and prove the main results regarding the zero viscosity limit, firstly for the noise scaling of $\nu^{\frac{1}{2}}$ and then for a general parameter. We inherit the techniques of Kato's original paper, particularly referring to the boundary corrector function.
    \item Section \ref{section weak solutions} contains the proofs of the existence and uniqueness results for weak solutions stated in Section \ref{section preliminaries}. The method of existence is classical in the sense that we consider a finite dimensional approximation with relative compactness arguments due to a tightness criterion, perhaps most similar to the approach of [\cite{rockner2022well}]. The pathwise uniqueness in 2D is then verified with the typical Ladyzhenskaya inequality for the nonlinear term, leading to probabilistically strong solutions as a result of a Yamada-Watanabe theorem. We are careful to rigorously justify the application of the It\^{o} Formula (Proposition \ref{rockner prop}) in the proof of uniqueness in 2D, which is absent in the aforementioned [\cite{mikulevicius2005global}] and emphasised by the authors in [\cite{debussche2023consistent}]. We find this important as if one were to assume such an It\^{o} Formula holds in 3D, they would immediately recover the continuity of solutions which is false in general.
    \item An appendix, Section \ref{section appendix}, containing useful results from the literature regarding stochastic partial differential equations and tightness criteria concludes the paper.
\end{itemize}

\section{Preliminaries} \label{section preliminaries}

\subsection{Elementary Notation}

In the following $\mathscr{O} \subset \R^N$ will be a smooth bounded domain, for $N$ either $2$ or $3$, equipped with Euclidean norm and  Lebesgue measure $\lambda$. We consider Banach Spaces as measure spaces equipped with their corresponding Borel $\sigma$-algebra. Let $(\mathcal{X},\mu)$ denote a general topological measure space, $(\mathcal{Y},\norm{\cdot}_{\mathcal{Y}})$ and $(\mathcal{Z},\norm{\cdot}_{\mathcal{Z}})$ be separable Banach Spaces, and $(\mathcal{U},\inner{\cdot}{\cdot}_{\mathcal{U}})$, $(\mathcal{H},\inner{\cdot}{\cdot}_{\mathcal{H}})$ be general separable Hilbert spaces. We also introduce the following spaces of functions. 
\begin{itemize}
    \item $L^p(\mathcal{X};\mathcal{Y})$ is the  class of measurable $p-$integrable functions from $\mathcal{X}$ into $\mathcal{Y}$, $1 \leq p < \infty$, which is a Banach space with norm $$\norm{\phi}_{L^p(\mathcal{X};\mathcal{Y})}^p := \int_{\mathcal{X}}\norm{\phi(x)}^p_{\mathcal{Y}}\mu(dx).$$ In particular $L^2(\mathcal{X}; \mathcal{Y})$ is a Hilbert Space when $\mathcal{Y}$ itself is Hilbert, with the standard inner product $$\inner{\phi}{\psi}_{L^2(\mathcal{X}; \mathcal{Y})} = \int_{\mathcal{X}}\inner{\phi(x)}{\psi(x)}_\mathcal{Y} \mu(dx).$$ In the case $\mathcal{X} = \mathscr{O}$ and $\mathcal{Y} = \R^N$ note that $$\norm{\phi}_{L^2(\mathscr{O};\R^N)}^2 = \sum_{l=1}^N\norm{\phi^l}^2_{L^2(\mathscr{O};\R)}, \qquad \phi = \left(\phi^1, \dots, \phi^N\right), \quad \phi^l: \mathscr{O} \rightarrow \R.$$ We denote $\norm{\cdot}_{L^p(\mathscr{O};\R^N)}$ by $\norm{\cdot}_{L^p}$ and $\norm{\cdot}_{L^2(\mathscr{O};\R^N)}$ by $\norm{\cdot}$.
    
\item $L^{\infty}(\mathcal{X};\mathcal{Y})$ is the class of measurable functions from $\mathcal{X}$ into $\mathcal{Y}$ which are essentially bounded. $L^{\infty}(\mathcal{X};\mathcal{Y})$ is a Banach Space when equipped with the norm $$ \norm{\phi}_{L^{\infty}(\mathcal{X};\mathcal{Y})} := \inf\{C \geq 0: \norm{\phi(x)}_Y \leq C \textnormal{ for $\mu$-$a.e.$ }  x \in \mathcal{X}\}.$$
    
    \item $L^{\infty}(\mathscr{O};\R^N)$ is the class of measurable functions from $\mathscr{O}$ into $\R^N$ such that $\phi^l \in L^{\infty}(\mathscr{O};\R)$ for $l=1,\dots,N$, which is a Banach Space when equipped with the norm $$ \norm{\phi}_{L^{\infty}}:= \sup_{l \leq N}\norm{\phi^l}_{L^{\infty}(\mathscr{O};\R)}.$$
    
      \item $C(\mathcal{X};\mathcal{Y})$ is the space of continuous functions from $\mathcal{X}$ into $\mathcal{Y}$.

      \item $C_w(\mathcal{X};\mathcal{Y})$ is the space of `weakly continuous' functions from $\mathcal{X}$ into $\mathcal{Y}$, by which we mean continuous with respect to the given topology on $\mathcal{X}$ and the weak topology on $\mathcal{Y}$. 
      
    \item $C^m(\mathscr{O};\R)$ is the space of $m \in \N$ times continuously differentiable functions from $\mathscr{O}$ to $\R$, that is $\phi \in C^m(\mathscr{O};\R)$ if and only if for every $N$ dimensional multi index $\alpha = \alpha_1, \dots, \alpha_N$ with $\abs{\alpha}\leq m$, $D^\alpha \phi \in C(\mathscr{O};\R)$ where $D^\alpha$ is the corresponding classical derivative operator $\partial_{x_1}^{\alpha_1} \dots \partial_{x_N}^{\alpha_N}$.
    
    \item $C^\infty(\mathscr{O};\R)$ is the intersection over all $m \in \N$ of the spaces $C^m(\mathscr{O};\R)$.
    
    \item $C^m_0(\mathscr{O};\R)$ for $m \in \N$ or $m = \infty$ is the subspace of $C^m(\mathscr{O};\R)$ of functions which have compact support.
    
    \item $C^m(\mathscr{O};\R^N), C^m_0(\mathscr{O};\R^N)$ for $m \in \N$ or $m = \infty$ is the space of functions from $\mathscr{O}$ to $\R^N$ whose $N$ component mappings each belong to $C^m(\mathscr{O};\R), C^m_0(\mathscr{O};\R)$.
    
        \item $W^{m,p}(\mathscr{O}; \R)$ for $1 \leq p < \infty$ is the sub-class of $L^p(\mathscr{O}, \R)$ which has all weak derivatives up to order $m \in \N$ also of class $L^p(\mathscr{O}, \R)$. This is a Banach space with norm $$\norm{\phi}^p_{W^{m,p}(\mathscr{O}, \R)} := \sum_{\abs{\alpha} \leq m}\norm{D^\alpha \phi}_{L^p(\mathscr{O}; \R)}^p,$$ where $D^\alpha$ is the corresponding weak derivative operator. In the case $p=2$ the space $W^{m,2}(\mathscr{O}, \R)$ is Hilbert with inner product $$\inner{\phi}{\psi}_{W^{m,2}(\mathscr{O}; \R)} := \sum_{\abs{\alpha} \leq m} \inner{D^\alpha \phi}{D^\alpha \psi}_{L^2(\mathscr{O}; \R)}.$$
    
    \item $W^{m,\infty}(\mathscr{O};\R)$ for $m \in \N$ is the sub-class of $L^\infty(\mathscr{O}, \R)$ which has all weak derivatives up to order $m \in \N$ also of class $L^\infty(\mathscr{O}, \R)$. This is a Banach space with norm $$\norm{\phi}_{W^{m,\infty}(\mathscr{O}, \R)} := \sup_{\abs{\alpha} \leq m}\norm{D^{\alpha}\phi}_{L^{\infty}(\mathscr{O};\R^N)}.$$
    
        \item $W^{m,p}(\mathscr{O}; \R^N)$ for $1 \leq p < \infty$ is the sub-class of $L^p(\mathscr{O}, \R^N)$ which has all weak derivatives up to order $m \in \N$ also of class $L^p(\mathscr{O}, \R^N)$. This is a Banach space with norm $$\norm{\phi}^p_{W^{m,p}(\mathscr{O}, \R^N)} := \sum_{l=1}^N\norm{\phi^l}_{W^{m,p}(\mathscr{O}; \R)}^p.$$ In the case $p=2$ the space $W^{m,2}(\mathscr{O}, \R^N)$ is Hilbertian with inner product $$\inner{\phi}{\psi}_{W^{m,2}(\mathscr{O}; \R^N)} := \sum_{l=1}^N \inner{\phi^l}{\psi^l}_{W^{m,2}(\mathscr{O}; \R)}.$$
    
          \item $W^{m,\infty}(\mathscr{O}; \R^N)$ is the sub-class of $L^\infty(\mathscr{O}, \R^N)$ which has all weak derivatives up to order $m \in \N$ also of class $L^\infty(\mathscr{O}, \R^N)$. This is a Banach space with norm $$\norm{\phi}_{W^{m,\infty}(\mathscr{O}, \R^N)} := \sup_{l \leq N}\norm{\phi^l}_{W^{m,\infty}(\mathscr{O}; \R)}.$$

    \item $W^{m,p}_0(\mathscr{O};\R), W^{m,p}_0(\mathscr{O};\R^N)$ for $m \in N$ and $1 \leq p \leq \infty$ is the closure of $C^\infty_0(\mathscr{O};\R),C^\infty_0(\mathscr{O};\R^N)$ in $W^{m,p}(\mathscr{O};\R), W^{m,p}(\mathscr{O};\R^N)$.
    
    \item $\mathscr{L}(\mathcal{Y};\mathcal{Z})$ is the space of bounded linear operators from $\mathcal{Y}$ to $\mathcal{Z}$. This is a Banach Space when equipped with the norm $$\norm{F}_{\mathscr{L}(\mathcal{Y};\mathcal{Z})} = \sup_{\norm{y}_{\mathcal{Y}}=1}\norm{Fy}_{\mathcal{Z}}$$ and is simply the dual space $\mathcal{Y}^*$ when $\mathcal{Z}=\R$, with operator norm $\norm{\cdot}_{\mathcal{Y}^*}.$
    
     \item $\mathscr{L}^2(\mathcal{U};\mathcal{H})$ is the space of Hilbert-Schmidt operators from $\mathcal{U}$ to $\mathcal{H}$, defined as the elements $F \in \mathscr{L}(\mathcal{U};\mathcal{H})$ such that for some basis $(e_i)$ of $\mathcal{U}$, $$\sum_{i=1}^\infty \norm{Fe_i}_{\mathcal{H}}^2 < \infty.$$ This is a Hilbert space with inner product $$\inner{F}{G}_{\mathscr{L}^2(\mathcal{U};\mathcal{H})} = \sum_{i=1}^\infty \inner{Fe_i}{Ge_i}_{\mathcal{H}}$$ which is independent of the choice of basis.

     \item For any $T>0$, $\mathscr{S}_T$ is the subspace of $C\left([0,T];[0,T]\right)$ of strictly increasing functions.

     \item For any $T>0$, $\mathcal{D}\left([0,T];\mathcal{Y}\right)$ is the space of c\'{a}dl\'{a}g functions from $[0,T]$ into $\mathcal{Y}$. It is a complete separable metric space when equipped with the metric $$d(\phi,\psi) := \inf_{\eta \in \mathscr{S}_T}\left[\sup_{t \in [0,T]}\left\vert \eta(t)- t\right\vert \vee \sup_{t \in [0,T]}\left\Vert \phi(t)-\psi(\eta(t)) \right\Vert_{\mathcal{Y}} \right]$$ which induces the so called Skorohod Topology (see [\cite{billingsley2013convergence}] pp124 for details).

\end{itemize}
We also now introduce some more precise spaces in greater detail.

\begin{definition}
We define $C^{\infty}_{0,\sigma}(\mathscr{O};\R^N)$ as the subset of $C^{\infty}_0(\mathscr{O};\R^N)$ of functions which are divergence-free. $L^2_\sigma$ is defined as the completion of $C^{\infty}_{0,\sigma}(\mathscr{O};\R^N)$ in $L^2(\mathscr{O};\R^N)$, whilst we introduce $W^{1,2}_\sigma$ as the intersection of $W^{1,2}_0(\mathscr{O};\R^N)$ with $L^2_\sigma$ and $W^{2,2}_{\sigma}$ as the intersection of $W^{2,2}(\mathscr{O};\R^N)$ with $W^{1,2}_\sigma$. 
\end{definition}
Of course the dependency on $N$ is implicit in the given definitions and it will be made clear if $N$ is required to be specifically $2$ or $3$. We next give the probabilistic set up. Let $(\Omega,\mathcal{F},(\mathcal{F}_t), \mathbb{P})$ be a fixed filtered probability space satisfying the usual conditions of completeness and right continuity. We take $\mathcal{W}$ to be a cylindrical Brownian motion over some Hilbert Space $\mathfrak{U}$ with orthonormal basis $(e_i)$. Recall (e.g. [\cite{lototsky2017stochastic}], Definition 3.2.36) that $\mathcal{W}$ admits the representation $\mathcal{W}_t = \sum_{i=1}^\infty e_iW^i_t$ as a limit in $L^2(\Omega;\mathfrak{U}')$ whereby the $(W^i)$ are a collection of i.i.d. standard real valued Brownian Motions and $\mathfrak{U}'$ is an enlargement of the Hilbert Space $\mathfrak{U}$ such that the embedding $J: \mathfrak{U} \rightarrow \mathfrak{U}'$ is Hilbert-Schmidt and $\mathcal{W}$ is a $JJ^*-$cylindrical Brownian Motion over $\mathfrak{U}'$. Given a process $F:[0,T] \times \Omega \rightarrow \mathscr{L}^2(\mathfrak{U};\mathscr{H})$ progressively measurable and such that $F \in L^2\left(\Omega \times [0,T];\mathscr{L}^2(\mathfrak{U};\mathscr{H})\right)$, for any $0 \leq t \leq T$ we define the stochastic integral $$\int_0^tF_sd\mathcal{W}_s:=\sum_{i=1}^\infty \int_0^tF_s(e_i)dW^i_s,$$ where the infinite sum is taken in $L^2(\Omega;\mathscr{H})$. We can extend this notion to processes $F$ which are such that $F(\omega) \in L^2\left( [0,T];\mathscr{L}^2(\mathfrak{U};\mathscr{H})\right)$ for $\mathbb{P}-a.e.$ $\omega$ via the traditional localisation procedure. In this case the stochastic integral is a local martingale in $\mathscr{H}$. \footnote{A complete, direct construction of this integral, a treatment of its properties and the fundamentals of stochastic calculus in infinite dimensions can be found in [\cite{prevot2007concise}] Section 2.} We shall make frequent use of the Burkholder-Davis-Gundy Inequality ([\cite{da2014stochastic}] Theorem 4.36), passage of a bounded linear operator through the stochastic integral ([\cite{prevot2007concise}] Lemma 2.4.1) and the It\^{o} Formula (in particular, Proposition \ref{rockner prop}). 

\subsection{Functional Framework} \label{functional framework subsection}

We now recap the classical functional framework for the study of the deterministic Navier-Stokes Equation. A more detailed summary with explicit proofs can be found in [\cite{goodair2022navier}] Subsection 1.3. We now formally define the operator $\mathcal{L}$ that appears in (\ref{number2equation}, \ref{number3equation}), as well as the divergence-free and no-slip boundary conditions. Firstly though we briefly comment on the pressure term $\nabla \rho$, which will not play any role in our analysis. $\rho$ does not come with an evolution equation and is simply chosen to ensure the incompressibility (divergence-free) condition; moreover we will eliminate this term via a suitable projection and treat the projected equation, with the understanding that we may append a pressure to it to recover the original equation. This procedure is well discussed in [\cite{robinson2016three}] Section 5 and [\cite{bertozzi2002vorticity}], and an explicit form for the pressure for the SALT Euler Equation is given in [\cite{street2021semi}] Subsection 3.3.\\ 

The nonlinear operator $\mathcal{L}$ is defined for sufficiently regular functions $f,g:\mathscr{O} \rightarrow \R^N$ by $\mathcal{L}_fg:= \sum_{j=1}^Nf^j\partial_jg.$ Here and throughout the text we make no notational distinction between differential operators acting on a vector valued function or on a scalar valued one; that is, we understand $\partial_jg$ by its component mappings $(\partial_lg)^l := \partial_jg^l$. We now give some clarification as to 'sufficiently regular', by stating basic properties of this mapping. For any $m \in \N$, the mapping $\mathcal{L}: W^{m+1,2}(\mathscr{O};\R^N) \rightarrow W^{m,2}(\mathscr{O};\R^N)$ defined by $f \mapsto \mathcal{L}_ff$ is continuous. Additionally there exists a constant $c$ such that for any $f,g \in W^{k,2}(\mathscr{O};\R^N)$ for $k \in \N$ as appropriate, we have the bounds:
    \begin{align} \label{first begin align}
        \norm{\mathcal{L}_{f}{g}} + \norm{\mathcal{L}_{g}{f}}&\leq c\norm{g}_{W^{1,2}}\norm{f}_{W^{2,2}};\\ \label{second begin align}
        \norm{\mathcal{L}_{g}{f}}_{W^{1,2}} &\leq c\norm{g}_{W^{1,2}}\norm{f}_{W^{3,2}};\\ \label{third begin align}
        \norm{\mathcal{L}_{g}{f}}_{W^{1,2}} &\leq c\norm{g}_{W^{2,2}}\norm{f}_{W^{2,2}},
        \end{align}
    see [\cite{goodair2022navier}] Lemma 1.3. For the divergence-free condition we mean a function $f$ such that the property $$\textnormal{div}f := \sum_{j=1}^N \partial_jf^j = 0$$ holds. We require this property and the boundary condition to hold for our solution $u$ at all times, though there is some ambiguity as to how we understand these conditions for a solution $u$ which need not be defined pointwise everywhere on $\bar{\mathscr{O}}$. We shall understand these conditions in their traditional weak sense, that is for weak derivatives $\partial_j$ so $\sum_{j=1}^N \partial_jf^j = 0$ holds as an identity in $L^2(\mathscr{O};\R)$ whilst the boundary condition $u=0$ is understood as each component mapping $u^j$ having zero trace (recall e.g. [\cite{evans2010partial}] that $f^j \in W^{1,2}(\mathscr{O};\R) \cap C(\bar{\mathscr{O}};\R)$ has zero trace if and only if $f^j(x) = 0$ for all $x \in \partial \mathscr{O}$). We impose these conditions by incorporating them into the function spaces where our solution takes value.

\begin{remark}  \label{first labelled remark}
    $W^{1,2}_{\sigma}$ is precisely the subspace of $W^{1,2}_0(\mathscr{O};\R^N)$ consisting of divergence-free functions. Moreover, $W^{1,2}_{\sigma}$ is the completion of $C^{\infty}_{0,\sigma}(\mathscr{O};\R^N)$ in $W^{1,2}(\mathscr{O};\R^N)$. The general space $W^{1,2}_{\sigma}$ thus incorporates the divergence-free and zero-trace condition (see [\cite{goodair2022navier}] Lemma 1.7). 
\end{remark}
We introduce the Leray Projector $\mathcal{P}$ as the orthogonal projection in $L^2(\mathscr{O};\R^N)$ onto $L^2_{\sigma}$. It is well known (see e.g. [\cite{temam2001navier}] Remark 1.6.) that for any $m \in \N$, $\mathcal{P}$ is continuous as a mapping $\mathcal{P}: W^{m,2}(\mathscr{O};\R^N) \rightarrow W^{m,2}(\mathscr{O};\R^N)$. In fact, the complement space of $L^2_{\sigma}$ can be characterised (this is the so-called Helmholtz-Weyl decomposition), and we direct the reader to [\cite{temam2001navier}] Theorems 1.4, 1.5 and [\cite{robinson2016three}] Theorem 2.6 for such an explicit characterisation. Through $\mathcal{P}$ we define the Stokes Operator $A: W^{2,2}(\mathscr{O};\R^N) \rightarrow L^2_{\sigma}$ by $A:= -\mathcal{P}\Delta$. We understand the Laplacian as an operator on vector valued functions through the component mappings, $(\Delta f)^l := \Delta f^l$. From the continuity of $\mathcal{P}$ we have immediately that for $m \in \{0\} \cup \N$, $A: W^{m+2,2}(\mathscr{O};\R^N) \rightarrow W^{m,2}(\mathscr{O};\R^N)$ is continuous. Moreover (see [\cite{robinson2016three}] Theorem 2.24) there exists a collection of functions $(a_k)$, $a_k \in W^{1,2}_{\sigma} \cap C^{\infty}(\overbar{\mathscr{O}};\R^N)$ such that the $(a_k)$ are eigenfunctions of $A$, are an orthonormal basis in $L^2_{\sigma}$ and an orthogonal basis in $W^{1,2}_{\sigma}$. The eigenvalues $(\lambda_k)$ are strictly positive and approach infinity as $k \rightarrow \infty$. Therefore every $f \in L^2_{\sigma}$ admits the representation \begin{equation}\label{L2sigrep}
    f = \sum_{k=1}^\infty f_ka_k
\end{equation}
where $f_k = \inner{f}{a_k}$, as a limit in $L^2(\mathscr{O};\R^N)$.

\begin{definition}
For $m \in \N$ we introduce the spaces $D(A^{m/2})$ as the subspaces of $L^2_{\sigma}$ consisting of functions $f$ such that $$\sum_{k=1}^\infty \lambda_k^m f_k^2 < \infty $$ for $f_k$ as in (\ref{L2sigrep}). Then $A^{m/2}:D(A^{m/2}) \rightarrow L^2_{\sigma}$ is defined by $$A^{m/2}:f \mapsto \sum_{k=1}^\infty \lambda_k^{m/2}f_ka_k.$$
\end{definition}
We present some fundamental properties regarding these spaces, which are justified in [\cite{constantin1988navier}] Proposition 4.12, as well as [\cite{robinson2016three}] Exercises 2.12, 2.13 and the discussion in Subsection 2.3.

\begin{enumerate}
    \item $D(A^{m/2}) \subset W^{m,2}(\mathscr{O};\R^N) \cap W^{1,2}_{\sigma}$ and the bilinear form $$\inner{f}{g}_m:= \inner{A^{m/2}f}{A^{m/2}g}$$ is an inner product on $D(A^{m/2})$;\\
    \item For $m$ even the induced norm is equivalent to the $W^{m,2}(\mathscr{O};\R^N)$ norm, and for $m$ odd there is a constant $c$ such that $$\norm{\cdot}_{W^{m,2}} \leq  c\norm{\cdot}_m;$$
    \item $D(A) = W^{2,2}_{\sigma}$ and $D(A^{1/2}) = W^{1,2}_{\sigma}$ with the additional property that $\norm{\cdot}_1$ is equivalent to $\norm{\cdot}_{W^{1,2}}$ on this space.
\end{enumerate}
It can be directly shown that for any $p,q \in \N$ with $p \leq q$, $p + q = 2m$ and $f \in D(A^{m/2})$, $g \in D(A^{q/2})$ we have that \begin{equation}\label{prop for moving stokes'}\inner{f}{g}_m = \inner{A^{p/2}f}{A^{q/2}g}.\end{equation}
From here we can also see that the collection of functions $(a_k)$ form an orthogonal basis of $W^{1,2}_{\sigma}$ equipped with the $\inner{\cdot}{\cdot}_1$ inner product. In addition to using these spaces defined by powers of the Stokes Operator, we also use the basis $(a_k)$ to consider finite dimensional approximations of these spaces.

\begin{definition}
We define $\mathcal{P}_n$ as the orthogonal projection onto $\textnormal{span}\{a_1, \dots, a_n\}$ in $L^2(\mathscr{O};\R^N)$. That is $\mathcal{P}_n$ is given by $$\mathcal{P}_n:f \mapsto \sum_{k=1}^n\inner{f}{a_k}a_k$$ for $f \in L^2(\mathscr{O};\R^N)$.
\end{definition}

From [\cite{robinson2016three}] Lemma 4.1, we have that the restriction of $\mathcal{P}_n$ to $D(A^{m/2})$ is self-adjoint for the $\inner{\cdot}{\cdot}_m$ inner product, and there exists a constant $c$ independent of $n$ such that for all $f\in D(A^{m/2})$, \begin{equation}\label{P_nbounds}
        \norm{\mathcal{P}_nf}_{W^{m,2}} \leq c\norm{f}_{W^{m,2}}.
    \end{equation}
Similar ideas justify that there exists a constant $c$ such that for all $f \in W^{1,2}_{\sigma}$, $g \in W^{2,2}_{\sigma}$ we have that 
    \begin{align} \label{one over n bound one}
        \norm{(I-\mathcal{P}_n)f}^2 &\leq \frac{1}{\lambda_n}\norm{f}_1^2\\
        \norm{(I-\mathcal{P}_n)g}_1^2 &\leq \frac{1}{\lambda_n}\norm{g}_2^2 \label{one over n bound two}
    \end{align}
    where $I$ represents the identity operator in the relevant spaces. To conclude this subsection we present identities related to the nonlinear term, which will be used in our analysis. For every $\phi \in W^{1,2}_{\sigma}$ and $f, g \in W^{1,2}(\mathscr{O};\R^N)$, we have that \begin{equation}\label{wloglhs}\inner{\mathcal{L}_{\phi}f}{g}= -\inner{f}{\mathcal{L}_{\phi}g}.\end{equation} 
and moreover \begin{equation} \label{cancellationproperty'} \inner{\mathcal{L}_{\phi}f}{f}= 0.\end{equation} In fact from inspecting the proof of, for example, Lemma 1.23 of [\cite{goodair2022navier}], we see that (\ref{wloglhs}) still holds for $\phi \in L^2_{\sigma} \cap W^{1,2}(\mathscr{O};\R^N)$ if we assume that $f,g \in W^{1 + \frac{N}{2}}(\mathscr{O};\R^N)$ through using an approximation by compactly supported functions only in $L^2(\mathscr{O};\R^N)$ and then the Sobolev Embedding of $f,g$ into $W^{1,\infty}(\mathscr{O};\R^N)$. This extension will be needed in the treatment of $J_1$ in Subsection \ref{sub big based}.

\subsection{Assumptions on the Noise} \label{subsection assumptions}

With the framework established we now precisely introduce the stochastic Navier-Stokes equation

\begin{equation} \label{projected Ito}
    u_t = u_0 - \int_0^t\mathcal{P}\mathcal{L}_{u_s}u_s\ ds - \nu\int_0^t A u_s\, ds + \frac{\nu}{2}\int_0^t\sum_{i=1}^\infty \mathcal{P}\mathcal{Q}_i^2u_s ds - \nu^{\frac{1}{2}}\int_0^t \mathcal{P}\mathcal{G}u_s d\mathcal{W}_s 
\end{equation}
where $\mathcal{Q}_i$ is either $\mathcal{P}\mathcal{G}_i$ or $0$, satisfying assumptions to be stated in this subsection. The case $\mathcal{Q}_i = 0$ leaves us with the projected form of (\ref{number2equation}) whilst $\mathcal{Q}_i=\mathcal{P}\mathcal{G}_i$ corresponds to (\ref{number3equation}) via taking the Leray Projection and then converting to It\^{o} Form. This conversion is rigorously justified in [\cite{goodair2022stochastic}] Subsection 2.3. In the case where $\mathcal{P}\mathcal{G}_i^2 = (\mathcal{P}\mathcal{G}_i )^2$ then we can instead take $\mathcal{Q}_i = \mathcal{G}_i$ as the resulting equation (\ref{projected Ito}) is the same; this is the case for SALT noise, discussed in the next subsection. The key definitions and results regarding the existence and uniqueness of solutions is given in Subsection \ref{sub def}. As we are interested in the inviscid limit, we assume here and throughout that $0 < \nu < 1$.\\

We impose the existence of some $p,q,r \in \R$ and constants $(c_i)$ such that for all $f, g \in L^2_{\sigma} \cap W^{1,2}(\mathscr{O};\R^N)$, $\phi,\psi \in W^{2,2}_{\sigma}(\mathscr{O};\R^N)$, defining $K(f,g):= 1 + \norm{f}^p + \norm{g}^q + \norm{f}_{W^{1,2}}^2 + \norm{g}_{W^{1,2}}^2$,
\begin{align}
   \label{assumpty 1}  \norm{\mathcal{G}_if}^2 &\leq c_i\left(1 + \norm{f}_{W^{1,2}}^2\right)\\
   \label{assumpty 4}  \norm{\mathcal{G}_if - \mathcal{G}_ig}^2 &\leq c_i\left[1 + \norm{f}_{W^{1,2}}^p + \norm{g}_{W^{1,2}}^q\right]\norm{f-g}_{W^{1,2}}^2\\
   \label{assumpty 2} \norm{\mathcal{Q}_i\phi}_{W^{1,2}}^2 &\leq c_i\norm{\phi}_2^2\\
    \label{assumpty 7} \inner{\mathcal{Q}_i^2\phi }{\phi} + \norm{\mathcal{G}_i\phi}^2 &\leq c_i\left(1 + \norm{\phi}^2\right) + k_i\norm{\phi}_1^2\\
    \label{assumpty8} \inner{\mathcal{G}_if}{f}^2 &\leq c_i\left(1 + \norm{f}^4\right)\\
    \label{assumpty9} \inner{\mathcal{G}_if}{g}^2 &\leq c_i\left[1 + \norm{f}^2 + \norm{g}^p\right]\norm{g}_{W^{1,2}}^2\\
    \label{assumpty10} \inner{\mathcal{G}_if - \mathcal{G}_ig}{\phi}^2 &\leq  c_i\left[1  + \norm{\phi}_2^p\right]\norm{f-g}^2\\
    \label{assumpty 6}
\inner{\mathcal{G}_if - \mathcal{G}_ig}{f-g}^2 &\leq c_iK(f,g)\norm{f -g}^4
\end{align}
where $\sum_{i=1}^\infty c_i < \infty$ and $\sum_{i=1}^\infty k_i \leq 1$\footnote{Actually, we only need that $\sum_{i=1}^\infty k_i < 2$. Our choice just avoids additional technical details.}. In fact we require that these bounds hold on any measurable subset of $\mathscr{O}$ with smooth boundary. For each $i \in \N$, $\mathcal{Q}_i$ must be linear and possess a densely defined adjoint $\mathcal{Q}_i^*$ in $L^2(\mathscr{O};\R^N)$ with domain of definition $W^{1,2}(\mathscr{O};\R^N)$ where for every $\varepsilon > 0$ there exists a constant $c(\varepsilon)$ such that, if $f,g$ also belong to $L^2_{\sigma}$,
\begin{align}
   \label{assumpty 3} \norm{\mathcal{Q}_i^*f}^2 &\leq c_i\norm{f}_{W^{1,2}}^2\\ 
   \inner{\mathcal{Q}_i(f - g)}{\mathcal{Q}_i^*(f-g)} + \norm{\mathcal{G}_if - \mathcal{G}_ig}^2 &\leq c_iK(f,g)\norm{f -g}^2 + k_i\norm{f -g}^2_{W^{1,2}}. \label{assumpty 5}
\end{align}
Moreover we assume that $\mathcal{Q}_i^*$ has structure $\mathcal{Q}_i^* = \mathcal{A}_i + \hat{\mathcal{A}_i}$ where if $f\in W^{1,2}(\mathscr{O};\R^N)$ has support in a set $\mathcal{U} \subset \bar{\mathscr{O}}$, then $\mathcal{A}_if$ again has support in $\mathcal{U}$ and $\norm{\hat{\mathcal{A}_i}f}^2 \leq c_i\norm{f}^2$.

\subsection{Examples} \label{section applications}

We consider examples of noise for which the assumptions imposed in Subsection \ref{subsection assumptions} are satisfied. It is immediate that our setting covers the additive noise used in the works of [\cite{bessaih2013inviscid}], [\cite{glatt2015inviscid}] and [\cite{luongo2021inviscid}], whilst also enabling linear multiplicative noise as seen in [\cite{gao2019existence}] and Nemytskii operators as present in [\cite{glatt2009strong}] and [\cite{taniguchi2014global}]. Using the property (\ref{wloglhs}) it is largely straightforwards to see that the usual transport noise $\mathcal{G}_i = \mathcal{P}\mathcal{L}_{\xi_i}$, $\mathcal{Q}_i = \mathcal{G}_i$ for $\xi_i \in W^{1,2}_{\sigma} \cap W^{1,\infty}(\mathscr{O};\R^N)$ with $\sum_{i=1}^\infty \norm{\xi_i}_{W^{1,\infty}}^2 < \infty$ also satisfies our assumptions. We note that imposing $\mathcal{P}$ into $\mathcal{G}_i$ makes no difference to the equation (\ref{projected Ito}) however it will be necessary to verify the assumptions. We only draw attention to the condition (\ref{assumpty 7}), and by extension (\ref{assumpty 5}), as there are no subtleties in the other inequalities. The argument is that $$\inner{(\mathcal{P}\mathcal{L}_{\xi_i})^2\phi}{\phi} = \inner{\mathcal{L}_{\xi_i}\mathcal{P}\mathcal{L}_{\xi_i}\phi}{\phi} = -\inner{\mathcal{P}\mathcal{L}_{\xi_i}\phi}{\mathcal{L}_{\xi_i}\phi} =  -\inner{\mathcal{P}\mathcal{L}_{\xi_i}\phi}{\mathcal{P}\mathcal{L}_{\xi_i}\phi}$$ so $$\inner{(\mathcal{P}\mathcal{L}_{\xi_i})^2\phi}{\phi} + \norm{\mathcal{P}\mathcal{L}_{\xi_i}\phi}^2 =  -\inner{\mathcal{P}\mathcal{L}_{\xi_i}\phi}{\mathcal{P}\mathcal{L}_{\xi_i}\phi} + \inner{\mathcal{P}\mathcal{L}_{\xi_i}\phi}{\mathcal{P}\mathcal{L}_{\xi_i}\phi} = 0$$
hence the assumption certainly holds. This noise is at the core of [\cite{mikulevicius2005global}] and many developments in stochastic fluid dynamics, as discussed in the introduction. We also note that the assumptions hold for $\mathcal{G}_i = \mathcal{L}_{\xi_i}$ and $\mathcal{Q}_i = 0$ if $\sum_{i=1}^\infty \norm{\xi_i}_{W^{1,\infty}}^2 \leq 1$, which is an It\^{o} transport noise with sufficiently small gradient dependency. We now explicitly address the application to SALT noise.\\

The SALT Navier-Stokes Equation as first introduced in [\cite{holm2015variational}] is given by (\ref{number3equation}) for the operator $\mathcal{G}:=B$ where $$B_i:f \mapsto \mathcal{L}_{\xi_i}f + \mathcal{T}_{\xi_i}f, \qquad  \mathcal{T}_{g}f := \sum_{j=1}^N f^j\nabla g^j$$ for $\xi_i \in W^{1,2}_{\sigma} \cap W^{2,\infty}(\mathscr{O};\R^N)$ such that $\sum_{i=1}^\infty \norm{\xi_i}_{W^{2,\infty}}^2 < \infty$. The vector fields $(\xi_i)$ physically represent spatial correlations; they can be determined at coarse-grain resolutions from finely resolved numerical simulations, and mathematically are derived as eigenvectors of a velocity-velocity correlation matrix (see [\cite{cotter2020data}, \cite{cotter2018modelling}, \cite{cotter2019numerically}]). The equivalence between the Stratonovich form and It\^{o} Form (\ref{projected Ito}) is rigorously understood in [\cite{goodair2022navier}] Subsection 2.1. Verification of the assumptions of Subsection \ref{subsection assumptions} is almost immediate from the analysis of this operator in [\cite{goodair2022navier}] Subsection 1.4. We first note that the property $\mathcal{P}B_i = \mathcal{P}B_i\mathcal{P}$ was proven in [\cite{goodair2022navier}] Lemma 1.28 so $(\mathcal{P}B_i)^2 = \mathcal{P}B_i^2$ and we take $\mathcal{Q}_i = B_i$ without Leray Projection. Indeed (\ref{assumpty 1}), (\ref{assumpty 4}) and (\ref{assumpty 2}) are given by [\cite{goodair2022navier}] Corollary 1.26.1 and the linearity of $B_i$, whilst (\ref{assumpty 7}), (\ref{assumpty8}) and (\ref{assumpty 6}) are contained in Proposition 1.27. (\ref{assumpty9}), (\ref{assumpty10}) and (\ref{assumpty 3}) all follow from the adjoint property of Corollary 1.26.1. The final numbered assumption (\ref{assumpty 5}) is contained in the proof of Proposition 1.27, and is near identical to (\ref{assumpty 7}) given that $B_i$ is linear. It then only remains to address if the structure $B_i^* = \mathcal{A}_i + \hat{\mathcal{A}_i}$ holds, which is clear as $\mathcal{L}_{\xi_i}^* = - \mathcal{L}_{\xi_i}$ preserves the support, and $\mathcal{T}_{\xi_i}^*$ is bounded on $L^2(\mathscr{O};\R^N)$.

\subsection{Notions of Solution and Well-Posedness Results} \label{sub def}

We fix an arbitrary $T>0$ and give two definitions for weak solutions of the equation (\ref{projected Ito}).

\begin{definition} \label{definitionofspatiallyweak}
Let $u_0: \Omega \rightarrow L^2_{\sigma}$ be $\mathcal{F}_0-$measurable. A process $u$ which is progressively measurable in $W^{1,2}_{\sigma}$ and such that for $\mathbb{P}-a.e.$ $\omega$, $u_{\cdot}(\omega) \in L^{\infty}\left([0,T];L^2_{\sigma}\right)\cap C_w\left([0,T];L^2_{\sigma}\right) \cap L^2\left([0,T];W^{1,2}_{\sigma}\right)$, is said to be a spatially weak solution of the equation (\ref{projected Ito}) if the identity
\begin{align} \nonumber
     \inner{u_t}{\phi} = \inner{u_0}{\phi} - \int_0^{t}\inner{\mathcal{L}_{u_s}u_s}{\phi}ds &- \nu\int_0^{t} \inner{u_s}{\phi}_1 ds\\ &+ \frac{\nu}{2}\int_0^{t}\sum_{i=1}^\infty \inner{\mathcal{Q}_iu_s}{\mathcal{Q}_i^*\phi} ds - \nu^{\frac{1}{2}}\int_0^{t} \inner{\mathcal{G}u_s}{\phi} d\mathcal{W}_s\label{identityindefinitionofspatiallyweak}
\end{align}
holds for every $\phi \in W^{1,2}_{\sigma}$, $\mathbb{P}-a.s.$ in $\R$ for all $t\in[0,T]$.
\end{definition}

We briefly note from two applications of H\"{o}lder's Inequality and (recalling that $N= 2,3$) the Sobolev Embedding $W^{1,2}(\mathscr{O};\R^N) \xhookrightarrow{}L^6(\mathscr{O};\R^N)$ that
\begin{align} \label{why the star}
    \abs{\inner{\mathcal{L}_{u_s}u_s}{\phi}} \leq \norm{\mathcal{L}_{u_s}u_s}_{L^{6/5}}\norm{\phi}_{L^{6}} \leq \sum_{k=1}^Nc\norm{u_s}_{L^3}\norm{\partial_k u_s}\norm{\phi}_{L^{6}} \leq c\norm{u_s}_1^2\norm{\phi}_1
\end{align}
so this first integral is indeed well defined. The remaining integrals are much clearer, noting (\ref{assumpty 1}).

\begin{definition} \label{definitionofspacetimeweak}
Let $u_0: \Omega \rightarrow L^2_{\sigma}$ be $\mathcal{F}_0-$measurable. A process $u$ which is progressively measurable in $W^{1,2}_{\sigma}$ and such that for $\mathbb{P}-a.e.$ $\omega$, $u_{\cdot}(\omega) \in L^{\infty}\left([0,T];L^2_{\sigma}\right)\cap C_w\left([0,T];L^2_{\sigma}\right) \cap L^2\left([0,T];W^{1,2}_{\sigma}\right)$, is said to be a space-time weak solution of the equation (\ref{projected Ito}) if the identity
\begin{align} \nonumber
     \inner{u_t}{\phi_t} = \inner{u_0}{\phi_0}  + \int_0^t\inner{u_s}{\partial_s\phi_s}ds &- \int_0^{t}\inner{\mathcal{L}_{u_s}u_s}{\phi_s}ds - \nu\int_0^{t} \inner{u_s}{\phi_s}_1 ds\\ &+ \frac{\nu}{2}\int_0^{t}\sum_{i=1}^\infty \inner{\mathcal{Q}_iu_s}{\mathcal{Q}_i^*\phi_s} ds - \nu^{\frac{1}{2}}\int_0^{t} \inner{\mathcal{G}u_s}{\phi_s } d\mathcal{W}_s\label{identityindefinitionofspacetimeweak}
\end{align}
holds for every $\phi \in C^1\left([0,t] \times \overbar{\mathscr{O}}; \R^N\right)$ such that $\phi_s \in W^{1,2}_{\sigma}$ for every $s \in [0,t]$, $\mathbb{P}-a.s.$ in $\R$ for all $t \in [0,T]$.
\end{definition}

The difference in the two definitions comes from whether or not there is time-dependency in the test function. We will see in Section \ref{section zero viscous} that the time-dependency is necessary for us to characterise the zero viscosity limit, as we shall use a corrected solution of the nonstationary Euler Equation as a test function in this weak formulation. On the other hand this formulation is impractical to demonstrate and to work with on the whole in the stochastic setting, as we do not have differentiability in time for our approximate solutions. Therefore both representations serve a purpose, and we must show their equivalence. This is stated in the following proposition, whose proof we leave for Subsection \ref{sub equivalence}.

\begin{proposition} \label{prop for equivalence}
    A process $u$ is a spatially weak solution of the equation (\ref{projected Ito}) if and only if it is a space-time weak solution.
\end{proposition}

We will refer to such a solution as simply a weak solution of the equation (\ref{projected Ito}). We now define notions of (pathwise) uniqueness and probabilistically weak solutions.

\begin{definition}
    A weak solution of the equation (\ref{projected Ito}) is said to be unique if for any other such solution $w$, $$ \mathbbm{P}\left(\left\{\omega \in \Omega: u_t(\omega) = w_t(\omega) \quad \forall t \geq 0\right\}\right) = 1.$$
\end{definition}

\begin{definition} \label{definitionofspacetimeweakmartingale}
Let $u_0: \Omega \rightarrow L^2_{\sigma}$ be $\mathcal{F}_0-$measurable. If there exists a filtered probability space $\left(\tilde{\Omega},\tilde{\mathcal{F}},(\tilde{\mathcal{F}}_t), \tilde{\mathbbm{P}}\right)$, a cyclindrical Brownian Motion $\tilde{\mathcal{W}}$ over $\mathfrak{U}$ with respect to $\left(\tilde{\Omega},\tilde{\mathcal{F}},(\tilde{\mathcal{F}}_t), \tilde{\mathbbm{P}}\right)$, an $\mathcal{F}_0-$measurable $\tilde{u}_0: \tilde{\Omega} \rightarrow L^2_{\sigma}$ with the same law as $u_0$, and a progressively measurable process $\tilde{u}$ in $W^{1,2}_{\sigma}$ such that for $\tilde{\mathbb{P}}-a.e.$ $\tilde{\omega}$, $\tilde{u}_{\cdot}(\omega) \in L^{\infty}\left([0,T];L^2_{\sigma}\right)\cap C_w\left([0,T];L^2_{\sigma}\right) \cap L^2\left([0,T];W^{1,2}_{\sigma}\right)$ and
\begin{align} \nonumber
     \inner{\tilde{u}_t}{\phi} = \inner{\tilde{u}_0}{\phi} - \int_0^{t}\inner{\mathcal{L}_{\tilde{u}_s}\tilde{u}_s}{\phi}ds &- \nu\int_0^{t} \inner{\tilde{u}_s}{\phi}_1 ds\\ &+ \frac{\nu}{2}\int_0^{t}\sum_{i=1}^\infty \inner{\mathcal{Q}_i\tilde{u}_s}{\mathcal{Q}_i^*\phi} ds - \nu^{\frac{1}{2}}\int_0^{t} \inner{\mathcal{G}\tilde{u}_s}{\phi} d\tilde{\mathcal{W}}_s\label{newid1}
\end{align}
holds for every $\phi \in W^{1,2}_{\sigma}$ $\tilde{\mathbb{P}}-a.s.$ in $\R$ for all $t \in [0,T]$, then $\tilde{u}$ is said to be a martingale weak solution of the equation (\ref{projected Ito}).
\end{definition}

This positions us to state the following existence and uniqueness results, which are proven in Section \ref{section weak solutions}.

\begin{theorem}  \label{existence of weak}
    For any given  $\mathcal{F}_0-$measurable $u_0 \in L^\infty\left( \Omega; L^2_{\sigma}\right)$, there exists a martingale weak solution of the equation (\ref{projected Ito}).
\end{theorem}

\begin{theorem} \label{theorem 2D}
    If $N=2$ then for any given $\mathcal{F}_0-$measurable $u_0: \Omega \rightarrow L^2_{\sigma}$, there exists a unique weak solution $u$ of the equation (\ref{projected Ito}) with the property that for $\mathbbm{P}-a.e.$ $\omega$, $u_{\cdot}(\omega) \in C\left([0,T];L^2_{\sigma}\right)$.
\end{theorem}

\section{The Zero Viscosity Limit} \label{section zero viscous}

In this section we shall consider the zero viscosity limit of the equation (\ref{projected Ito}). As we are interested in the inviscid limit, we assume that $0 < \nu < 1$. We pose this for a deterministic $u_0 \in W^{m,2}\left(\mathscr{O};\R^N\right) \cap L^2_{\sigma}\left(\mathscr{O};\R^N\right)$ for $m > 1 + \frac{N}{2}$, and wish to characterise the convergence of selected martingale weak solutions of (\ref{projected Ito}) to the solution of the Euler equation. This selection is made precise in Subsection \ref{subsection selection of martingale weak solutions} and the solution of the Euler equation is defined now. We no longer consider an arbitrary $T>0$ but in this section fix $T$ as specified in the following: recall a result proved by many authors, specifically referring to [\cite{bourguignon1974remarks}] Theorem 1 and the near immediate identity (3.4) in [\cite{kato1984remarks}], which is that there exists a $T>0$ and a unique $\bar{u} \in C\left([0,T]; W^{m,2}(\mathscr{O};\R^N) \cap L^2_{\sigma}\left(\mathscr{O};\R^N\right)\right) \cap C^1\left([0,T] \times \bar{\mathscr{O}};\R^N\right)$ such that the identity \begin{equation} \label{euler identity} \partial_t \bar{u} = -\mathcal{P}\mathcal{L}_{\bar{u}}\bar{u}\end{equation} holds on $\mathscr{O} \times [0,T]$, and $\bar{u}_t|_{t=0} = u_0$ holds on $\mathscr{O}$. Moreover for every $t \in [0,T]$, \begin{equation}\label{energy identity}\norm{\bar{u}_t}^2 = \norm{u_0}^2.\end{equation} 
In Subsection \ref{subsection selection of martingale weak solutions} we make explicit the martingale weak solutions used for the Kato Criterion. This criterion is then stated and addressed in Subsection \ref{sub based}, with the key implication proven in Subsection \ref{sub big based}. In Subsection \ref{subsection optimal noise scaling} we consider a new parameter in the noise which approaches zero at a (possibly) different rate, determining the implications of this for our criterion and clarifying a sense in which the scaling of $\nu^\frac{1}{2}$ is optimal.

\subsection{Selection of Martingale Weak Solutions} \label{subsection selection of martingale weak solutions}

We introduce the notation $o_{\nu}$ to represent any constant dependent on $\nu$ such that $\lim_{\nu \rightarrow 0}o_{\nu} = 0$, and build upon the existence result Theorem \ref{existence of weak} in the case of equation (\ref{projected Ito}) for some more precise energy estimates.

\begin{proposition} \label{ex theorem}
    There exists a martingale weak solution $\tilde{u}$ of the equation (\ref{projected Ito}) which satisfies
    \begin{align} \label{hello20}
        \tilde{\mathbbm{E}}\left(\sup_{r\in[0,T]}\norm{\tilde{u}_{r}}^2 \right) \leq \left(1 + o_{\nu}\right) \norm{u_0}^2 + o_{\nu}        
    \end{align}
    and for every $t \in [0,T]$, \begin{equation} \label{hello21}
        \tilde{\mathbbm{E}}\left[\norm{\tilde{u}_{t}}^2 + \nu\int_0^{t} \norm{\tilde{u}_s}^2_1 ds\right] \leq \left(1 + o_{\nu}\right) \norm{u_0}^2 + o_{\nu}.
         \end{equation}
\end{proposition}

\begin{proof}
    See Subsection \ref{sub energy}.
\end{proof}

\begin{remark}
    Estimates (\ref{hello20}) and (\ref{hello21}) were stated in this way to make the dependency on the initial condition explicit, though we shall henceforth keep the initial condition constant in $\nu$ and as such we may refer to the more direct inequalities
    \begin{align}
        \label{hello1}
\tilde{\mathbbm{E}}\left(\sup_{r\in[0,T]}\norm{\tilde{u}_{r}}^2 \right) &\leq \norm{u_0}^2 + o_{\nu}\\
        \label{hello3}
        \tilde{\mathbbm{E}}\left[\norm{\tilde{u}_{t}}^2 + \nu\int_0^{t} \norm{\tilde{u}_s}^2_1 ds\right] &\leq \norm{u_0}^2 + o_{\nu}.
    \end{align}
\end{remark}

We now clarify how the martingale weak solutions of (\ref{projected Ito}) are selected and some notation for the following subsection. Formally we must work with an arbitrary sequence of viscosities $(\nu_k)$ such that $\nu_k \rightarrow 0$ as $k \rightarrow \infty$. For each such $k$ we then choose a martingale weak solution $\tilde{u}^k$ as specified in Proposition \ref{ex theorem}, which we recall is defined with respect to a filtered probability space $\left(\tilde{\Omega}^k,\tilde{\mathcal{F}}^k,(\tilde{\mathcal{F}}_t^k), \tilde{\mathbbm{P}}^k\right)$, a cyclindrical Brownian Motion $\tilde{\mathcal{W}}^k$ over $\mathfrak{U}$ with respect to $\left(\tilde{\Omega}^k,\tilde{\mathcal{F}}^k,(\tilde{\mathcal{F}}_t^k), \tilde{\mathbbm{P}}^k\right)$, and an $\mathcal{F}_0-$measurable $\tilde{u}^k_0: \tilde{\Omega} \rightarrow L^2_{\sigma}$ with the same law as $u_0$. Immediately we note that as $u_0$ is deterministic then $\tilde{u}^k_0$ of the same law must simply be $u_0$ itself. It is less obvious how we can consider the limiting properties of this sequence of solutions where each $\tilde{u}^k$ is defined on a different probability space. We rectify this with the following:
\begin{itemize}
    \item The standard infinite dimensional product space $$\tilde{\Omega}:= \bigtimes_{k=0}^\infty \tilde{\Omega}^k, \quad \tilde{\mathcal{F}}:= \bigotimes_{k=0}^\infty \tilde{\mathcal{F}}^k, \quad \tilde{\mathcal{F}}_t:= \bigotimes_{k=0}^\infty \tilde{\mathcal{F}}^k_t, \quad \tilde{\mathbb{P}}:= \bigtimes_{k=0}^\infty \tilde{\mathbb{P}}^k$$
    such that $\left(\tilde{\Omega},\tilde{\mathcal{F}},(\tilde{\mathcal{F}}_t), \tilde{\mathbbm{P}}\right)$ is a filtered probability space;
    \item The component projections $(\mathcal{P}^k)$, $\mathcal{P}^k:\tilde{\Omega} \rightarrow \tilde{\Omega}^k$ and subsequently defined $(\hat{u}^k)$, $(\hat{\mathcal{W}}^k)$ by $$\hat{u}^k:= \tilde{u}^k\mathcal{P}^k, \qquad \hat{\mathcal{W}}^k= \tilde{\mathcal{W}}^k\mathcal{P}^k.$$
\end{itemize}
By construction for each $k$, $\hat{u}^k$ is a martingale weak solution of (\ref{projected Ito}) relative to the Cylindrical Brownian Motion $\hat{\mathcal{W}}^k$ and filtered probability space $\left(\tilde{\Omega},\tilde{\mathcal{F}},(\tilde{\mathcal{F}}_t), \tilde{\mathbbm{P}}\right)$. We can now make sense of taking the limit as $\nu \rightarrow 0$ in expectation for martingale weak solutions of (\ref{projected Ito}), by choosing an arbitrary sequence of viscosities $(\nu_k)$ convergent to zero and constructing the new solutions relative to a single probability space as above. We now fix this sequence and the related constructions.

\subsection{The Main Result} \label{sub based}

We state the main result of this section, and assess what needs to be proved. We introduce notation for $f \in W^{1,2}(\mathscr{O};\R^N)$, $$\norm{\nabla f}^2_{\Gamma_{c}} = \sum_{k=1}^N\norm{\partial_kf}_{L^2(\Gamma_c;\R^N)}^2$$ where for a constant $c$, $\Gamma_c$ is the boundary strip of width radius $c$, defined by the set of all points $x\in \mathscr{O}$ such that there exists a $y$ on the boundary with the distance from $x$ to $y$ less than $c$. Following the construction in Subsection \ref{subsection selection of martingale weak solutions}, for notational simplicity we consider an arbitrary $\hat{u}^k$ and relabel it as $u$, understanding that there is an implicit dependency on $\nu$ and that it is of course still defined over the new filtered probability space $\left(\tilde{\Omega},\tilde{\mathcal{F}},(\tilde{\mathcal{F}}_t), \tilde{\mathbbm{P}}\right)$. We shall use $\mathbbm{E}$ to represent the expectation taken with respect to $\tilde{\mathbbm{P}}$ on this probability space. We formally consider the process $\bar{u}$ representing the solution of the Euler Equation defined at (\ref{euler identity}) to reside on this space as a constant.

\begin{theorem} \label{the main one}
    The following conditions are equivalent:
    \begin{enumerate}
        \item $\mathbbm{E}\left(\sup_{r \in [0,T]}\norm{u_r-\bar{u}_r}^2\right) = o_{\nu}$,\label{item1}
        \item For every $t \in [0,T]$ and $\phi \in L^2(\Omega \times \mathscr{O};\R^N)$, $\mathbbm{E}\left(\inner{u_t - \bar{u}_t}{\phi}\right) = o_{\nu}$, \label{item2}
        \item $\nu \mathbbm{E}\int_0^T \norm{u_s}_1^2ds = o_{\nu}$, \label{item3}
        \item For any constant $\tilde{c}>0$, $\nu \mathbbm{E}\int_0^T \norm{\nabla u_s}^2_{\Gamma_{\tilde{c}\nu}}ds = o_{\nu}$. \label{item4}   
    \end{enumerate}
\end{theorem}

\begin{remark}
    This is a stochastic parallel of [\cite{kato1984remarks}] Theorem 1.
\end{remark}

We shall prove the theorem by demonstrating that \ref{item1} $\implies$ \ref{item2} $\implies$ \ref{item3} $\implies$ \ref{item4} $\implies$ \ref{item1}, and now identify what needs proving in this. Of course \ref{item1} $\implies$ \ref{item2} is trivial, as is \ref{item3} $\implies$ \ref{item4}. For \ref{item2} $\implies$ \ref{item3}, from (\ref{hello3}) we have that $$ \nu\mathbbm{E}\int_0^{T} \norm{u_s}^2_1 ds \leq \norm{u_0}^2 + o_{\nu} - \mathbbm{E}\left(\norm{u_{T}}^2\right).$$
To show that the limit exists and is zero it is sufficient to demonstrate that the limit supremum is zero, and as \begin{align*}\limsup_{\nu \rightarrow 0}\left[ \nu\mathbbm{E}\int_0^{T} \norm{u_s}^2_1 ds\right] &\leq \limsup_{\nu \rightarrow 0}\left[\norm{u_0}^2 - \mathbbm{E}\left(\norm{u_T}^2\right) + o_{\nu} \right]\\ &\leq \norm{u_0}^2 - \liminf_{\nu \rightarrow 0}\mathbbm{E}\left(\norm{u_T}^2\right)\end{align*}
then we only need to verify that \begin{equation}\label{this verification}\norm{u_0}^2 \leq \liminf_{\nu \rightarrow 0}\mathbbm{E}\left(\norm{u_{T}}^2\right).\end{equation} Item \ref{item2} is the statement that for every $t \in [0,T]$, $(u_t)$ converges to $\bar{u}_t$ weakly in $L^2(\Omega \times \mathscr{O};\R^N)$. With the known result that norms are weakly lower semicontinuous\footnote{Observe that if $(x_n)$ is weakly convergent to $x$, then $\norm{x}^2 = \inner{x}{x} = \lim_{n \rightarrow \infty}\inner{x_n}{x} = \liminf_{n \rightarrow \infty}\inner{x_n}{x} \leq \liminf_{n \rightarrow \infty}\norm{x_n}\norm{x}$ which implies $\norm{x}\leq \liminf_{n \rightarrow \infty}\norm{x_n}$.}, we employ the assumed Item \ref{item2} for time $T$ to see that $$\norm{\bar{u}_T}^2 \leq \liminf_{\nu \rightarrow 0}\mathbbm{E}\left(\norm{u_{T}}^2\right).$$ The property (\ref{this verification}) then follows from the energy identity (\ref{energy identity}).

\subsection{Proof of the Remaining Implication} \label{sub big based}

This subsection is dedicated to proving the final implication \ref{item4} $\implies$ \ref{item1} of Theorem \ref{the main one}. We recall a result proved in Kato's paper [\cite{kato1984remarks}], stated for a fixed $\nu < 1$. 

\begin{lemma} \label{v lemma}
    There exists a function $v \in C^1\left( [0,T] \times \bar{\mathscr{O}}; \R^N\right)$ and a constant $c$ (which may depend on $\tilde{c}$) such that:
    \begin{enumerate}
        \item For every $t\in [0,T]$, $v_t \in L^2_{\sigma}$, $v_t = \bar{u}_t$ on $\partial \mathscr{O}$ and $v_t$ is supported on $\Gamma_{\tilde{c}\nu}$, \label{number 1}
        \item $v$ satisfies the estimates
        \begin{align} \label{esty 1}
        \sup_{r \in [0,T]}\norm{v_r} &\leq c\nu^{\frac{1}{2}}\\ \label{esty 2}\sup_{r \in [0,T]}\norm{\partial_tv_r} &\leq c\nu^{\frac{1}{2}}\\ \label{esty 3}
        \sup_{r \in [0,T]}\norm{v_r}_{W^{1,2}} &\leq c\nu^{-\frac{1}{2}}
        \end{align}
         \item For any $f \in W^{1,2}_{\sigma}$, \begin{align} \label{esty 4}
             \sup_{r \in [0,T]}\left\vert\inner{\mathcal{L}_ff}{v_r}\right\vert \leq c\nu \norm{\nabla f}^2_{\Gamma_{\tilde{c}\nu}}.
         \end{align}
    \end{enumerate}

\end{lemma}

We explain the significance of the lemma now: with the identity
$$\norm{u_r - \bar{u}_r}^2 = \norm{u_r}^2 + \norm{\bar{u}_r}^2 - 2\inner{u_r}{\bar{u}_r}$$ then we want to make use of the formulation (\ref{identityindefinitionofspacetimeweak}), which we cannot immediately do as $\bar{u}$ does not necessarily vanish on the boundary. This is where we introduce $v$ from the lemma, as $v$ is prescribed to equal $\bar{u}$ on the boundary so that $\bar{u}-v$ satisfies the regularity of $\phi$ required in Definition \ref{definitionofspacetimeweak}. The idea is that the terms involving $\bar{u}$ can be well controlled using the smoothness of $\bar{u}$, and the terms involving $v$ require only an assumption on the energy dissipation within the boundary strip as this is where $v$ is supported. As $v$ is small in $L^2_{\sigma}$ with low viscosity, then the excess terms in $v$ will be small as well. Thus we rewrite \begin{equation}\label{thus we rewrite}\norm{u_r - \bar{u}_r}^2 = \norm{u_r}^2 + \norm{\bar{u}_r}^2 + 2\inner{u_r}{v_r} - 2\inner{u_r}{\bar{u}_r - v_r}\end{equation}
where \begin{align*} 
     \inner{u_r}{\bar{u}_r - v_r} &= \inner{u_0}{u_0 - v_0}  + \int_0^r\inner{u_s}{\partial_s(\bar{u}_s - v_s)}ds - \int_0^{r}\inner{\mathcal{L}_{u_s}u_s}{\bar{u}_s - v_s}ds\\ &- \nu\int_0^{r} \inner{u_s}{\bar{u}_s - v_s}_1 ds + \frac{\nu}{2}\int_0^{r}\sum_{i=1}^\infty \inner{\mathcal{Q}_iu_s}{\mathcal{Q}_i^*(\bar{u}_s - v_s)} ds - \nu^{\frac{1}{2}}\int_0^{t} \inner{\mathcal{G}u_s}{\bar{u}_s - v_s } d\mathcal{W}_s.
\end{align*}
Before taking the supremum and then expectation in the direction of Item \ref{item1}, we appreciate that $$\inner{u_0}{u_0 - v_0} = \norm{u_0}^2  -\inner{u_0}{v_0}$$ and with (\ref{euler identity}), (\ref{wloglhs}):
\begin{align*}
    \int_0^r\inner{u_s}{\partial_s(\bar{u}_s - v_s}ds &- \int_0^{r}\inner{\mathcal{L}_{u_s}u_s}{\bar{u}_s - v_s}ds\\ &=  \int_0^r \inner{u_s}{\partial_s\bar{u}_s} - \inner{u_s}{\partial_s v_s} -\inner{\mathcal{L}_{u_s}u_s}{\bar{u}_s} + \inner{\mathcal{L}_{u_s}u_s}{v_s}ds\\
    &=  \int_0^r - \inner{u_s}{\mathcal{P}\mathcal{L}_{\bar{u}_s}\bar{u}_s} - \inner{u_s}{\partial_s v_s} -\inner{\mathcal{L}_{u_s}u_s}{\bar{u}_s} + \inner{\mathcal{L}_{u_s}u_s}{v_s}ds\\
    &=  \int_0^r - \inner{u_s}{\mathcal{L}_{\bar{u}_s}\bar{u}_s} - \inner{u_s}{\partial_s v_s} + \inner{u_s}{\mathcal{L}_{u_s}\bar{u}_s} + \inner{\mathcal{L}_{u_s}u_s}{v_s}ds\\
    &= \int_0^r \inner{u_s}{\mathcal{L}_{u - \bar{u}_s}\bar{u}_s}ds + \int_0^r \inner{\mathcal{L}_{u_s}u_s}{v_s}ds  - \int_0^r\inner{u_s}{\partial_s v_s}ds.
\end{align*}
Substituting all of this into (\ref{thus we rewrite}) gives
\begin{align*}
    \norm{u_r - \tilde{u}_r}^2 &= \norm{u_r}^2 + \norm{\bar{u}_r}^2 + 2\inner{u_r}{v_r} - 2\norm{u_0}^2  +2\inner{u_0}{v_0}\\
    & -2\int_0^r \inner{u_s}{\mathcal{L}_{u - \bar{u}_s}\bar{u}_s}ds -2 \int_0^r \inner{\mathcal{L}_{u_s}u_s}{v_s}ds  + 2\int_0^r\inner{u_s}{\partial_s v_s}ds\\
    &+2 \nu\int_0^{r} \inner{u_s}{\bar{u}_s - v_s}_1 ds - \nu\int_0^{r}\sum_{i=1}^\infty \inner{\mathcal{Q}_iu_s}{\mathcal{Q}_i^*(\bar{u}_s - v_s)} ds +2 \nu^{\frac{1}{2}}\int_0^{t} \inner{\mathcal{G}u_s}{\bar{u}_s - v_s } d\mathcal{W}_s.
\end{align*}
We now take the supremum followed by the expectation, considering in the first line
\begin{align*}
    &\mathbbm{E}\left[\sup_{r \in [0,T]}\left(\norm{u_r}^2 + \norm{\bar{u}_r}^2\right) - 2\norm{u_0}^2 \right] +  \mathbbm{E}\left[\sup_{r \in [0,T]}2\inner{u_r}{v_r}+2\inner{u_0}{v_0}\right]\\
    & \qquad \qquad \qquad \qquad \leq \norm{u_0}^2 + o_{\nu} + \norm{u_0}^2 - 2\norm{u_0}^2 + 2\mathbbm{E}\left[\sup_{r\in[0,T]}\norm{u_r}\norm{v_r}\right] + 2\norm{u_0}\norm{v_0}
\end{align*}
having used (\ref{hello1}) and (\ref{energy identity}). Through another application of (\ref{hello1}) and employing (\ref{esty 1}), then this entire expression is bounded by $o_\nu$. Overall
\begin{align*}
    \mathbbm{E}\left(\sup_{r\in[0,T]}\norm{u_r-\bar{u}_r}^2\right) &\leq o_\nu + 2\mathbbm{E}\int_0^T\abs{\inner{u_s}{\mathcal{L}_{u_s - \bar{u}_s}\bar{u}_s}}ds\\ &+ 2\mathbbm{E}\int_0^{T}\abs{\inner{\mathcal{L}_{u_s}u_s}{v_s}}ds + 2\mathbbm{E}\int_0^{T}\abs{\inner{u_s}{\partial_sv_s}}ds\\ &+ 2\nu\mathbbm{E}\int_0^{T} \abs{\inner{u_s}{\bar{u}_s - v_s}_1} ds + \nu\mathbbm{E}\int_0^{T}\sum_{i=1}^\infty \abs{\inner{\mathcal{Q}_iu_s}{\mathcal{Q}_i^*(\bar{u}_s - v_s)}} ds\\ &+ 2\nu^{\frac{1}{2}}\mathbbm{E}\left(\sup_{r\in[0,T]}\left\vert\int_0^{r} \inner{\mathcal{G}u_s}{\bar{u}_s - v_s } d\mathcal{W}_s\right\vert\right)\\
    &:= o_{\nu} + \sum_{k=1}^6J_k
\end{align*}
and we now treat the integrals individually. For $J_1$, in the first line we use (\ref{cancellationproperty'}) and recall that $m>1+N/2$ is fixed from the start of Section \ref{section zero viscous}:
\begin{align*}
    \abs{\inner{u_s}{\mathcal{L}_{u_s - \bar{u}_s}\bar{u}_s}} &=  \abs{\inner{u_s}{\mathcal{L}_{u_s - \bar{u}_s}\bar{u}_s} - \inner{\bar{u}_s}{\mathcal{L}_{u_s - \bar{u}_s}\bar{u}_s}}\\
    &= \abs{\inner{u_s - \bar{u}_s}{\mathcal{L}_{u_s - \bar{u}_s}\bar{u}_s}}\\
    &\leq \sum_{k=1}^N\left\vert\inner{u_s - \bar{u}_s}{(u_s-\bar{u}_s)^k\partial_k\bar{u}_s}\right\vert\\
    &\leq \sum_{k=1}^N\sum_{l=1}^N \left\vert\inner{(u_s - \bar{u}_s)^l}{(u_s-\bar{u}_s)^k\partial_k\bar{u}^l_s}_{L^2(\mathscr{O};\R)}\right\vert\\
    &\leq \sum_{k=1}^N\sum_{l=1}^N \norm{(u_s - \bar{u}_s)^l(u_s-\bar{u}_s)^k}_{L^1(\mathscr{O};\R)}\norm{\partial_k\bar{u}^l_s}_{L^\infty(\mathscr{O};\R)}\\
    &\leq \sum_{k=1}^N\sum_{l=1}^N \norm{(u_s - \bar{u}_s)^l}_{L^2(\mathscr{O};\R)}\norm{(u_s-\bar{u}_s)^k}_{L^2(\mathscr{O};\R)}\norm{\bar{u}^l_s}_{W^{1,\infty}(\mathscr{O};\R)}\\
    &\leq c\norm{u_s - \bar{u}_s}^2\norm{\bar{u}_s}_{W^{m,2}}.
\end{align*}
Just as we did with the initial condition, we will now freely assimilate finite norms of $\bar{u}$ into our constants. As $\bar{u}\in C\left([0,T];W^{m,2}(\mathscr{O};\R^N)\right)$, then we bound the above by simply $c\norm{u_s - \bar{u}_s}^2$ so
$$J_1 := 2\mathbbm{E}\int_0^T\abs{\inner{u_s}{\mathcal{L}_{u_s - \bar{u}_s}\bar{u}_s}}ds \leq c\mathbbm{E}\int_0^T\norm{u_s - \bar{u}_s}^2ds.$$
With (\ref{esty 4}) then we have
$$ J_2:=2\mathbbm{E}\int_0^{T}\abs{\inner{\mathcal{L}_{u_s}u_s}{v_s}}ds \leq c\nu\mathbbm{E}\int_0^{T} \norm{\nabla u_s}_{\Gamma_{\tilde{c}\nu}}^2ds = o_{\nu}$$ critically applying the assumption \ref{item4}, and likewise
$$J_3:= 2\mathbbm{E}\int_0^{T}\abs{\inner{u_s}{\partial_sv_s}}ds \leq 2\mathbbm{E}\int_0^{T}\norm{u_s}\norm{\partial_sv_s}ds \leq c\nu\mathbbm{E}\int_0^{T}\norm{u_s}ds = o_{\nu}$$
using (\ref{esty 2}) then (\ref{hello1}). Next we treat $$J_4:=2\nu\mathbbm{E}\int_0^{T} \abs{\inner{u_s}{\bar{u}_s - v_s}_1} ds \leq 2\nu\mathbbm{E}\int_0^{T} \sum_{k=1}^N\abs{\inner{\partial_ku_s}{\partial_k\bar{u}_s}} ds + 2\nu\mathbbm{E}\int_0^{T} \sum_{k=1}^N\abs{\inner{\partial_ku_s}{\partial_kv_s}} ds$$ with the integrals individually:
\begin{align*}
    2\nu\mathbbm{E}\int_0^{T} \sum_{k=1}^N\abs{\inner{\partial_ku_s}{\partial_k\bar{u}_s}} ds &\leq c\nu\mathbbm{E}\int_0^{T} \norm{u_s}_1\norm{\bar{u}_s}_{W^{1,2}} ds\\ &\leq c\nu\mathbbm{E}\int_0^{T} \norm{u_s}_1 ds\\ &\leq c\nu^{\frac{1}{2}}\left(\nu\mathbbm{E}\int_0^{T} \norm{u_s}_1^2 ds  \right)^{\frac{1}{2}}\\ &= c\nu^{\frac{1}{2}} = o_\nu
\end{align*}
having used (\ref{hello3}). With the fact that $v$ has support in $\Gamma_{\tilde{c}\nu}$ and property (\ref{esty 4}) then
\begin{align*}
    2\nu\mathbbm{E}\int_0^{T} \sum_{k=1}^N\abs{\inner{\partial_ku_s}{\partial_kv_s}} ds &= 2\nu\mathbbm{E}\int_0^{T} \sum_{k=1}^N\abs{\inner{\partial_ku_s}{\partial_kv_s}_{L^2(\Gamma_{\tilde{c}\nu};\R^N)}}ds\\
    &\leq c\nu\mathbbm{E}\int_0^{T} \norm{\nabla u_s}_{\Gamma_{\tilde{c}\nu}}\norm{v_s}_{W^{1,2}}ds\\
    &\leq c\nu^{\frac{1}{2}}\mathbbm{E}\int_0^{T} \norm{\nabla u_s}_{\Gamma_{\tilde{c}\nu}}ds\\
    &= c\left(\nu \mathbbm{E}\int_0^{T} \norm{\nabla u_s}_{\Gamma_{\tilde{c}\nu}}^2ds\right)^{\frac{1}{2}}\\
    &= o_{\nu}
\end{align*}
using the assumption \ref{item4} once more, hence $J_4 = o_\nu$. We now move on to the noise terms, the first of which is dealt with near identically:
$$J_5:=\nu\mathbbm{E}\int_0^{T}\sum_{i=1}^\infty \abs{\inner{\mathcal{Q}_iu_s}{\mathcal{Q}_i^*(\bar{u}_s - v_s)}} ds \leq \nu\mathbbm{E}\int_0^{T}\sum_{i=1}^\infty \abs{\inner{\mathcal{Q}_iu_s}{\mathcal{Q}_i^*\bar{u}_s}}ds + \nu\mathbbm{E}\int_0^{T}\sum_{i=1}^\infty \abs{\inner{\mathcal{Q}_iu_s}{\mathcal{Q}_i^*v_s}}ds$$
having used the linearity of $\mathcal{Q}_i^*$. Then from (\ref{assumpty 1}) and (\ref{assumpty 3}),
\begin{align*}
    \nu\mathbbm{E}\int_0^{T}\sum_{i=1}^\infty \abs{\inner{\mathcal{Q}_iu_s}{\mathcal{Q}_i^*\bar{u}_s}}ds &\leq c\nu\mathbbm{E}\int_0^{T} \left( 1 + \norm{u_s}_1\right)\norm{\bar{u}_s}_{W^{1,2}}ds \leq c\nu^{\frac{1}{2}}\left(\nu\mathbbm{E}\int_0^{T}\left(1 + \norm{u_s}_1\right)^2 ds  \right)^{\frac{1}{2}}
\end{align*}
which we write as $o_\nu$ using (\ref{hello3}) again. For the second integral we must use the assumed structure $\mathcal{Q}_i^* = \mathcal{A}_i + \hat{\mathcal{A}_i}$ with $\mathcal{A}_i$ preserving the support on $\Gamma_{\tilde{c}\nu}$,
\begin{align*}\nu\mathbbm{E}\int_0^{T}\sum_{i=1}^\infty \abs{\inner{\mathcal{Q}_iu_s}{\mathcal{Q}_i^*v_s}}ds &= \nu\mathbbm{E}\int_0^{T}\sum_{i=1}^\infty \abs{\inner{\mathcal{Q}_iu_s}{\mathcal{A}_iv_s}_{L^2(\Gamma_{\tilde{c}\nu};\R^N)} + \inner{\mathcal{Q}_iu_s}{\hat{\mathcal{A}_i}v_s}}ds\\
&\leq c\nu\mathbbm{E}\int_0^{T}\left( 1 + \norm{u_s}_{W^{1,2}(\Gamma_{\tilde{c}\nu};\R^N)}\right)\norm{v_s}_{W^{1,2}} + \left( 1 + \norm{u_s}_1\right)\norm{v_s}ds\\
&= c\nu\mathbbm{E}\int_0^{T}\left( 1 + \norm{u_s}_{W^{1,2}(\Gamma_{\tilde{c}\nu};\R^N)}\right)\norm{v_s}_{W^{1,2}}ds + o_\nu\\
&\leq c\nu^{\frac{1}{2}}\mathbbm{E}\int_0^{T}1 + \norm{u_s} +\norm{\nabla u_s}_{\Gamma_{\tilde{c}\nu}}ds + o_\nu\\
&= o_{\nu} + c\left(\nu \mathbbm{E}\int_0^{T} \norm{\nabla u_s}_{\Gamma_{\tilde{c}\nu}}^2ds\right)^{\frac{1}{2}}\\
&= o_{\nu}
\end{align*}
with (\ref{hello1}) and the assumption \ref{item4}. Therefore $J_5 = o_\nu$. Only one term now remains:
\begin{align*}
    J_6&:= 2\nu^{\frac{1}{2}}\mathbbm{E}\left(\sup_{r\in[0,T]}\left\vert\int_0^{r} \inner{\mathcal{G}u_s}{\bar{u}_s - v_s } d\mathcal{W}_s\right\vert\right)\\ &\leq c\nu^{\frac{1}{2}}\mathbbm{E}\left(\int_0^{T} \sum_{i=1}^\infty \inner{\mathcal{G}_iu_s}{\bar{u}_s - v_s }^2 ds\right)^{\frac{1}{2}}\\
    &\leq c\nu^{\frac{1}{2}}\mathbbm{E}\left(\int_0^{T} \sum_{i=1}^\infty \inner{\mathcal{G}_iu_s}{\bar{u}_s}^2 ds\right)^{\frac{1}{2}} + c\nu^{\frac{1}{2}}\mathbbm{E}\left(\int_0^{T} \sum_{i=1}^\infty \inner{\mathcal{G}_iu_s}{v_s }^2 ds\right)^{\frac{1}{2}}
\end{align*}
having used the BDG Inequality. Using (\ref{assumpty9}) we see that
\begin{align*}
    c\nu^{\frac{1}{2}}\mathbbm{E}\left(\int_0^{T} \sum_{i=1}^\infty \inner{\mathcal{G}_iu_s}{\bar{u}_s}^2 ds\right)^{\frac{1}{2}} &\leq c\nu^{\frac{1}{2}}\mathbbm{E}\left(\int_0^{T}  \left[1 + \norm{u_s}^2 + \norm{\bar{u}_s}^p\right]\norm{\bar{u}_s}^2_{W^{1,2}} ds\right)^{\frac{1}{2}}\\ &= c\nu^{\frac{1}{2}}\mathbbm{E}\left(\int_0^{T} 1 + \norm{u_s}^2 ds\right)^{\frac{1}{2}}
\end{align*}
which is just $o_\nu$ again from (\ref{hello1}), and with (\ref{assumpty 1}) then
\begin{align*}
    c\nu^{\frac{1}{2}}\mathbbm{E}\left(\int_0^{T} \sum_{i=1}^\infty \inner{\mathcal{G}_iu_s}{v_s }^2 ds\right)^{\frac{1}{2}} &= c\nu^{\frac{1}{2}}\mathbbm{E}\left(\int_0^{T} \sum_{i=1}^\infty \inner{\mathcal{G}_iu_s}{v_s }_{L^2(\Gamma_{\tilde{c}\nu};\R^N)}^2 ds\right)^{\frac{1}{2}}\\
    &\leq c\nu^{\frac{1}{2}}\mathbbm{E}\left(\int_0^{T} \left[1 + \norm{u_s}_{W^{1,2}(\Gamma_{\tilde{c}\nu};\R^N)}^2\right]\norm{v_s }^2 ds\right)^{\frac{1}{2}}\\
    &\leq c\nu\mathbbm{E}\left(\int_0^{T} 1 + \norm{u_s}_{W^{1,2}(\Gamma_{\tilde{c}\nu};\R^N)}^2 ds\right)^{\frac{1}{2}}\\
    &\leq c\nu\mathbbm{E}\left(\int_0^{T} 1 + \norm{u_s}^2 + \norm{\nabla u_s}_{\Gamma_{\tilde{c}\nu}}^2 ds\right)^{\frac{1}{2}}\\
    &= o_\nu.
\end{align*}
Therefore
$$ \mathbbm{E}\left(\sup_{r\in[0,T]}\norm{u_r-\bar{u}_r}^2\right) \leq o_v + \sum_{k=1}^6J_k \leq o_\nu + \int_0^T\mathbbm{E}\left(\norm{u_s-\bar{u}_s}^2\right)ds$$
so from the standard Gr\"{o}nwall Inequality, $\mathbbm{E}\left(\sup_{r\in[0,T]}\norm{u_r-\bar{u}_r}^2\right) \leq o_{\nu}$ which gives the result.

\subsection{Optimal Noise Scaling} \label{subsection optimal noise scaling}

To further motivate the scaling rate of $\nu^\frac{1}{2}$ in the stochastic integral, we consider a different parameter $\mu$ in the equation
\begin{equation} \label{projected Ito viscous 2}
    u_t = u_0 - \int_0^t\mathcal{P}\mathcal{L}_{u_s}u_s\ ds - \nu\int_0^t A u_s\, ds + \frac{\mu}{2}\int_0^t\sum_{i=1}^\infty \mathcal{P}\mathcal{Q}_i^2u_s ds - \mu^{\frac{1}{2}}\int_0^t \mathcal{P}\mathcal{G}u_s d\mathcal{W}_s 
\end{equation}
where we shall consider the limit $\mu \rightarrow 0$ with $\nu$ formally as in Subsection \ref{subsection selection of martingale weak solutions} by a sequence $(\mu_k)$ such that $\mu_k \rightarrow 0$ as $k \rightarrow \infty$. Corresponding notation $o_{\nu,\mu}$ is introduced to mean any constant dependent on $\nu$ and $\mu$ such that $\lim_{\nu,\mu \rightarrow 0}o_{\nu,\mu} = 0$ for the limit taken jointly, which again is formally understood by $\lim_{k \rightarrow \infty}o_{\nu_k,\mu_k} = 0$. We shall similarly use $o_{\mu}$. Our assumptions on the noise must now be tweaked to accommodate this difference, and by replacing $\mathcal{G}$ with $(\frac{\mu}{\nu})^{\frac{1}{2}}\mathcal{G}$ then we see that the only necessary change of assumption is in (\ref{assumpty 7}) and (\ref{assumpty 5}) where we now impose that $\sum_{i=1}^\infty k_i \leq (\frac{\nu}{\mu})^{\frac{1}{2}}$. The following lemma is deduced exactly as in Proposition \ref{ex theorem}.
\begin{lemma} \label{ex prop}
    There exists a martingale weak solution $\tilde{u}$ of the equation (\ref{projected Ito viscous 2}) which satisfies
    \begin{align} \label{hello6}
        \tilde{\mathbbm{E}}\left(\sup_{r\in[0,T]}\norm{\tilde{u}_{r}}^2 \right) &\leq \left(1 + o_{\mu}\right) \norm{u_0}^2 + o_{\mu}        
    \end{align}
    and for every $t \in [0,T]$, \begin{equation} \label{hello9}
        \tilde{\mathbbm{E}}\left[\norm{\tilde{u}_{t}}^2 + \nu\int_0^{t} \norm{\tilde{u}_s}^2_1 ds\right] \leq \left(1 + o_{\mu}\right) \norm{u_0}^2 + o_{\mu}.
         \end{equation}
\end{lemma}

The selection of martingale weak solutions is now as in Subsection \ref{subsection selection of martingale weak solutions}, for the solution $\tilde{u}^k$ corresponding to parameters $\nu_k,\mu_k$. We can now state the main result in this context. 

\begin{proposition} \label{theorem with mew}
    Suppose that $\mu \nu^{-\frac{1}{2}} = o_{\nu,\mu}$ and for any constant $\tilde{c}>0$, $$\left(1 + \frac{\mu^2}{\nu^2}\right) \nu\mathbbm{E}\int_0^T \norm{\nabla u_s}^2_{\Gamma_{\tilde{c}\nu}}ds = o_{\nu,\mu}.$$ Then $\mathbbm{E}\left(\sup_{r \in [0,T]}\norm{u_r-\bar{u}_r}^2\right) = o_{\nu,\mu}$.
\end{proposition}

\begin{remark}
    We do not achieve an equivalence of conditions as in Theorem \ref{the main one}, as $\frac{\mu^2}{\nu}\mathbbm{E}\int_0^T \norm{u_s}^2_1ds = o_{\nu,\mu}$ may not be necessary for the weak convergence in item \ref{item2}. 
\end{remark}

\begin{proof}
    Identically to Theorem \ref{the main one}, we obtain that
    \begin{align*}
    \mathbbm{E}\left(\sup_{r\in[0,T]}\norm{u_r-\bar{u}_r}^2\right) &\leq o_{\nu,\mu} + 2\mathbbm{E}\int_0^T\abs{\inner{u_s}{\mathcal{L}_{u_s - \bar{u}_s}\bar{u}_s}}ds\\ &+ 2\mathbbm{E}\int_0^{T}\abs{\inner{\mathcal{L}_{u_s}u_s}{v_s}}ds + 2\mathbbm{E}\int_0^{T}\abs{\inner{u_s}{\partial_sv_s}}ds\\ &+ 2\nu\mathbbm{E}\int_0^{T} \abs{\inner{u_s}{\bar{u}_s - v_s}_1} ds + \mu\mathbbm{E}\int_0^{T}\sum_{i=1}^\infty \abs{\inner{\mathcal{Q}_iu_s}{\mathcal{Q}_i^*(\bar{u}_s - v_s)}} ds\\ &+ 2\mu^{\frac{1}{2}}\mathbbm{E}\left(\sup_{r\in[0,T]}\left\vert\int_0^{r} \inner{\mathcal{G}u_s}{\bar{u}_s - v_s } d\mathcal{W}_s\right\vert\right)\\
    &:= o_{\nu,\mu} + \sum_{k=1}^6J_k.
\end{align*}
The integrals $J_1$ to $J_4$ are controlled in the same way so we begin with $J_5$, again writing $$J_5 \leq \mu\mathbbm{E}\int_0^{T}\sum_{i=1}^\infty \abs{\inner{\mathcal{Q}_iu_s}{\mathcal{Q}_i^*\bar{u}_s}}ds + \mu\mathbbm{E}\int_0^{T}\sum_{i=1}^\infty \abs{\inner{\mathcal{Q}_iu_s}{\mathcal{Q}_i^*v_s}}ds$$
and further 
\begin{align*}
    \mu\mathbbm{E}\int_0^{T}\sum_{i=1}^\infty \abs{\inner{\mathcal{Q}_iu_s}{\mathcal{Q}_i^*\bar{u}_s}}ds &\leq c\mu\mathbbm{E}\int_0^{T} \left( 1 + \norm{u_s}_1\right)ds \leq c\mu\nu^{-\frac{1}{2}}\left(\nu\mathbbm{E}\int_0^{T}\left(1 + \norm{u_s}_1\right)^2 ds  \right)^{\frac{1}{2}}
\end{align*}
so from (\ref{hello9}) and the first assumption, this is just $o_{\nu,\mu}$. In addition
\begin{align*}\mu\mathbbm{E}\int_0^{T}\sum_{i=1}^\infty \abs{\inner{\mathcal{Q}_iu_s}{\mathcal{Q}_i^*v_s}}ds
&\leq c\mu\nu^{-\frac{1}{2}}\mathbbm{E}\int_0^{T}1 + \norm{u_s} +\norm{\nabla u_s}_{\Gamma_{\tilde{c}\nu}}ds\\
&\leq  c\mu\nu^{-\frac{1}{2}} + c\left(\frac{\mu^2}{\nu} \mathbbm{E}\int_0^{T} \norm{\nabla u_s}_{\Gamma_{\tilde{c}\nu}}^2ds\right)^{\frac{1}{2}}\\
&= o_{\nu,\mu}
\end{align*}
from (\ref{hello6}) and the assumptions. This demonstrates that $J_5 = o_{\nu,\mu}$ so it only remains to consider $J_6$, which we again estimate with 
\begin{align*}
    J_6 \leq c\mu^{\frac{1}{2}}\mathbbm{E}\left(\int_0^{T} \sum_{i=1}^\infty \inner{\mathcal{G}_iu_s}{\bar{u}_s}^2 ds\right)^{\frac{1}{2}} + c\mu^{\frac{1}{2}}\mathbbm{E}\left(\int_0^{T} \sum_{i=1}^\infty \inner{\mathcal{G}_iu_s}{v_s }^2 ds\right)^{\frac{1}{2}}.
\end{align*}
The first term is controlled as before, and for the second
\begin{align*}
    c\mu^{\frac{1}{2}}\mathbbm{E}\left(\int_0^{T} \sum_{i=1}^\infty \inner{\mathcal{G}_iu_s}{v_s }^2 ds\right)^{\frac{1}{2}}
    &\leq c\mu^{\frac{1}{2}}\mathbbm{E}\left(\int_0^{T} \left[1 + \norm{u_s}_{W^{1,2}(\Gamma_{\tilde{c}\nu};\R^N)}^2\right]\norm{v_s }^2 ds\right)^{\frac{1}{2}}\\
    &\leq c\mu^{\frac{1}{2}}\left(\nu\mathbbm{E}\int_0^{T} 1 + \norm{u_s}_{W^{1,2}(\Gamma_{\tilde{c}\nu};\R^N)}^2 ds\right)^{\frac{1}{2}}\\
    &\leq c\mu^{\frac{1}{2}}\left(\nu\mathbbm{E}\int_0^{T} 1 + \norm{u_s}^2 + \norm{\nabla u_s}_{\Gamma_{\tilde{c}\nu}}^2 ds\right)^{\frac{1}{2}}\\
    &= o_{\nu,\mu}
\end{align*}
which completes the proof.
\end{proof}

An interesting corollary of this result comes from considering $\mu = \nu^{\alpha}$ for different values of $\alpha$.

\begin{corollary}
    Suppose that $\mu = \nu^{\alpha}$ in (\ref{projected Ito viscous 2}). If $\frac{1}{2} < \alpha < 1$ then the condition $$\nu^{2\left(\alpha - \frac{1}{2}\right)}\mathbbm{E}\int_0^T \norm{\nabla u_s}^2_{\Gamma_{\tilde{c}\nu}}ds = o_{\nu}$$ implies that $\mathbbm{E}\left(\sup_{r \in [0,T]}\norm{u_r-\bar{u}_r}^2\right) = o_{\nu}$. If $\alpha \geq 1$ then all four conditions of Theorem \ref{the main one} are equivalent. 
\end{corollary}

\begin{proof}
    For $\mu = \nu^{\alpha}$ and $\frac{1}{2} < \alpha < 1$ it is sufficient to note that $\mu\nu^{-\frac{1}{2}} = \nu^{\alpha - \frac{1}{2}} = o_{\nu}$ as well as that $$\frac{\mu^2}{\nu} = \frac{\nu^{2\alpha}}{\nu} = \nu^{2\left(\alpha - \frac{1}{2}\right)}$$ which is also greater than or equal to $\nu$ (recall $\nu < 1$), so the assumptions of Proposition \ref{theorem with mew} are satisfied which proves the result in this case. In the case $\alpha \geq 1$ then we have $\nu \geq \nu^{2\left(\alpha - \frac{1}{2}\right)}$ so this implication again holds assuming just \ref{item4} in Theorem \ref{the main one}. The proof of the remaining implications is unchanged from Subsection \ref{sub based}.
\end{proof}

\begin{remark}
    For $\frac{1}{2} < \alpha < 1$ we cannot show the four conditions are equivalent as in order to demonstrate that \ref{item2} $\implies$ \ref{item3} in the same way we would need an estimate $$\nu^{2\left(\alpha - \frac{1}{2}\right)}\mathbbm{E}\int_0^T \norm{\nabla u_s}^2_{1}ds \leq C$$ for some $C$ independent of $\nu$, which we cannot achieve.
\end{remark}
This corollary highlights the critical cases of $\alpha = \frac{1}{2}$ and $\alpha = 1$.\\

\textbf{Case $\alpha = \frac{1}{2}$:} We lose the key property that $\nu^{\alpha - \frac{1}{2}} = o_{\nu}$ which was used in control of $J_5$, so even the limiting condition that $$\mathbbm{E}\int_0^T \norm{\nabla u_s}^2_{\Gamma_{\tilde{c}\nu}}ds = o_{\nu} $$ is insufficient for the result. Inspecting the proof though, one can estimate $J_5$ and hence show the result with the stronger assumptions $$\mathbbm{E}\int_0^T \norm{\nabla u_s}^2_{W^{1,2}\left(\Gamma_{\tilde{c}\nu};\R^N\right)}ds = o_{\nu} \qquad \textnormal{and} \qquad \nu\mathbbm{E}\int_0^T \norm{u_s}^2_{1}ds = o_{\nu}$$ noting that the second assumption is item \ref{item3} of Theorem \ref{the main one}.\\

\textbf{Case $\alpha = 1$:} This is the smallest $\alpha$ in which we can show the equivalence of all conditions as well as the smallest $\alpha$ in which we do not need to impose a stronger condition than Kato's (\ref{item4}) to determine the convergence. The scaling rate of $\mu = \nu$ is thus considered optimal.

\section{Weak Solutions of the Stochastic Navier-Stokes Equation} \label{section weak solutions}

This section is dedicated to the four outstanding proofs regarding weak solutions of the stochastic Navier-Stokes Equation.

\begin{itemize}
    \item The equaivalence of the notions of weak solution (Proposition \ref{prop for equivalence}) is proven in Subsection \ref{sub equivalence}.
    \item The existence of martingale weak solutions of the equation (\ref{projected Ito}), Theorem \ref{existence of weak}, is proven across Subsections \ref{sub galerkin}, \ref{sub tight} and \ref{existence for bounded ic}. An approximating sequence of finite dimensional solutions is considered in Subsection \ref{sub galerkin}, which is shown to satisfy tightness properties in Subsection \ref{sub tight}. We can then pass to the limit of this approximation and show that this limit is a solution in Subsection \ref{existence for bounded ic}.
    \item The existence and uniqueness of weak solutions of the equation (\ref{projected Ito}) in 2D, Theorem \ref{theorem 2D}, is proven in Subsection \ref{sub 2d}. A classical Yamada-Watanabe type result allows us to pass from the martingale weak solutions of Theorem \ref{existence of weak} to (probabilistically strong) weak solutions as considered in Theorem \ref{theorem 2D}.
    \item The precise energy estimates, Proposition \ref{ex theorem}, is proven in Subsection \ref{sub energy}.
\end{itemize}

\subsection{Equivalence of the Notions of Weak Solution} \label{sub equivalence}

We prove Proposition \ref{prop for equivalence}.

\begin{proof}[Proof of Proposition \ref{prop for equivalence}:]
    We fix a $u$ with the regularity specified in the definitions, and wish to show the relation (\ref{identityindefinitionofspatiallyweak}) $\iff$ (\ref{identityindefinitionofspacetimeweak}). We consider the two implications:
    \begin{itemize}
        \item[$\impliedby$:] We fix a $\phi \in W^{1,2}_{\sigma}$, and note that for any constant function in time $\psi_s = \psi$, $\psi\in C^1\left(\mathscr{O};\R^N\right) \cap W^{1,2}_{\sigma}$, we do indeed have the identity (\ref{identityindefinitionofspatiallyweak}) for $\psi$ as the time derivative is null. Given that even $C^{\infty}_{0,\sigma}(\mathscr{O};\R^N)$ is dense in $W^{1,2}_{\sigma}$ then most certainly too is $C^1\left(\overbar{\mathscr{O}};\R^N\right) \cap W^{1,2}_{\sigma}$ so we can take a sequence $(\psi^n)$ in this space convergent to $\phi$ in $\norm{\cdot}_1$. The result then follows from a straightforwards application of the dominated convergence theorem, noting (\ref{assumpty 3}) and showing convergence in $L^2\left(\Omega \times [0,t];\R\right)$ of the truncated integrand in the stochastic integral (where truncation is up to the localising stopping time).\\
        \item[$\implies$:] We fix a $\phi \in C^1\left([0,t] \times \overbar{\mathscr{O}}; \R^N\right)$ such that $\phi_s \in W^{1,2}_{\sigma}$ for every $s \in [0,t]$, and shall consider simple approximations of this $\phi$ and use the relation (\ref{identityindefinitionofspatiallyweak}) on the constant time steps. To this end for every stopping time $\tau$ with $u_\cdot \mathbbm{1}_{\cdot \leq \tau} \in L^2\left(\Omega \times [0,t]; W^{1,2}_{\sigma}\right)$, we introduce a sequence of partitions $$I_l \subset I_{l+1}, \quad I_l:=\left\{0=t^l_0 < t^l_1 < \dots < t^l_{k_l} = t \right\}, \quad \max_{j}\abs{t^l_j-t^l_{j-1}} \rightarrow 0 \textnormal{ as } l \rightarrow \infty$$ which are such that the process defined by $$\tilde{u}^l_s(\omega):= \sum_{j=1}^{k_l-1}\mathbbm{1}_{[t^l_{j-1},t^l_j]}(s)u_{t^l_j}(\omega)\mathbbm{1}_{t^l_j \leq \tau(\omega)}$$ belongs to $L^2\left(\Omega \times [0,t]; W^{1,2}_{\sigma}\right)$ and converges to $u_{\cdot}\mathbbm{1}_{\cdot \leq \tau}$ in this space. Such a partition is available to us due to Lemma 4.2.6 of [\cite{prevot2007concise}]\footnote{ We would formally have to apply this for the modification $\hat{u}$ of $u$ which is genuinely progressively measurable, but then Remark 4.2.7 ensures that we could choose our partition (except for $t^l_0$, $t^l_{k_l}$) outside of this $\lambda-$zero on which the modification occurs. Thus $u$ freely replaces $\hat{u}$ in the definition of $\tilde{u}^l$.}.  Now for each $l,j$ from the identity (\ref{identityindefinitionofspatiallyweak}) we have that
        \begin{align} \nonumber
     \inner{u_{t^l_j}}{\phi_{t^l_j}} = \inner{u_{t^l_{j-1}}}{\phi_{t^l_{j}}} &- \int_{t^l_{j-1}}^{t^l_j}\inner{\mathcal{L}_{u_s}u_s}{\phi_{t^l_{j}}}ds - \nu\int_{t^l_{j-1}}^{t^l_{j}} \inner{u_s}{\phi_{t^l_{j}}}_1 ds\\ &+ \frac{\nu}{2}\int_{t^l_{j-1}}^{t^l_{j}}\sum_{i=1}^\infty \inner{\mathcal{Q}_iu_s}{\mathcal{Q}_i^*\phi_{t^l_{j}}} ds - \nu^{\frac{1}{2}}\int_{t^l_{j-1}}^{t^l_{j}} \inner{\mathcal{G}u_s}{\phi_{t^l_{j}}} d\mathcal{W}_s \label{anotheralign}
\end{align}
which we shall utilise in the context of \begin{align}\nonumber \inner{u_t}{\phi_t} - \inner{u_0}{\phi_0} &= \sum_{j=1}^{k_l}\left(\inner{u_{t^l_{j}}}{\phi_{t^l_j}} -  \inner{u_{t^l_{j-1}}}{\phi_{t^l_{j-1}}}\right)\\ &= \sum_{j=1}^{k_l}\left( \inner{u_{t^l_{j}}}{\phi_{t^l_j}} - \inner{u_{t^l_{j-1}}}{\phi_{t^l_{j}}}\right) + \sum_{j=1}^{k_l}\left(\inner{u_{t^l_{j-1}}}{\phi_{t^l_{j}}}-  \inner{u_{t^l_{j-1}}}{\phi_{t^l_{j-1}}}\right). \label{context} \end{align}
To address the first sum we introduce some familiar notation in $$\tilde{\phi}^l_s:= \sum_{j=1}^{k_l-1}\mathbbm{1}_{[t^l_{j-1},t^l_j]}(s)\phi_{t^l_j},$$ allowing us to deduce from (\ref{anotheralign}) that 
\begin{align} \nonumber
     \sum_{j=1}^{k_l}\left(\inner{u_{t^l_j}}{\phi_{t^l_j}} - \inner{u_{t^l_{j-1}}}{\phi_{t^l_{j}}}\right) =  &- \int_{0}^{t}\inner{\mathcal{L}_{u_s}u_s}{\tilde{\phi}^l_{s}}ds - \nu\int_{0}^{t} \inner{u_s}{\tilde{\phi}^l_{s}}_1 ds\\ &+ \frac{\nu}{2}\int_{0}^{t}\sum_{i=1}^\infty \inner{\mathcal{Q}_iu_s}{\mathcal{Q}_i^*\tilde{\phi}^l_{s}} ds - \nu^{\frac{1}{2}}\int_{0}^{t} \inner{\mathcal{G}u_s}{\tilde{\phi}^l_{s}} d\mathcal{W}_s. \label{anotheralign2}
\end{align}
As for the second sum of (\ref{context}), observe that for each $j$ we can rewrite this as
\begin{align*}\inner{u_{t^l_{j-1}}}{\phi_{t^l_{j}}}-  \inner{u_{t^l_{j-1}}}{\phi_{t^l_{j-1}}} &= \int_{\mathscr{O}}u_{t^l_{j-1}}(x)\left[\phi_{t^l_{j}}(x) - \phi_{t^l_{j-1}}(x)\right]dx\\
&= \int_{\mathscr{O}}u_{t^l_{j-1}}(x)\left[\int_{t^l_{j-1}}^{t^l_j}\partial_s \phi_s(x)\right]dx\\
&= \int_{t^l_{j-1}}^{t^l_j} \inner{\tilde{u}^l_s}{\partial_s\phi_s} ds
\end{align*}
using the $C^1\left([0,t]\times \overbar{\mathscr{O}};\R^N\right)$ regularity of $\phi$ and Fubini's Theorem. Summation then gives the integral over $[0,t]$ hence we deduce that
\begin{align} \nonumber
     \inner{u_t}{\phi_t} = \inner{u_0}{\phi_0}  + \int_0^t\inner{\tilde{u}^l_s}{\partial_s\phi_s}ds &- \int_0^{t}\inner{\mathcal{L}_{u_s}u_s}{\tilde{\phi}^l_s}ds - \nu\int_0^{t} \inner{u_s}{\tilde{\phi}^l_s}_1 ds\\ &+ \frac{\nu}{2}\int_0^{t}\sum_{i=1}^\infty \inner{\mathcal{Q}_iu_s}{\mathcal{Q}_i^*\tilde{\phi}^l_s} ds - \nu^{\frac{1}{2}}\int_0^{t} \inner{\mathcal{G}u_s}{\tilde{\phi}^l_s } d\mathcal{W}_s\label{again limits}
\end{align}
holds for all $l \in \N$. We look to analyse the limit, noting that the convergence $\tilde{u}^l \rightarrow u$ in $L^2\left(\Omega \times [0,t]; W^{1,2}_{\sigma}\right)$ implies that in $L^2\left(\Omega; L^2\left([0,t]; W^{1,2}_{\sigma}\right)\right)$ so we can extract a subsequence (relabelled again as $\tilde{u}^l$ for simplicity) which is convergent $\mathbbm{P}-a.e.$ in $L^2\left([0,t]; W^{1,2}_{\sigma}\right)$, from which we see that
\begin{align}
    \nonumber \left\vert \int_0^t\inner{\tilde{u}^l_s}{\partial_s\phi_s}ds - \int_0^t\inner{u_s}{\partial_s\phi_s}ds \right\vert &\leq \int_0^t\norm{\tilde{u}^l_s - u_s}\norm{\partial_s\phi_s}ds\\\nonumber  &\leq \norm{\tilde{u}^l_s - u_s}_{L^2\left([0,t];L^2(\mathscr{O};\R^N)\right)} \norm{\partial_s\phi_s}_{L^2\left([0,t];L^2(\mathscr{O};\R^N)\right)}\\
    &\longrightarrow 0. \label{exactly as in}
\end{align}
Addressing the convergence in the other terms, we need to establish in what sense $\tilde{\phi}^l \rightarrow \phi$. We claim that the convergence holds in $L^2\left([0,t];W^{1,2}_{\sigma}\right)$. As $\phi \in C^1\left([0,t]\times \overbar{\mathscr{O}};\R^N\right)$ then for each $j = 1, \dots, N$ and every fixed $x \in \mathscr{O}$ then $\partial_j\phi(x)\in C^1\left([0,t];\R^N\right)$ so is in particular Lipschitz Continuous. Of course $$ \partial_j\tilde{\phi}^l_s:= \sum_{j=1}^{k_l-1}\mathbbm{1}_{[t^l_{j-1},t^l_j]}(s)\partial_j\phi_{t^l_j}$$ so $$ \partial_j\phi_s - \partial_j\tilde{\phi}^l_s:= \sum_{j=1}^{k_l-1}\mathbbm{1}_{[t^l_{j-1},t^l_j]}(s)\left[\partial_j\phi_s - \partial_j\phi_{t^l_j}\right]$$ thus if $K_x$ is the constant of Lipschitz Continuity, then $$\norm{\partial_j\phi(x) - \partial_j\tilde{\phi}^l(x)}_{L^\infty\left([0,t];\R^N\right)} \leq K_x\max_j \abs{t^l_{j} - t^l_{j-1}}$$ which converges to zero as $l \rightarrow \infty$. Now then we have that
\begin{align*}
    \norm{\phi - \tilde{\phi}^l}^2_{L^2\left([0,t];W^{1,2}_{\sigma}\right)} &= \int_0^t\sum_{j=1}^N\norm{\partial_j\phi_s-\partial_j\tilde{\phi}^l_s}^2ds\\
    &= \sum_{j=1}^N \int_0^t\int_{\mathscr{O}}\abs{\partial_j\phi_s(x)-\partial_j\tilde{\phi}^l_s(x)}^2dxds\\
    &= \sum_{j=1}^N\int_{\mathscr{O}}\norm{\partial_j\phi(x)-\partial_j\tilde{\phi}^l(x)}_{L^2([0,t];\R^N)}^2dx\\
    &\leq c\sum_{j=1}^N\int_{\mathscr{O}}\norm{\partial_j\phi(x)-\partial_j\tilde{\phi}^l(x)}_{L^{\infty}([0,t];\R^N)}^2dx
\end{align*}
where $c$ continues to represent a generic constant. The claim then follows from the Dominated Convergence Theorem, with dominating function $4\norm{\partial_j\phi}_{L^{\infty}([0,t];\R^N)}^2$ which is continuous on $\overbar{\mathscr{O}}$ hence square integrable. In fact with the observation that the supremum of the integral is bounded by the integral of the supremum, we have in fact shown the stronger convergence in $L^\infty\left([0,t];W^{1,2}_{\sigma}\right)$. With the claim established we again look to take limits in (\ref{again limits}), and using the linearity of $\mathcal{Q}_i^*$ and (\ref{assumpty 3}) then the convergence of the time integrals follows exactly as in (\ref{exactly as in}). As for the stochastic integral we look to first show a convergence in expectation and make use of the It\^{o} Isometry, hence we introduce stopping times $$\theta_n := t \wedge \inf \left\{r \geq 0: \int_0^r \norm{u_s}_1^2ds \geq n \right\}$$ and consider \begin{align*}\mathbbm{E}\bigg\vert \int_0^{t\wedge {\theta_n}} \inner{\mathcal{G}u_s}{\tilde{\phi}^l_s } d\mathcal{W}_s &- \int_0^{t\wedge {\theta_n}} \inner{\mathcal{G}u_s}{\phi_s } d\mathcal{W}_s \bigg\vert^2\\ &= \mathbbm{E}\left\vert \int_0^{t\wedge {\theta_n}} \inner{\mathcal{G}u_s}{\tilde{\phi}^l_s - \phi_s } d\mathcal{W}_s \right\vert^2\\ &= \mathbbm{E} \int_0^{t\wedge {\theta_n}}\sum_{i=1}^\infty \inner{\mathcal{G}_iu_s}{\tilde{\phi}^l_s - \phi_s }^2 ds \\
&\leq \mathbbm{E} \int_0^{t\wedge {\theta_n}}\sum_{i=1}^\infty c_i\norm{u_s}_1^2\norm{\tilde{\phi}^l_s - \phi_s }^2 ds \\
&\leq \sum_{i=1}^\infty c_i \left[ \mathbbm{E}\int_0^{t\wedge {\theta_n}}\norm{u_s}_1^2ds \right]\norm{\tilde{\phi}^l - \phi }^2_{L^\infty([0,t];L^2_{\sigma})}
\end{align*}
which approaches zero as $l \rightarrow \infty$. So for each fixed $n \in \N$ there exists a subsequence indexed by $(l^n_k)_{k}$ such that
$$\int_0^{t\wedge {\theta_n}} \inner{\mathcal{G}u_s}{\tilde{\phi}^{l^n_k}_s } d\mathcal{W}_s \longrightarrow \int_0^{t\wedge {\theta_n}} \inner{\mathcal{G}u_s}{\phi_s } d\mathcal{W}_s$$ $\mathbbm{P}-a.s.$ in $\R$. Defining $\Omega_n$ as this set on which convergence occurs, and $\tilde{\Omega}:=\bigcup_{n \in \N}\Omega_n$ which is again of full measure, then we can take any $\omega \in \tilde{\Omega}$ and select an $n$ such that $\theta_n(\omega) \geq t$, which exists from the continuity of $\int_0^{\cdot}\norm{u_s}_1^2ds$. For this fixed $n$ we may extract the subsequence $(l^n_k)_k$ and deduce the required convergence, thus concluding the proof. 

    \end{itemize}

\end{proof}

\subsection{Galerkin Scheme} \label{sub galerkin}

We fix an $\mathcal{F}_0-$measurable $u_0 \in L^\infty\left(\Omega; L^2_{\sigma}\right)$ as in Theorem \ref{existence of weak}. We shall work with a Galerkin Approximation and do so by considering (\ref{projected Ito}) in its spatially strong form, projected by $\mathcal{P}_n$ to give
\begin{equation} \label{projected Ito galerkin}
    u^n_t = u^n_0 - \int_0^t\mathcal{P}_n\mathcal{P}\mathcal{L}_{u^n_s}u^n_s\ ds - \nu\int_0^t \mathcal{P}_n A u^n_s\, ds + \frac{\nu}{2}\int_0^t\sum_{i=1}^\infty \mathcal{P}_n \mathcal{P}\mathcal{Q}_i^2u^n_s ds - \nu^{\frac{1}{2}}\int_0^t \mathcal{P}_n\mathcal{P}\mathcal{G}u^n_s d\mathcal{W}_s 
\end{equation}
where $u^n_0:=\mathcal{P}_nu_0$. We wish to show that strong solutions of (\ref{projected Ito galerkin}) exist, where the notion of solution is typical and as given in [\cite{goodair2022existence1}] Proposition 6.1 for $\mathcal{H}:= V_n:= \textnormal{span}\{ a_1, \dots, a_n\}$ (which is a Hilbert Space equipped with the $L^2(\mathscr{O};\R^N)$ or any equivalent inner product). We look to apply Lemma 3.18 of this work to deduce that for any $M > 1$ and $S > 0$, there exists a local strong solution of (\ref{projected Ito galerkin}) up until the stopping time $$\tau^{M,S}_n := S \wedge \inf \left\{s \geq 0: \sup_{r \in [0,s]}\norm{u^n_r}^2 + \int_0^s \norm{u^n_r}_1^2dr \geq M + \norm{u^n_0}^2 \right\}.$$ Without a direct application of this lemma we reach the conclusion here in the same manner, where one can consider $V:= W^{2,2}(\mathscr{O};\R^N)$, $H:= W^{1,2}(\mathscr{O};\R^N)$ and $U:= L^2(\mathscr{O};\R^N)$ with (\ref{assumpty 4}) and the equivalence of the norms on $V_n$. Having established local existence of solutions we note the uniqueness en route to showing that such a solution exists on $[0,T]$; recalling (\ref{assumpty 5}) and (\ref{assumpty 6}) we can prove the uniqueness identically to Theorem 3.29 of [\cite{goodair2022existence1}], noting again the equivalence of the norms on $V_n$. From here we deduce the existence of a unique maximal strong solution $(u^n,\Theta^n)$ as in Theorems 3.32 and 3.34, and from Lemma 3.36 we also have the relation \begin{equation} \label{relation} \mathbbm{P}\left( \{ \omega \in \Omega: \tau^{M,S}_{n}(\omega) < \Theta^n(\omega)\}\right) = 1\end{equation} for any $M>1$ and $S>0$. We must be precise in using the characterisation of $\tau^{M,S}_n$, as this was initially defined (in [\cite{goodair2022existence1}]) as a first hitting time for a globally existing truncated process, which was then stopped at this time and relabelled to ignore the truncation. The maximal solution $u^n$ is of course different to this process, but the uniqueness ensures that $\tau^{M,S}_n$ is genuinely a first hitting time for the maximal $u^n$ (as this process must be indistinguishable from the truncated one up until this stopping time). To prove the existence on $[0,T]$ we want to show that $\mathbbm{P}\left(\{\omega \in \Omega: \Theta^n(\omega) \leq T\}\right) = 0$. Moreover note that
\begin{align*}\mathbbm{P}\left(\{\omega \in \Omega: \Theta^n(\omega) \leq T\}\right) &\leq \mathbbm{P}\left(\left\{\omega \in \Omega: \sup_{M \in \N}\tau^{M,T+1}_n(\omega) \leq T\right\}\right)\\
&= \mathbbm{P}\left(\bigcap_{M \in \N}\left\{\omega \in \Omega: \tau^{M,T+1}_n(\omega) \leq T\right\}\right)\\
&= \lim_{M \rightarrow \infty}\mathbbm{P}\left(\left\{\omega \in \Omega: \tau^{M,T+1}_n(\omega) \leq T\right\}\right)
\end{align*}
from (\ref{relation}) and the fact that $\tau^{M,T+1}_n$ is increasing in $M$. From the characterisation of $\tau^{M,T+1}_n$ note that 
\begin{align*}&\left\{\omega \in \Omega: \tau^{M,T+1}_n(\omega)\leq T\right\} \\& \qquad \qquad \qquad = \left\{\omega \in \Omega: \sup_{r \in [0,T \wedge  \tau^{M,T+1}_n(\omega)]}\norm{u^n_r(\omega)}^2 + \int_0^{T \wedge  \tau^{M,T+1}_n(\omega)}\norm{u^n_r(\omega)}_1^2dr \geq M + \norm{u^n_0(\omega)}^2\right\}\end{align*} so a simple application of Chebyshev's Inequality informs us that \begin{equation}\label{information}\mathbbm{P}\left(\left\{\omega \in \Omega: \tau^{M,T+1}_n(\omega) \leq T\right\}\right) \leq \frac{1}{M}\mathbbm{E}\left[\sup_{r \in [0,T \wedge  \tau^{M,T+1}_n]}\norm{u^n_r}^2 - \norm{u^n_0}^2 + \int_0^{T \wedge  \tau^{M,T+1}_n}\norm{u^n_r}_1^2dr \right].\end{equation} 
This prompts the following result.

\begin{proposition} \label{prop for first energy}
    Let $(u^n,\Theta^n)$ be the maximal strong solution of equation (\ref{projected Ito galerkin}). There exists a constant $C$ independent of $M,n,\nu$ such that
    \begin{equation}\label{first result}\mathbbm{E}\left[\sup_{r \in [0,T \wedge  \tau^{M,T+1}_n]}\norm{u^n_r}^2 + \nu\int_0^{T \wedge  \tau^{M,T+1}_n}\norm{u^n_r}_1^2dr  \right] \leq C\left[\mathbbm{E}\left(\norm{u^n_0}^2\right) + 1\right].\end{equation}
\end{proposition}

\begin{proof}
    We can apply the It\^{o} Formula to see that for any $0 \leq r \leq T$, the identity
    \begin{align*}
        \norm{u^n_{r \wedge \tau^{M,T+1}_n}}^2 &= \norm{u^n_0}^2 - 2\int_0^{r \wedge \tau^{M,T+1}_n}\inner{\mathcal{P}_n\mathcal{P}\mathcal{L}_{u^n_s}u^n_s}{u^n_s}\ ds - 2\nu\int_0^{r \wedge \tau^{M,T+1}_n} \inner{\mathcal{P}_n A u^n_s}{u^n_s}\, ds\\ &+ \nu\int_0^{r \wedge \tau^{M,T+1}_n}\left\langle\sum_{i=1}^\infty \mathcal{P}_n \mathcal{P}\mathcal{Q}_i^2u^n_s, u^n_s\right\rangle ds +\nu\int_0^{r \wedge \tau^{M,T+1}_n}\sum_{i=1}^\infty \norm{\mathcal{P}_n\mathcal{P}\mathcal{G}_iu^n_s}^2ds\\ &- 2\nu^{\frac{1}{2}}\int_0^{r \wedge \tau^{M,T+1}_n} \inner{\mathcal{P}_n\mathcal{P}\mathcal{G}u^n_s}{u^n_s} d\mathcal{W}_s 
    \end{align*}
holds $\mathbbm{P}-a.s.$. We now immediately simplify the expression, using that $\mathcal{P}_n$ and $\mathcal{P}$ are orthogonal projections in $L^2(\mathscr{O};\R^N)$, as well as the properties (\ref{cancellationproperty'}) and (\ref{prop for moving stokes'}) to see that
    \begin{align*}
        \norm{u^n_{r \wedge \tau^{M,T+1}_n}}^2  &+ 2\nu\int_0^{r \wedge \tau^{M,T+1}_n} \norm{u^n_s}^2_1 ds = \norm{u^n_0}^2\\ &+ \nu\int_0^{r \wedge \tau^{M,T+1}_n}\sum_{i=1}^\infty\left( \inner{\mathcal{Q}_i^2u^n_s}{u^n_s} + \norm{\mathcal{P}_n\mathcal{P}\mathcal{G}_iu^n_s}^2\right) ds - 2\nu^{\frac{1}{2}}\int_0^{r \wedge \tau^{M,T+1}_n} \inner{\mathcal{G}u^n_s}{u^n_s} d\mathcal{W}_s.
    \end{align*}
As a characteristic of the orthogonal projections we can additionally say that $\norm{\mathcal{P}_n\mathcal{P}\mathcal{G}_iu^n_s}^2 \leq \norm{\mathcal{G}_iu^n_s}^2$, so with (\ref{assumpty 7}) we pass further to the bound 
   \begin{align*}
        \norm{u^n_{r \wedge \tau^{M,T+1}_n}}^2  &+ \nu\int_0^{r \wedge \tau^{M,T+1}_n} \norm{u^n_s}^2_1 ds \leq \norm{u^n_0}^2\\ &+ c\int_0^{r \wedge \tau^{M,T+1}_n} 1 + \norm{u^n_s}^2 ds - 2\nu^{\frac{1}{2}}\int_0^{r \wedge \tau^{M,T+1}_n} \inner{\mathcal{G}u^n_s}{u^n_s} d\mathcal{W}_s.
    \end{align*}
We shall now look to take expectation, introducing the notation \begin{equation}\label{notation} \hat{u}^n_{\cdot} := u^n_{\cdot}\mathbbm{1}_{\cdot \leq \tau^{M,T+1}_n}\end{equation} where we appreciate that \begin{equation} \label{stopping time boundedness}
    \sup_{r \in [0,T]}\norm{\hat{u}^n_r}^2  + \int_0^{T}\norm{\hat{u}^n_s}_1^2ds \leq M + \norm{u^n_0}^2 \leq M + \norm{u_0}^2_{L^{\infty}(\Omega;L^2(\mathscr{O};\R^N))}.
\end{equation} This boundedness ensures the integrability of all terms. From here we can take the absolute value of the stochastic integral followed by the supremum over $r\in[0,T]$ (we take the supremum for each term on the left hand side individually, then sum them), then apply the Burkholder-Davis-Gundy Inequality to obtain 
   \begin{align*}
        \mathbbm{E}\left(\sup_{r\in[0,T]}\norm{\hat{u}^{n}_{r}}^2\right) + \nu\mathbbm{E}\int_0^{T} \norm{\hat{u}^{n}_s}^2_1 ds &\leq 2\mathbbm{E}\left(\norm{u^n_0}^2\right) + c\\ &+ c\int_0^{T} \mathbbm{E}\left(\norm{\hat{u}^{n}_s}^2\right) ds + c\mathbbm{E}\left(\int_0^{T}\sum_{i=1}^\infty \inner{\mathcal{G}_i\hat{u}^{n}_s}{\hat{u}^{n}_s}^2 ds\right)^\frac{1}{2}
    \end{align*}
where $c$ continues to represent a generic constant now dependent on $T$ and the constant from the Burkholder-Davis-Gundy Inequality. We recall again that $\nu$ is assumed to be less than $1$. Focusing further on the stochastic term, we use (\ref{assumpty8}) to see
   \begin{align}\nonumber
        c\mathbbm{E}\left(\int_0^{T}\sum_{i=1}^\infty \inner{\mathcal{G}_i\hat{u}^{n}_s}{\hat{u}^{n}_s}^2 ds\right)^\frac{1}{2} &\leq  c\mathbbm{E}\left(\int_0^{T} 1 + \norm{\hat{u}^{n}_s}^4ds\right)^\frac{1}{2}\\\nonumber
        &\leq c + c\mathbbm{E}\left(\int_0^{T} \norm{\hat{u}^{n}_s}^4ds\right)^\frac{1}{2}\\\nonumber
        &\leq c + c\mathbbm{E}\left(\sup_{r\in[0,T]}\norm{\hat{u}^{n}_r}^2\int_0^{T} \norm{\hat{u}^{n}_s}^2ds\right)^\frac{1}{2}\\
        &\leq c + \frac{1}{2}\mathbbm{E}\left(\sup_{r\in[0,T]}\norm{\hat{u}^{n}_r}^2\right) + c\mathbbm{E}\int_0^{T} \norm{\hat{u}^{n}_s}^2ds \label{the same process}
    \end{align}
having applied Young's Inequality. Substituting this into our running inquality achieves 
 \begin{align} \label{subbed}
        \frac{1}{2} \mathbbm{E}\left(\sup_{r\in[0,T]}\norm{\hat{u}^{n}_{r}}^2\right) + \nu\mathbbm{E}\int_0^{T} \norm{\hat{u}^{n}_s}^2_1 ds \leq 2\mathbbm{E}\left(\norm{u^n_0}^2\right) + c + c\int_0^{T} \mathbbm{E}\left(\norm{\hat{u}^{n}_s}^2\right) ds
    \end{align}
and in particular
 \begin{align*}
    \mathbbm{E}\left(\sup_{r\in[0,T]}\norm{\hat{u}^{n}_{r}}^2\right) \leq 4\mathbbm{E}\left(\norm{u^n_0}^2\right) + c + c\int_0^{T} \mathbbm{E}\left(\sup_{r\in[0,s]}\norm{\hat{u}^{n}_r}^2\right) ds
    \end{align*}
to which an application of the standard Gr\"{o}nwall Inequality yields
\begin{equation}\label{yields}\mathbbm{E}\left(\sup_{r\in[0,T]}\norm{\hat{u}^{n}_{r}}^2\right) \leq c\left[\mathbbm{E}\left(\norm{u^n_0}^2\right) + 1\right].\end{equation}
    Of course from (\ref{subbed}) we also see that
    $$\nu\mathbbm{E}\int_0^{T} \norm{\hat{u}^{n}_s}^2_1 ds \leq \mathbbm{E}\left(\norm{u^n_0}^2\right) + c + c\int_0^{T} \mathbbm{E}\left(\sup_{r \in [0,s]}\norm{\hat{u}^{n}_r}^2\right) ds $$
to which we substitute in (\ref{yields}) to the right hand side, then summing the resultant inequality with (\ref{yields}) gives $$ \mathbbm{E}\left(\sup_{r\in[0,T]}\norm{\hat{u}^{n}_{r}}^2\right) + \nu\mathbbm{E}\int_0^{T} \norm{\hat{u}^{n}_s}^2_1 ds  \leq c\left[\mathbbm{E}\left(\norm{u^n_0}^2\right) + 1\right]$$ which can simply be rewritten as (\ref{first result}), concluding the proof.
\end{proof}

The expectation in (\ref{information}) is thus finite, so taking the limit $M \rightarrow \infty$ achieves that\\  $\mathbbm{P}\left(\{\omega \in \Omega: \Theta^n(\omega) \leq T\}\right)=0$. It is further evident from our calculations that for any $\tau$ such that $(u^n,\tau)$ is a local strong solution of the equation (\ref{projected Ito galerkin}), the inequality (\ref{first result}) holds (that is, with $\tau$ replacing $\tau^{M,T+1}_n$) where $C$ is independent of the choice of $\tau$. Therefore we can choose a $\mathbbm{P}-a.s.$ increasing sequence of stopping times which approach $\Theta^n$ by definition of the maximal time, and applying the Monotone Convergence Theorem just as we did in the proof yields that  \begin{equation}\label{second result}\mathbbm{E}\left[\sup_{r \in [0,T]}\norm{u^n_r}^2 + \nu\int_0^{T}\norm{u^n_r}_1^2dr  \right] \leq C\left[\mathbbm{E}\left(\norm{u^n_0}^2\right) + 1\right].\end{equation}
Moreover for any given $t \in [0,T]$ and $\omega \in \Omega$, we can choose a $\tau(\omega)>t$ such that $u^n$ does indeed satisfy the identity (\ref{projected Ito galerkin}) without localisation and $u\in C\left([0,t];V_n\right)$. Of course we can bound $\norm{u^n_0}^2 \leq \norm{u_0}^2$ and $\mathbbm{E}\left(\norm{u^n_0}^2 \right) \leq \norm{u_0}_{L^{\infty}(\Omega;L^2(\mathscr{O};\R^N))}^2$ which is finite independent of $n$ and can be substituted in to (\ref{first result}) and (\ref{second result}). Combining this with (\ref{information}) achieves that \begin{equation}\label{achievement}\lim_{M \rightarrow \infty}\sup_{n \in \N}\mathbbm{P}\left(\left\{\omega \in \Omega: \tau^{M,T+1}_n(\omega) \leq T\right\}\right) =0.\end{equation}

\subsection{Tightness} \label{sub tight}

We now look to deduce the existence of a process taken as the limit of $u^n$ in some sense, which is done through a tightness argument. We pursue this with Lemma \ref{Lemma 5.2} in the Appendix precisely as in [\cite{rockner2022well}], with the spaces $\mathcal{H}_1:= W^{1,2}_{\sigma}$, $\mathcal{H}_2:= L^2_{\sigma}$. Having already demonstrated (\ref{first condition}) we now justify (\ref{second condition}), fixing a $T>0$ and introducing new notation
\begin{equation} \label{new notation}
    u^{n,M}_{\cdot}:= u^{n}_{\cdot \wedge \tau^{M,T+1}_n}.
\end{equation}
Observe that for any $\varepsilon, \delta > 0$ and removing the explicit reference to $\omega \in \Omega$ for brevity,
\begin{align*}
    &\mathbbm{P}\left(\left\{\int_0^{T-\delta}\norm{u^n_{s + \delta} - u^n_s}^2ds > \varepsilon\right\} \right) \\&= \mathbbm{P}\left(\left\{\int_0^{T-\delta}\norm{u^n_{s + \delta} - u^n_s}^2ds > \varepsilon\right\} \cap \left[\left\{ \tau^{M,T+1}_n > T \right\} \cup \left\{ \tau^{M,T+1}_n \leq T \right\} \right]\right)\\
    &\leq  \mathbbm{P}\left(\left\{\int_0^{T-\delta}\norm{u^n_{s + \delta} - u^n_s}^2ds > \varepsilon\right\} \cap \left\{ \tau^{M,T+1}_n > T \right\}\right) + \mathbbm{P}\left(\left\{ \tau^{M,T+1}_n \leq T \right\} \right)\\
    &=  \mathbbm{P}\left(\left\{\int_0^{T-\delta}\norm{u^{n,M}_{s + \delta} - u^{n,M}_s}^2ds > \varepsilon\right\} \cap \left\{ \tau^{M,T+1}_n > T \right\}\right) + \mathbbm{P}\left(\left\{ \tau^{M,T+1}_n \leq T \right\} \right)\\
    &\leq  \mathbbm{P}\left(\left\{\int_0^{T-\delta}\norm{u^{n,M}_{s + \delta} - u^{n,M}_s}^2ds > \varepsilon\right\}\right) + \mathbbm{P}\left(\left\{ \tau^{M,T+1}_n \leq T \right\} \right)
\end{align*}
returning to the notation introduced in (\ref{notation}). Therefore
\begin{align*}
    &\lim_{\delta \rightarrow 0^+}\sup_{n \in \N}\mathbbm{P}\left(\left\{\int_0^{T-\delta}\norm{u^n_{s + \delta} - u^n_s}^2ds > \varepsilon\right\} \right)\\ &\leq  \lim_{\delta \rightarrow 0^+}\sup_{n \in \N} \left[ \mathbbm{P}\left(\left\{\int_0^{T-\delta}\norm{u^{n,M}_{s + \delta} - u^{n,M}_s}^2ds > \varepsilon\right\}\right) + \mathbbm{P}\left(\left\{ \tau^{M,T+1}_n \leq T \right\} \right)\right]
\end{align*}
holds for every $M \in \N$, so indeed
\begin{align*}
    &\lim_{\delta \rightarrow 0^+}\sup_{n \in \N}\mathbbm{P}\left(\left\{\int_0^{T-\delta}\norm{u^n_{s + \delta} - u^n_s}^2ds > \varepsilon\right\} \right)\\ &\leq  \lim_{M \rightarrow \infty}\lim_{\delta \rightarrow 0^+}\sup_{n \in \N} \left[ \mathbbm{P}\left(\left\{\int_0^{T-\delta}\norm{u^{n,M}_{s + \delta} - u^{n,M}_s}^2ds > \varepsilon\right\}\right) + \mathbbm{P}\left(\left\{ \tau^{M,T+1}_n \leq T \right\} \right)\right]\\
    &\leq \lim_{M \rightarrow \infty}\lim_{\delta \rightarrow 0^+}\sup_{n \in \N} \left[ \mathbbm{P}\left(\left\{\int_0^{T-\delta}\norm{u^{n,M}_{s + \delta} - u^{n,M}_s}^2ds > \varepsilon\right\}\right)\right] + \lim_{M \rightarrow \infty}\sup_{n \in \N}\left[\mathbbm{P}\left(\left\{ \tau^{M,T+1}_n \leq T \right\} \right)\right]\\
    &= \lim_{M \rightarrow \infty}\lim_{\delta \rightarrow 0^+}\sup_{n \in \N} \left[ \mathbbm{P}\left(\left\{\int_0^{T-\delta}\norm{u^{n,M}_{s + \delta} - u^{n,M}_s}^2ds > \varepsilon\right\}\right)\right]\\
    &\leq \lim_{M \rightarrow \infty}\lim_{\delta \rightarrow 0^+}\sup_{n \in \N}\left[ \frac{1}{\varepsilon}\mathbbm{E}\int_0^{T-\delta}\norm{u^{n,M}_{s + \delta} - u^{n,M}_s}^2ds\right]
\end{align*}
owing to (\ref{achievement}). The required condition (\ref{second condition}) is thus shown with the following proposition.

\begin{proposition} \label{prop for tightness}
    For any $M \in \N$,
    \begin{equation} \label{required condition}
    \lim_{\delta \rightarrow 0^+}\sup_{n \in \N}\mathbbm{E}\int_0^{T-\delta}\norm{u^{n,M}_{s + \delta} - u^{n,M}_s}^2ds= 0.
\end{equation} 
Therefore the sequence of the laws of $(u^n)$ is tight in the space of probability measures over $L^2\left([0,T];L^2_{\sigma}\right)$.
\end{proposition}

\begin{proof}
    We shall again use the notation $\hat{u}^n$ established in (\ref{notation}), noting that the dependence on $M$ is implicit. Following this our generic constant $c$ may also depend on $M$ in this proof as this remains fixed for the duration. Similarly to the proof of Proposition \ref{prop for first energy}, observe that for any $s\in[0,T]$,
    \begin{align*}u^{n,M}_{s+\delta} = u^n_0 &- \int_0^{s+\delta}\mathcal{P}_n\mathcal{P}\mathcal{L}_{\hat{u}^n_r}\hat{u}^n_r\ dr - \nu\int_0^{s+\delta} \mathcal{P}_n A \hat{u}^n_r\, dr\\ &+ \frac{\nu}{2}\int_0^{s+\delta}\sum_{i=1}^\infty \mathcal{P}_n \mathcal{P}\mathcal{Q}_i^2\hat{u}^n_r dr - \nu^{\frac{1}{2}}\int_0^{s+\delta} \mathbbm{1}_{r \leq \tau^{M,T+1}_n}\mathcal{P}_n\mathcal{P}\mathcal{G}\hat{u}^n_r d\mathcal{W}_r
    \end{align*}
where an indicator function has had to be included in the stochastic integral as it may not be the case that $\mathcal{G}(0) = 0$. Therefore
   \begin{align} \nonumber u^{n,M}_{s+\delta} - u^{n,M}_{s} = &- \int_s^{s+\delta}\mathcal{P}_n\mathcal{P}\mathcal{L}_{\hat{u}^n_r}\hat{u}^n_r\ dr - \nu\int_s^{s+\delta} \mathcal{P}_n A \hat{u}^n_r\, dr\\ &+ \frac{\nu}{2}\int_s^{s+\delta}\sum_{i=1}^\infty \mathcal{P}_n \mathcal{P}\mathcal{Q}_i^2\hat{u}^n_r dr - \nu^{\frac{1}{2}}\int_s^{s+\delta} \mathbbm{1}_{r \leq \tau^{M,T+1}_n}\mathcal{P}_n\mathcal{P}\mathcal{G}\hat{u}^n_r d\mathcal{W}_r \label{difference identity}
    \end{align}
which for any fixed $s$ is just an evolution equation in parameter $\delta$, so we can apply the It\^{o} Formula (e.g. Proposition \ref{rockner prop}) to deduce that
\begin{align*}
    \norm{u^{n,M}_{s+\delta} - u^{n,M}_{s}}^2 = &- 2\int_s^{s+\delta}\inner{\mathcal{P}_n\mathcal{P}\mathcal{L}_{\hat{u}^n_r}\hat{u}^n_r}{\hat{u}^n_r-\hat{u}^n_s} dr - 2\nu\int_s^{s+\delta} \inner{\mathcal{P}_n A \hat{u}^n_r}{\hat{u}^n_r-\hat{u}^n_s} dr\\ &+ \nu\int_s^{s+\delta}\left \langle \sum_{i=1}^\infty \mathcal{P}_n \mathcal{P}\mathcal{Q}_i^2\hat{u}^n_r, \hat{u}^n_r-\hat{u}^n_s\right\rangle dr +\nu\int_s^{s+\delta}\sum_{i=1}^{\infty}\norm{\mathcal{P}_n\mathcal{P}\mathcal{G}_i\hat{u}^n_r}^2dr\\ &- 2\nu^{\frac{1}{2}}\int_s^{s+\delta} \inner{\mathcal{P}_n\mathcal{P}\mathcal{G}\hat{u}^n_r}{\hat{u}^n_r-\hat{u}^n_s} d\mathcal{W}_r
\end{align*}
and taking expectation with some simplification of the projections,
\begin{align} \nonumber
    \mathbb{E}\norm{u^{n,M}_{s+\delta} - u^{n,M}_{s}}^2 \leq &- 2\mathbb{E}\int_s^{s+\delta}\inner{\mathcal{L}_{\hat{u}^n_r}\hat{u}^n_r}{\hat{u}^n_r-\hat{u}^n_s} dr - 2\nu\mathbb{E}\int_s^{s+\delta} \inner{ A \hat{u}^n_r}{\hat{u}^n_r-\hat{u}^n_s} dr\\ &+ \nu\mathbb{E}\int_s^{s+\delta}\left \langle \sum_{i=1}^\infty \mathcal{Q}_i^2\hat{u}^n_r, \hat{u}^n_r-\hat{u}^n_s\right\rangle dr +\nu\mathbb{E}\int_s^{s+\delta}\sum_{i=1}^{\infty}\norm{\mathcal{G}_i\hat{u}^n_r}^2dr \nonumber
\end{align}
where the stochastic integral drops out. Note that having now established (\ref{second result}) we can take the expectation without having to stop the processes as we did with $\theta_R$ in Proposition \ref{prop for first energy}; moreover the stochastic integral is a genuine square integrable martingale. We split the inner products from the right hand side into the term with $\hat{u}^n_r$ and $\hat{u}^n_s$. We again have that $$\inner{\mathcal{L}_{\hat{u}^n_r}\hat{u}^n_r}{\hat{u}^n_r} = 0$$ and $$\inner{ A \hat{u}^n_r}{\hat{u}^n_r} = \norm{\hat{u}^n_r}_1^2$$ so these terms can also be dropped from the inequality, combining with an application of (\ref{assumpty 7}) in the same way to reduce to
\begin{align} \nonumber
    \mathbb{E}\norm{u^{n,M}_{s+\delta} - u^{n,M}_{s}}^2 \leq  2&\mathbb{E}\int_s^{s+\delta}\inner{\mathcal{L}_{\hat{u}^n_r}\hat{u}^n_r}{\hat{u}^n_s} dr + 2\nu\mathbb{E}\int_s^{s+\delta} \inner{ A \hat{u}^n_r}{\hat{u}^n_s} dr\\ &- \nu\mathbb{E}\int_s^{s+\delta}\left \langle \sum_{i=1}^\infty \mathcal{Q}_i^2\hat{u}^n_r, \hat{u}^n_s\right\rangle dr +c\mathbb{E}\int_s^{s+\delta} 1 + \norm{\hat{u}^n_r}^2 dr. \label{reduction}
\end{align}
recalling that $\nu < 1$. The remaining terms are treated individually; just as seen in the justification of Definition \ref{definitionofspatiallyweak},
\begin{align} \label{a1}
    \inner{\mathcal{L}_{\hat{u}^n_r}\hat{u}^n_r}{\hat{u}^n_s}  \leq c\norm{\hat{u}^n_r}_1^2\norm{\hat{u}^n_s}_1
\end{align}
whilst of course \begin{equation}\label{a2}\inner{ A \hat{u}^n_r}{\hat{u}^n_s} = \inner{\hat{u}^n_r}{\hat{u}^n_s}_1 \leq \norm{\hat{u}^n_r}_1\norm{\hat{u}^n_s}_1\end{equation}
and by applying (\ref{assumpty 1}) and (\ref{assumpty 3}),
\begin{equation}\label{a3}\left \langle \sum_{i=1}^\infty \mathcal{Q}_i^2\hat{u}^n_r, \hat{u}^n_s\right\rangle = \sum_{i=1}^\infty \left \langle \mathcal{Q}_i\hat{u}^n_r, \mathcal{Q}_i^*\hat{u}^n_s\right\rangle \leq c\left(1 + \norm{\hat{u}^n_r}_1\right)\norm{\hat{u}^n_s}_1.\end{equation}
We use the inequality $\norm{\hat{u}^n_r}_1 \leq 1 + \norm{\hat{u}^n_r}_1^2$
in conjunction with (\ref{a1}), (\ref{a2}) and (\ref{a3}) substituted into (\ref{reduction}) to obtain 
\begin{align*}
        \mathbb{E}\norm{u^{n,M}_{s+\delta} - u^{n,M}_{s}}^2 \leq  &c\mathbb{E}\int_s^{s+\delta} \left(1 + \norm{\hat{u}^n_r}_1^2\right)\norm{\hat{u}^n_s}_1 dr +c\mathbb{E}\int_s^{s+\delta} 1 + \norm{\hat{u}^n_r}^2 dr.
\end{align*}
From (\ref{stopping time boundedness}) we can create a bound on the final term with \begin{equation} \label{its the final term}
    c\mathbb{E}\int_s^{s+\delta} 1 + \norm{\hat{u}^n_r}^2 dr \leq c\delta
\end{equation}
which, revisiting what we actually want to show in (\ref{required condition}),
\begin{align*}
        \lim_{\delta \rightarrow 0^+}\sup_{n \in \N}\mathbb{E}\int_0^{T-\delta}\norm{u^{n,M}_{s+\delta} - u^{n,M}_{s}}^2ds &= \lim_{\delta \rightarrow 0^+}\sup_{n \in \N}\int_0^{T-\delta}\mathbb{E}\norm{u^{n,M}_{s+\delta} - u^{n,M}_{s}}^2ds\\ &\leq  \lim_{\delta \rightarrow 0^+}\sup_{n \in \N}\int_0^{T-\delta}\left[c\mathbb{E}\int_s^{s+\delta} \left(1 + \norm{\hat{u}^n_r}_1^2\right)\norm{\hat{u}^n_s}_1 dr +c\delta \right]ds\\
        &=  c\lim_{\delta \rightarrow 0^+}\sup_{n \in \N}\mathbb{E}\int_0^{T-\delta}\int_s^{s+\delta} \left(1 + \norm{\hat{u}^n_r}_1^2\right)\norm{\hat{u}^n_s}_1 drds\\
        &=  c\lim_{\delta \rightarrow 0^+}\sup_{n \in \N}\mathbb{E}\int_0^{T}\int_{0 \vee (r - \delta)}^{r \wedge T - \delta} \left(1 + \norm{\hat{u}^n_r}_1^2\right)\norm{\hat{u}^n_s}_1 dsdr  
\end{align*}
with use of the Fubini-Tonelli Theorem and considering the iterated integral as an integral over the product space. Note that for each fixed $r$,
\begin{align*}\int_{0 \vee (r - \delta)}^{r \wedge T - \delta} \left(1 + \norm{\hat{u}^n_r}_1^2\right)\norm{\hat{u}^n_s}_1 ds &= \left(1 + \norm{\hat{u}^n_r}_1^2\right)\int_{0 \vee (r - \delta)}^{r \wedge T - \delta} \norm{\hat{u}^n_s}_1 ds\\ &\leq \left(1 + \norm{\hat{u}^n_r}_1^2\right)\left[\left(\int_{0 \vee (r - \delta)}^{r \wedge T - \delta} 1 ds\right)^{\frac{1}{2}}\left(\int_{0 \vee (r - \delta)}^{r \wedge T - \delta} \norm{\hat{u}^n_s}^2_1 ds\right)^{\frac{1}{2}} \right]\\
&\leq \delta^{\frac{1}{2}}\left(1 + \norm{\hat{u}^n_r}_1^2\right)\left(\int_{0 \vee (r - \delta)}^{r \wedge T - \delta} \norm{\hat{u}^n_s}^2_1 ds\right)^{\frac{1}{2}}\\
&\leq c\delta^{\frac{1}{2}}\left(1 + \norm{\hat{u}^n_r}_1^2\right)
\end{align*}
using (\ref{stopping time boundedness}) in the final line. Therefore
$$\lim_{\delta \rightarrow 0^+}\sup_{n \in \N}\mathbb{E}\int_0^{T-\delta}\norm{u^{n,M}_{s+\delta} - u^{n,M}_{s}}^2ds  \leq  c\lim_{\delta \rightarrow 0^+}\sup_{n \in \N}\mathbb{E}\int_0^{T}c\delta^{\frac{1}{2}}\left(1 + \norm{\hat{u}^n_r}_1^2\right)dr \leq \lim_{\delta \rightarrow 0^+}c\delta^{\frac{1}{2}} = 0 $$
which concludes the proof. 
\end{proof}

To achieve a characterisation of the limit process at each time, we will need to show tightness in $\mathcal{D}\left([0,T];\left(W^{1,2}_{\sigma}\right)^*\right)$ where $\left(W^{1,2}_{\sigma}\right)^*$ is the topological dual of $W^{1,2}_{\sigma}$, forming a Gelfand Triple $$W^{1,2}_{\sigma} \xhookrightarrow{} L^2_{\sigma} \xhookrightarrow{} \left(W^{1,2}_{\sigma}\right)^*.$$
The idea is to apply Lemma \ref{lemma for D tight}, of course for $\sy^n:=u^n$, $\mathcal{Y} = W^{1,2}_{\sigma}$ and $\mathcal{H} = L^2_{\sigma}$. The condition (\ref{first condition primed}) has already been shown from the stronger (\ref{second result}) so to apply the Lemma we only need to verify (\ref{second condition primed}). This is reminiscent of the condition (\ref{second condition}) just verified, so just as we saw for Proposition \ref{prop for tightness} it is sufficient to verify the following.

\begin{proposition} \label{prop for tightness two}
    For any sequence of stopping times $(\gamma_n)$ with $\gamma_n: \Omega \rightarrow [0,T]$, and any  $M \in \N$, $\phi \in W^{1,2}_{\sigma}$,
    \begin{equation} \label{required condition primed}
    \lim_{\delta \rightarrow 0^+}\sup_{n \in \N}\mathbbm{E}\left(\left\vert\left\langle u^{n,M}_{(\gamma_n + \delta) \wedge T} - u^{n,M}_{\gamma_n} , \phi \right\rangle \right\vert\right)= 0.
\end{equation} 
Therefore the sequence of the laws of $(u^n)$ is tight in the space of probability measures over $\mathcal{D}\left([0,T];\left(W^{1,2}_{\sigma}\right)\right)^*$.
\end{proposition}

Before proving this result, we state and prove an intermediary lemma.

\begin{lemma} \label{lemma application of gag}
    For every $\psi,\phi \in W^{1,2}_{\sigma}$,    \begin{align*}
        \abs{\inner{\mathcal{L}_{\psi}\psi}{\phi}} \leq c\norm{\psi}^{1/2}\norm{\psi}_1^{3/2}\norm{\phi}_1.
    \end{align*}
\end{lemma}

\begin{proof}
    Using (\ref{wloglhs}), two applications of H\"{o}lder's Inequality and the Sobolev Embedding $W^{1,2}(\mathscr{O};\R^N) \xhookrightarrow{} L^6(\mathscr{O};\R^N)$ (as seen in the justification of Defintion \ref{definitionofspatiallyweak}),
    \begin{align} \label{hello}
        \abs{\inner{\mathcal{L}_{\psi}\psi}{\phi}} = \abs{\inner{\psi}{\mathcal{L}_{\psi}\phi}} \leq \norm{\psi}_{L^6}\norm{\mathcal{L}_{\psi}\phi}_{L^{6/5}} \leq c\sum_{k=1}^N\norm{\psi}_{L^6}\norm{\psi}_{L^3}\norm{\partial_k \phi} \leq c\norm{\psi}_1\norm{\psi}_{L^3}\norm{\phi}_1
    \end{align}
    to which we shift our attentions to the control on $\norm{\psi}_{L^3}$. We apply the Gagliardo-Nirenberg Inequality (Theorem \ref{gagliardonirenberginequality}) to deduce that
    $$ \norm{\psi}_{L^3} \leq c\norm{\psi}^{1/2}\norm{\psi}_1^{1/2} $$
    which is achieved directly in $N=3$ with the values $p=3$, $\alpha = 1/2$, $q = 2$ and $m = 1$, whilst for $N = 2$ we take $p=4$ and use that $\norm{\psi}_{L^3} \leq c\norm{\psi}_{L^4}$. Plugging this into (\ref{hello}) yields the result.
\end{proof}

\begin{proof}[Proof of Proposition \ref{prop for tightness two}:]
    Recalling (\ref{difference identity}), substituting in $\gamma_n$ for $s$ and stopping the process at $T$, we see that
   \begin{align*}u^{n,M}_{(\gamma_n+\delta) \wedge T} - u^{n,M}_{\gamma_n} = &- \int_{\gamma_n}^{(\gamma_n+\delta)\wedge T}\mathcal{P}_n\mathcal{P}\mathcal{L}_{\hat{u}^n_r}\hat{u}^n_r\ dr - \nu\int_{\gamma_n}^{(\gamma_n+\delta)\wedge T} \mathcal{P}_n A \hat{u}^n_r\, dr\\ &+ \frac{\nu}{2}\int_{\gamma_n}^{(\gamma_n +\delta) \wedge T}\sum_{i=1}^\infty \mathcal{P}_n \mathcal{P}\mathcal{Q}_i^2\hat{u}^n_r dr - \nu^{\frac{1}{2}}\int_{\gamma_n}^{(\gamma_n +\delta) \wedge T} \mathbbm{1}_{r \leq \tau^{M,T+1}_n}\mathcal{P}_n\mathcal{P}\mathcal{G}\hat{u}^n_r d\mathcal{W}_r 
    \end{align*}
holds $\mathbbm{P}-a.s.$, to which we take the inner product with arbitrary $\phi \in W^{1,2}_{\sigma}$ and absolute value to see that 
 \begin{align*} \left\vert \left\langle u^{n,M}_{(\gamma_n+\delta) \wedge T} - u^{n,M}_{\gamma_n} , \phi \right\rangle \right\vert &\leq  \int_{\gamma_n}^{(\gamma_n+\delta)\wedge T}\left\vert \left\langle \mathcal{L}_{\hat{u}^n_r}\hat{u}^n_r, \mathcal{P}_n\phi \right\rangle \right\vert dr + \nu\int_{\gamma_n}^{(\gamma_n+\delta)\wedge T}\left\vert \left\langle  A \hat{u}^n_r , \mathcal{P}_n\phi \right\rangle \right\vert dr\\ &+ \frac{\nu}{2}\int_{\gamma_n}^{(\gamma_n +\delta) \wedge T}\left\vert \left\langle\sum_{i=1}^\infty \mathcal{Q}_i^2\hat{u}^n_r , \mathcal{P}_n\phi \right\rangle \right\vert dr\\ &+ \nu^{\frac{1}{2}}\left\vert\int_{\gamma_n}^{(\gamma_n +\delta) \wedge T} \mathbbm{1}_{r \leq \tau^{M,T+1}_n}\left\langle\mathcal{G}\hat{u}^n_r, \mathcal{P}_n \phi \right\rangle d\mathcal{W}_r  \right\vert
    \end{align*}
having also carried over the projections $\mathcal{P}_n$, $\mathcal{P}$. With Lemma \ref{lemma application of gag} and the ideas of (\ref{a2}) and (\ref{a3}), we reduce the above to
 \begin{align*} \left\vert \left\langle u^{n,M}_{(\gamma_n+\delta) \wedge T} - u^{n,M}_{\gamma_n} , \phi \right\rangle \right\vert &\leq  c\int_{\gamma_n}^{(\gamma_n+\delta)\wedge T} \norm{\hat{u}^n_r}^{1/2}\norm{\hat{u}^n_r}_1^{3/2}\norm{\phi}_1dr\\  &+ c\int_{\gamma_n}^{(\gamma_n+\delta)\wedge T}\norm{\hat{u}^n_r}_1\norm{\phi}_1 dr + c\int_{\gamma_n}^{(\gamma_n +\delta) \wedge T}\left(1 + \norm{\hat{u}^n_r}_1\right)\norm{\phi}_1 dr\\ &+ \left\vert\int_{\gamma_n}^{(\gamma_n +\delta) \wedge T} \mathbbm{1}_{r \leq \tau^{M,T+1}_n}\left\langle\mathcal{G}\hat{u}^n_r, \mathcal{P}_n \phi \right\rangle d\mathcal{W}_r  \right\vert
    \end{align*}
immediately using that $\norm{\mathcal{P}_n\phi}_1 \leq \norm{\phi}_1$, and again using that $\nu < 1$. Before addressing the stochastic integral we clean this up with (\ref{stopping time boundedness}), allowing $c$ to further depend on the again fixed $M$ and $\phi$, achieving that 
 \begin{align*} \left\vert \left\langle u^{n,M}_{(\gamma_n+\delta) \wedge T} - u^{n,M}_{\gamma_n} , \phi \right\rangle \right\vert &\leq  c\int_{\gamma_n}^{(\gamma_n+\delta)\wedge T} \norm{\hat{u}^n_r}_1^{3/2}dr\\  &+ c\int_{\gamma_n}^{(\gamma_n+\delta)\wedge T}\norm{\hat{u}^n_r}_1 dr + c\int_{\gamma_n}^{(\gamma_n +\delta) \wedge T}1 + \norm{\hat{u}^n_r}_1 dr\\ &+ \left\vert\int_{\gamma_n}^{(\gamma_n +\delta) \wedge T} \mathbbm{1}_{r \leq \tau^{M,T+1}_n}\left\langle\mathcal{G}\hat{u}^n_r, \mathcal{P}_n \phi \right\rangle d\mathcal{W}_r  \right\vert
    \end{align*}
which is further rewritten
 \begin{align} \nonumber \left\vert \left\langle u^{n,M}_{(\gamma_n+\delta) \wedge T} - u^{n,M}_{\gamma_n} , \phi \right\rangle \right\vert &\leq  c\int_{\gamma_n}^{(\gamma_n+\delta)\wedge T} \norm{\hat{u}^n_r}_1^{3/2} + 1\, dr\\  &+ \left\vert\int_{\gamma_n}^{(\gamma_n +\delta) \wedge T} \mathbbm{1}_{r \leq \tau^{M,T+1}_n}\left\langle\mathcal{G}\hat{u}^n_r, \mathcal{P}_n \phi \right\rangle d\mathcal{W}_r  \right\vert \label{rewrite}
    \end{align}
using that $\norm{\hat{u}^n_r}_1 \leq 1 + \norm{\hat{u}^n_r}_1^{3/2}$. We shall now take the expectation and apply the Burkholder-Davis-Gundy Inequality to the stochastic integral, informing us that
\begin{align*}
    \mathbbm{E}\left(\left\vert\int_{\gamma_n}^{(\gamma_n +\delta) \wedge T} \mathbbm{1}_{r \leq \tau^{M,T+1}_n}\left\langle\mathcal{G}\hat{u}^n_r, \mathcal{P}_n \phi \right\rangle d\mathcal{W}_r  \right\vert\right) &\leq c \mathbbm{E}\left(\int_{\gamma_n}^{(\gamma_n +\delta) \wedge T} \mathbbm{1}_{r \leq \tau^{M,T+1}_n}\sum_{i=1}^\infty \left\langle\mathcal{G}_i\hat{u}^n_r, \mathcal{P}_n \phi \right\rangle ^2 dr\right)^\frac{1}{2}\\
    &\leq c \mathbbm{E}\left(\int_{\gamma_n}^{(\gamma_n +\delta) \wedge T} 1 dr\right)^\frac{1}{2}\\
    &\leq c\delta^{\frac{1}{2}}
\end{align*}
via applying (\ref{assumpty9}) and absorbing this bound into $c$ as above. Returning to (\ref{rewrite}) and reducing the constant integral as just seen, then
$$ \mathbbm{E}\left(\left\vert \left\langle u^{n,M}_{(\gamma_n+\delta) \wedge T} - u^{n,M}_{\gamma_n} , \phi \right\rangle \right\vert\right) \leq  c\mathbbm{E}\left(\int_{\gamma_n}^{(\gamma_n+\delta)\wedge T} \norm{\hat{u}^n_r}_1^{3/2} dr\right) + c\delta^{\frac{1}{2}}\left(1 + \delta^{\frac{1}{2}}\right).$$
With an application of H\"{o}lder's Inequality, observe that 
    $$\int_{\gamma_n}^{(\gamma_n+\delta)\wedge T} \norm{\hat{u}^n_r}_1^{3/2} dr \leq \left(\int_{\gamma_n}^{(\gamma_n+\delta)\wedge T} 1\right)^{\frac{1}{4}} \left(\int_{\gamma_n}^{(\gamma_n+\delta)\wedge T} \norm{\hat{u}^n_r}_1^{2} dr\right)^{\frac{3}{4}} \leq c\delta^{\frac{1}{4}}$$
    using (\ref{stopping time boundedness}) once more. Simply taking the supremum over $n$ and limit as $\delta \rightarrow 0^+$ gives the result. 
\end{proof}

\subsection{Existence of Solutions} \label{existence for bounded ic}

With tightness achieved, it is now a standard procedure to apply the Prohorov and Skorohod Representation Theorems to deduce the existence of a new probability space on which a sequence of processes with the same distribution as a subsequence of $(u^{n})$ have some almost sure convergence to a limiting process. For notational simplicity we take this subsequence and keep it simply indexed by $n$. We state that precise result in the below theorem, following [\cite{nguyen2021nonlinear}].

\begin{theorem}
There exists a filtered probability space $\left(\tilde{\Omega},\tilde{\mathcal{F}},(\tilde{\mathcal{F}}_t), \tilde{\mathbbm{P}}\right)$, a cylindrical Brownian Motion $\tilde{\mathcal{W}}$ over $\mathfrak{U}$ with respect to $\left(\tilde{\Omega},\tilde{\mathcal{F}},(\tilde{\mathcal{F}}_t), \tilde{\mathbbm{P}}\right)$, a sequence of random variables $(\tilde{u}^n_0)$, $u^n_0: \tilde{\Omega} \rightarrow L^2_{\sigma}\left(\mathscr{O};\R^N\right)$ and a $\tilde{u}_0:\tilde{\Omega} \rightarrow L^2_{\sigma}\left(\mathscr{O};\R^N\right)$, a sequence of processes $(\tilde{u}^n)$, $\tilde{u}^n:\tilde{\Omega} \times [0,T] \rightarrow W^{1,2}_{\sigma}$ is progressively measurable and a process $\tilde{u}:\tilde{\Omega} \times [0,T] \rightarrow L^2_{\sigma}$ such that:
\begin{enumerate}
    \item For each $n \in \N$, $\tilde{u}^n_0$ has the same law as $u^{n}_0$;
    \item \label{new item 2} For $\tilde{\mathbbm{P}}-a.e.$ $\omega$, $\tilde{u}^n_0(\omega) \rightarrow \tilde{u}_0(\omega)$  in $L^2_{\sigma}$, and thus $\tilde{u}_0$ has the same law as $u_0$;
    \item \label{new item 3} For each $n \in \N$ and $t\in[0,T]$, $\tilde{u}^n$ satisfies the identity
    \begin{equation} \nonumber
    \tilde{u}^n_t = \tilde{u}^n_0 - \int_0^t\mathcal{P}_n\mathcal{P}\mathcal{L}_{\tilde{u}^n_s}\tilde{u}^n_s\ ds - \nu\int_0^t \mathcal{P}_n A \tilde{u}^n_s\, ds + \frac{\nu}{2}\int_0^t\sum_{i=1}^\infty \mathcal{P}_n \mathcal{P}\mathcal{Q}_i^2\tilde{u}^n_s ds - \nu^{\frac{1}{2}}\int_0^t \mathcal{P}_n\mathcal{P}\mathcal{G}\tilde{u}^n_s d\tilde{\mathcal{W}}_s 
\end{equation}
$\tilde{\mathbbm{P}}-a.s.$ in $V_n$;
\item For $\tilde{\mathbbm{P}}-a.e$ $\omega$, $\tilde{u}^n(\omega) \rightarrow \tilde{u}(\omega)$ in $L^2\left([0,T]; L^2_{\sigma} \right)$ and $\mathcal{D}\left([0,T];\left(W^{1,2}_{\sigma}\right)^* \right)$. \label{new item 4}
\end{enumerate}

\end{theorem}

\begin{proof}
    See [\cite{nguyen2021nonlinear}] Proposition 4.9 and Theorem 4.10.  
\end{proof}

We now have our candidate martingale weak solution, and to prove that this is such a solution we need only to verify that $\tilde{u}$ is progressively measurable in $W^{1,2}_{\sigma}$, for $\tilde{\mathbbm{P}}-a.e.$ $\omega$ $\tilde{u}_{\cdot}(\omega) \in L^{\infty}\left([0,T];L^2_{\sigma}\right)\cap C_w\left([0,T];L^2_{\sigma}\right) \cap L^2\left([0,T];W^{1,2}_{\sigma}\right)$ and the identity (\ref{newid1}). In fact from item \ref{new item 3}, we can deduce that \begin{equation}\label{second result new }\tilde{\mathbbm{E}}\left[\sup_{r \in [0,T]}\norm{\tilde{u}^n_r}^2 + \nu\int_0^{T}\norm{\tilde{u}^n_r}_1^2dr  \right] \leq C\left[\mathbbm{E}\left(\norm{\tilde{u}^n_0}^2\right) + 1\right] \leq C\left[\norm{\tilde{u}_0}^2_{L^\infty(\tilde{\Omega};L^2_{\sigma})} + 1\right] < \infty\end{equation}
in the same manner as we showed (\ref{second result}), without any need for localisation. The fact that $\norm{\tilde{u}^n_0} \leq \norm{\tilde{u}_0}$ $\tilde{\mathbbm{P}}-a.s.$ and $\norm{\tilde{u}_0}^2_{L^\infty(\tilde{\Omega};L^2_{\sigma})} < \infty$ is inherited from $u^n_0$, $u_0$ of the same law in $L^2_{\sigma}$. This prompts the following results.

\begin{lemma} \label{the lemma just before}
    $\tilde{u}^n \rightarrow \tilde{u}$ in $L^2\left(\tilde{\Omega};L^2\left([0,T]; L^2_{\sigma} \right)\right)$.
\end{lemma}

\begin{proof}
    This is immediate from an application of the Dominated Convergence Theorem, using the convergence in item \ref{new item 4} and the uniform boundedness (\ref{second result new }).
\end{proof}

\begin{proposition} \label{prop for regularity of limit}
    $\tilde{u}$ is progressively measurable in $W^{1,2}_{\sigma}$ and for $\tilde{\mathbbm{P}}-a.e.$ $\omega$, $\tilde{u}_{\cdot}(\omega) \in L^{\infty}\left([0,T];L^2_{\sigma}\right)\cap L^2\left([0,T];W^{1,2}_{\sigma}\right)$.
\end{proposition}

\begin{proof}
    From (\ref{second result new }) we have that the sequence $(\tilde{u}^n)$ is uniformly bounded in\\ $L^2\left(\tilde{\Omega};L^2\left([0,T];W^{1,2}_{\sigma}\right)\right)$ and $L^2\left(\tilde{\Omega};L^\infty\left([0,T];L^2_{\sigma}\right)\right)$. Firstly then we can deduce the existence of a subsequence $(\tilde{u}^{n_k})$ which is weakly convergent in the Hilbert Space $L^2\left(\tilde{\Omega};L^2\left([0,T];W^{1,2}_{\sigma}\right)\right)$ to some $\sy$, but we may also identify $L^2\left(\tilde{\Omega};L^\infty\left([0,T];L^2_{\sigma}\right)\right)$ with the dual space of $L^2\left(\tilde{\Omega};L^1\left([0,T];L^2_{\sigma}\right)\right)$ and as such from the Banach-Alaoglu Theorem we can extract a further subsequence $(\tilde{u}^{n_l})$ which is convergent to some $\py$ in the weak* topology. These limits imply that $(\tilde{u}^{n_l})$ is convergent to both $\sy$ and $\py$ in the weak topology of $L^2\left(\tilde{\Omega};L^2\left([0,T];L^2_{\sigma}\right)\right)$, but from Lemma \ref{the lemma just before} then $(\tilde{u}^{n_l})$ converges to $\tilde{u}$ strongly (hence weakly) in this space as well. By uniqueness of limits in the weak topology then $\tilde{u} = \sy = \py$ as elements of $L^2\left(\tilde{\Omega};L^2\left([0,T];L^2_{\sigma}\right)\right)$, so they agree $\tilde{\mathbbm{P}} \times \lambda-a.s.$. Thus for $\tilde{\mathbbm{P}}-a.e.$ $\omega$, $\tilde{u}_{\cdot}(\omega) \in L^{\infty}\left([0,T];L^2_{\sigma}\right)\cap L^2\left([0,T];W^{1,2}_{\sigma}\right)$.\\

    The progressive measurability is justified similarly; for any $t\in [0,T]$, we can use the progressive measurability of $(\tilde{u}^{n_k})$ to instead deduce $\tilde{u}$ as the weak limit in $L^2\left(\tilde{\Omega} \times [0,t];W^{1,2}_{\sigma}\right)$ where $\tilde{\Omega} \times [0,t]$ is equipped with the $\tilde{\mathcal{F}}_t \times \mathcal{B}\left([0,t]\right)$ sigma-algebra. Therefore $\tilde{u}: \tilde{\Omega} \times [0,t] \rightarrow W^{1,2}_{\sigma}$ is measurable with respect to this product sigma-algebra which justifies the progressive measurability. 
\end{proof}

\begin{proposition}
    $\tilde{u}$ satisfies the identity (\ref{newid1}): that is for each $\phi \in W^{1,2}_{\sigma}$ and $t\in[0,T]$,
    \begin{align} \nonumber
     \inner{\tilde{u}_t}{\phi} = \inner{\tilde{u}_0}{\phi} - \int_0^{t}\inner{\mathcal{L}_{\tilde{u}_s}\tilde{u}_s}{\phi}ds &- \nu\int_0^{t} \inner{\tilde{u}_s}{\phi}_1 ds\\ &+ \frac{\nu}{2}\int_0^{t}\sum_{i=1}^\infty \inner{\mathcal{Q}_i\tilde{u}_s}{\mathcal{Q}_i^*\phi} ds - \nu^{\frac{1}{2}}\int_0^{t} \inner{\mathcal{G}\tilde{u}_s}{\phi} d\tilde{\mathcal{W}}_s\nonumber
\end{align}
holds $\tilde{\mathbbm{P}}-a.s.$ in $\R$. Moreover for $\tilde{\mathbbm{P}}-a.e.\omega$, $\tilde{u}_{\cdot}(\omega) \in C_w\left([0,T];L^2_{\sigma}\right)$.
\end{proposition}

\begin{proof}
    We fix a $t\in[0,T]$ and $\phi\in W^{1,2}_{\sigma}$ as in the proposition, and consider an arbitrary $\psi \in W^{2,2}_{\sigma}$. Taking the inner product with $\psi$ of both sides in item \ref{new item 3} yields the identity 
       \begin{align} \nonumber
     \inner{\tilde{u}^n_t}{\psi} = \inner{\tilde{u}^n_0}{\psi} - \int_0^{t}\inner{\mathcal{P}_n\mathcal{P}\mathcal{L}_{\tilde{u}^n_s}\tilde{u}^n_s}{\psi}ds &- \nu\int_0^{t} \inner{\mathcal{P}_nA\tilde{u}^n_s}{\psi} ds\\ &+ \frac{\nu}{2}\int_0^{t}\sum_{i=1}^\infty \inner{\mathcal{P}_n\mathcal{P}\mathcal{Q}_i^2\tilde{u}^n_s}{\psi} ds - \nu^{\frac{1}{2}}\int_0^{t} \inner{\mathcal{P}_n\mathcal{P}\mathcal{G}\tilde{u}^n_s}{\psi} d\tilde{\mathcal{W}}_s\nonumber
\end{align}
for every $n \in \N$, $\tilde{\mathbbm{P}}-a.s.$, with the idea to first show (\ref{newid1}) for $\psi$ and then use the density of $W^{2,2}_{\sigma}$ in $W^{1,2}_{\sigma}$. We now take the limit $\tilde{\mathbbm{P}}-a.s.$ in $\R$ and demonstrate the appropriate convergence of each term individually. We note that convergence in the Skorohod Topology implies convergence at each $t$ (see e.g. [\cite{billingsley2013convergence}] pp.124), so
$$ \lim_{n \rightarrow \infty}\inner{\tilde{u}^n_t}{\psi} = \lim_{n \rightarrow \infty}\inner{\tilde{u}^n_t}{\psi}_{\left(W^{1,2}_{\sigma}\right)^* \times W^{1,2}_{\sigma}} = \inner{\tilde{u}_t}{\psi}_{\left(W^{1,2}_{\sigma}\right)^* \times W^{1,2}_{\sigma}}$$
where at this stage $\tilde{u}_t$ may not belong to $L^2_{\sigma}$. We clarify this immediately, using an identical argument to Proposition \ref{prop for regularity of limit} for the sequence $(\tilde{u}^n_t)$ uniformly bounded in $L^2\left(\tilde{\Omega};L^2_{\sigma}\right)$, and deducing the convergence to $\tilde{u}_t$ in $L^2\left(\tilde{\Omega};\left(W^{1,2}_{\sigma}\right)^*\right)$. Thus we can conclude that $\tilde{u}_t\in L^2_{\sigma}$ and therefore $$\lim_{n \rightarrow \infty}\inner{\tilde{u}^n_t}{\psi} = \inner{\tilde{u}_t}{\psi}. $$ is genuinely an inner product of $L^2_{\sigma}$ valued functions. The limit for the initial condition comes out of item \ref{new item 2}, so we move on now to the nonlinear term. At this point we remark that it is only necessary to show the desired limit for a subsequence, which could be extracted $\tilde{\mathbbm{P}}-a.s.$ if we were to instead show a limit in $L^p\left(\tilde{\Omega};\R\right)$, $p=1,2$,  for each term. Firstly observe that \begin{align*}\inner{\mathcal{P}_n\mathcal{P}\mathcal{L}_{\tilde{u}^n_s}\tilde{u}^n_s}{\psi} - \inner{\mathcal{P}\mathcal{L}_{\tilde{u}_s}\tilde{u}_s}{\psi} &= \left\langle \mathcal{P}_n\left[\mathcal{P}\mathcal{L}_{\tilde{u}^n_s}\tilde{u}^n_s - \mathcal{P}\mathcal{L}_{\tilde{u}_s}\tilde{u}_s\right], \psi \right\rangle + \inner{\left(I - \mathcal{P}_n\right)\mathcal{P}\mathcal{L}_{\tilde{u}_s}\tilde{u}_s}{\psi}\\
&= \left\langle \mathcal{L}_{\tilde{u}^n_s}\tilde{u}^n_s - \mathcal{L}_{\tilde{u}_s}\tilde{u}_s, \mathcal{P}_n\psi \right\rangle + \inner{\mathcal{L}_{\tilde{u}_s}\tilde{u}_s}{\left(I - \mathcal{P}_n\right)\psi}\\
&= \left\langle \mathcal{L}_{\tilde{u}^n_s - \tilde{u}_s}\tilde{u}^n_s, \mathcal{P}_n\psi \right\rangle + \left\langle \mathcal{L}_{ \tilde{u}_s}\left(\tilde{u}^n_s - \tilde{u}_s\right), \mathcal{P}_n\psi \right\rangle + \inner{\mathcal{L}_{\tilde{u}_s}\tilde{u}_s}{\left(I - \mathcal{P}_n\right)\psi}
\end{align*}
which we inspect term by term with
\begin{align*}\left\vert \left\langle \mathcal{L}_{\tilde{u}^n_s - \tilde{u}_s}\tilde{u}^n_s, \mathcal{P}_n\psi \right\rangle \right\vert &= \left\vert \left\langle \tilde{u}^n_s, \mathcal{L}_{\tilde{u}^n_s - \tilde{u}_s}(\mathcal{P}_n\psi) \right\rangle \right\vert\\ &\leq c\sum_{k=1}^N\norm{\tilde{u}^n_s}_{L^6}\norm{\tilde{u}^n_s - \tilde{u}_s}\norm{\partial_k \mathcal{P}_n\psi}_{L^3}\\ &\leq c\sum_{k=1}^N\norm{\tilde{u}^n_s}_{1}\norm{\tilde{u}^n_s - \tilde{u}_s}\norm{\partial_k \mathcal{P}_n\psi}_1\\ &\leq c\norm{\tilde{u}^n_s}_{1}\norm{\tilde{u}^n_s - \tilde{u}_s}\norm{\psi}_2
\end{align*}
for the first term (as seen similarly in Lemma \ref{lemma application of gag}),
\begin{align*}
    \left\vert \left\langle \mathcal{L}_{ \tilde{u}_s}\left(\tilde{u}^n_s - \tilde{u}_s\right), \mathcal{P}_n\psi \right\rangle \right\vert &=  \left\vert \left\langle \tilde{u}^n_s - \tilde{u}_s ,  \mathcal{L}_{ \tilde{u}_s}(\mathcal{P}_n\psi) \right\rangle \right\vert\\ &\leq c\norm{\tilde{u}^n_s - \tilde{u}_s}\norm{\mathcal{L}_{ \tilde{u}_s}(\mathcal{P}_n\psi)}\\ &\leq c\sum_{k=1}^N\norm{\tilde{u}^n_s - \tilde{u}_s}\norm{\tilde{u}_s}_{L^4}\norm{\partial_k \mathcal{P}_n\psi}_{L^4}\\ &\leq c\norm{\tilde{u}^n_s - \tilde{u}_s}\norm{\tilde{u}_s}_{1}\norm{\psi}_2
\end{align*}
for the second, and finally
\begin{align*}
    \left\vert\inner{\mathcal{L}_{\tilde{u}_s}\tilde{u}_s}{\left(I - \mathcal{P}_n\right)\psi}\right\vert &=    \left\vert\inner{\tilde{u}_s}{\mathcal{L}_{\tilde{u}_s}\left[\left(I - \mathcal{P}_n\right)\psi\right]}\right\vert\\
    &\leq c\sum_{k=1}^N\norm{\tilde{u}_s}_{L^6}\norm{\tilde{u}_s}_{L^3}\norm{\partial_k\left[\left(I - \mathcal{P}_n\right)\psi\right]}\\
    &\leq c\norm{\tilde{u}_s}_{1}\norm{\tilde{u}_s}_{1}\norm{\left(I - \mathcal{P}_n\right)\psi}_1\\
     &\leq \frac{c}{\sqrt{\lambda_n}}\norm{\tilde{u}_s}_{1}^2\norm{\psi}_2
\end{align*}
coming from (\ref{one over n bound two}). Therefore
$$\left\vert\inner{\mathcal{P}_n\mathcal{P}\mathcal{L}_{\tilde{u}^n_s}\tilde{u}^n_s}{\psi} - \inner{\mathcal{P}\mathcal{L}_{\tilde{u}_s}\tilde{u}_s}{\psi}\right\vert \leq c\big(\norm{\tilde{u}^n_s}_{1}+ \norm{\tilde{u}_s}_{1} \big)\norm{\psi}_2\norm{\tilde{u}^n_s - \tilde{u}_s} + \frac{c}{\sqrt{\lambda_n}}\norm{\tilde{u}_s}_{1}^2\norm{\psi}_2$$
so
\begin{align*}
    &\mathbbm{E}\left\vert \int_0^t\inner{\mathcal{P}_n\mathcal{P}\mathcal{L}_{\tilde{u}^n_s}\tilde{u}^n_s}{\psi} ds - \int_0^t\inner{\mathcal{P}\mathcal{L}_{\tilde{u}_s}\tilde{u}_s}{\psi} ds \right\vert\\ & \qquad \leq c\norm{\psi}_2\mathbbm{E}\int_0^t \big(\norm{\tilde{u}^n_s}_{1}+ \norm{\tilde{u}_s}_{1} \big)\norm{\tilde{u}^n_s - \tilde{u}_s}ds + \frac{c}{\sqrt{\lambda_n}}\norm{\psi}_2\mathbbm{E}\int_0^t\norm{\tilde{u}_s}_{1}^2ds \\ &  \qquad \leq c\norm{\psi}_2\left(\mathbbm{E}\int_0^t \big(\norm{\tilde{u}^n_s}_{1}+ \norm{\tilde{u}_s}_{1} \big)^2ds\right)^{\frac{1}{2}}\left(\mathbbm{E}\int_0^t\norm{\tilde{u}^n_s - \tilde{u}_s}^2ds\right)^{\frac{1}{2}} + \frac{c}{\sqrt{\lambda_n}}\norm{\psi}_2\mathbbm{E}\int_0^t\norm{\tilde{u}_s}_{1}^2ds
\end{align*}
which approaches zero as $n \rightarrow \infty$ from the uniform boundedness (\ref{second result new }), Lemma \ref{the lemma just before} and the fact that $\lambda_n \rightarrow \infty$. The Stokes Operator term is more straightforwards, as \begin{align*}\left\vert\inner{\mathcal{P}_nA\tilde{u}^n_s}{\psi} - \inner{\tilde{u}_s}{\psi}_1\right\vert = \left\vert\inner{A\tilde{u}^n_s}{\psi} - \inner{\tilde{u}_s}{\psi}_1\right\vert = \left\vert\inner{\tilde{u}^n_s-\tilde{u}_s}{\psi}_1\right\vert \leq \norm{\tilde{u}^n_s-\tilde{u}_s}\norm{\psi}_2
\end{align*}
thus
\begin{align*}
    \mathbbm{E}\left\vert \int_0^t\inner{\mathcal{P}_nA\tilde{u}^n_s}{\psi} ds - \int_0^t\inner{\tilde{u}_s}{\psi}_1 ds \right\vert\leq \norm{\psi}_2\mathbbm{E}\int_0^t\norm{\tilde{u}^n_s - \tilde{u}_s}ds
\end{align*}
which approaches zero as $n \rightarrow \infty$, seen explicitly from an application of Cauchy-Schwartz with the constant function. The following is similar: 
\begin{align*}
    \left\vert\inner{\mathcal{P}_n\mathcal{P}\mathcal{Q}_i^2\tilde{u}^n_s}{\psi} - \inner{\mathcal{Q}_i\tilde{u}_s}{\mathcal{Q}_i^*\psi}\right\vert &= \left\vert\inner{\mathcal{Q}_i\tilde{u}^n_s}{\mathcal{Q}_i^*[\mathcal{P}_n\psi]} - \inner{\mathcal{Q}_i\tilde{u}_s}{\mathcal{Q}_i^*\psi}\right\vert\\
    &= \left\vert\inner{\mathcal{Q}_i\left(\tilde{u}^n_s - \tilde{u}_s\right)}{\mathcal{Q}_i^*[\mathcal{P}_n\psi]} - \inner{\mathcal{Q}
_i\tilde{u}_s}{\mathcal{Q}_i^*\left[(I - \mathcal{P}_n)\psi\right]}\right\vert\\
    &= \left\vert\inner{\tilde{u}^n_s - \tilde{u}_s}{(\mathcal{Q}_i^*)^2[\mathcal{P}_n\psi]} - \inner{\mathcal{Q}
_i\tilde{u}_s}{\mathcal{Q}_i^*\left[(I - \mathcal{P}_n)\psi\right]}\right\vert\\
&\leq c_i\norm{\tilde{u}^n_s - \tilde{u}_s}\norm{\psi}_2 + \frac{c_i}{\sqrt{\lambda_n}}\left(1 +\norm{\tilde{u}_s}_1\right)\norm{\psi}_2
\end{align*}
using (\ref{assumpty 1}), (\ref{assumpty 3}), so the convergence follows as with the previous term. It now only remains to consider the stochastic integral, for which we take the same approach and observe that
\begin{align*}
    \left\vert\inner{\mathcal{P}_n\mathcal{P}\mathcal{G}_i\tilde{u}^n_s}{\psi} - \inner{\mathcal{G}_i\tilde{u}_s}{\psi}\right\vert^2 &= \left\vert\inner{\mathcal{P}_n\mathcal{P}\left[\mathcal{G}_i\tilde{u}^n_s - \mathcal{G}_i\tilde{u}_s\right]}{\psi} + \inner{(I -\mathcal{P}_n)\mathcal{G}_i\tilde{u}_s}{\psi}\right\vert^2\\
    &\leq 2\inner{\mathcal{G}_i\tilde{u}^n_s - \mathcal{G}_i\tilde{u}_s}{\mathcal{P}_n\psi}^2 + 2\inner{\mathcal{G}_i\tilde{u}_s}{(I - \mathcal{P}_n)\psi}^2\\
&\leq c_i\left[1 +\norm{\psi}_{2}^p\right]\norm{\tilde{u}^n_s - \tilde{u}_s}^2 + \frac{c_i}{\lambda_n}\left(1 + \norm{\tilde{u}_s}_1^2\right)\norm{\psi}_1^2
\end{align*}
using (\ref{assumpty 1}), (\ref{assumpty10}). Therefore 
\begin{align*}
    &\mathbbm{E}\left\vert \int_0^t\inner{\mathcal{P}_n\mathcal{P}\mathcal{G}\tilde{u}^n_s}{\psi} d\mathcal{W}_s - \int_0^t\inner{\mathcal{G}\tilde{u}_s}{\psi} d\mathcal{W}_s \right\vert^2\\ 
    &\qquad \qquad \qquad  = \mathbbm{E}\int_0^t \sum_{i=1}^\infty\left\vert\inner{\mathcal{P}_n\mathcal{P}\mathcal{G}_i\tilde{u}^n_s}{\psi} - \inner{\mathcal{G}_i\tilde{u}_s}{\psi}\right\vert^2 ds\\
    & \qquad \qquad \qquad \leq c\norm{\psi}_2^p\mathbbm{E}\int_0^t \norm{\tilde{u}^n_s - \tilde{u}_s}^2ds + \frac{c}{\lambda_n}\norm{\psi}_1^2\mathbbm{E}\int_0^t1 +\norm{\tilde{u}_s}_{1}^2ds 
\end{align*}
which evidently approaches zero, so as discussed by taking $\tilde{\mathbbm{P}}-a.s.$ convergent subsequences this is sufficient to conclude that  
\begin{align} \nonumber
     \inner{\tilde{u}_t}{\psi} = \inner{\tilde{u}_0}{\psi} - \int_0^{t}\inner{\mathcal{L}_{\tilde{u}_s}\tilde{u}_s}{\psi}ds &- \nu\int_0^{t} \inner{\tilde{u}_s}{\psi}_1 ds\\ &+ \frac{\nu}{2}\int_0^{t}\sum_{i=1}^\infty \inner{\mathcal{Q}_i\tilde{u}_s}{\mathcal{Q}_i^*\psi} ds - \nu^{\frac{1}{2}}\int_0^{t} \inner{\mathcal{G}\tilde{u}_s}{\psi} d\tilde{\mathcal{W}}_s.\nonumber
\end{align}
To pass to the identity for $\phi \in W^{1,2}_{\sigma}$ we now fix such a $\phi$ and consider a sequence $(\psi^k)$ in $W^{2,2}_{\sigma}$ convergent to $\phi$ in $W^{1,2}_{\sigma}$. As $k \rightarrow \infty$ the limits
\begin{align*}
\inner{\tilde{u}_t}{\psi^k} &\longrightarrow \inner{\tilde{u}_t}{\phi}\\
\inner{\tilde{u}_0}{\psi^k} &\longrightarrow \inner{\tilde{u}_0}{\phi}\\
\int_0^{t} \inner{\tilde{u}_s}{\psi^k}_1 ds = \left\langle \int_0^t \tilde{u}_s ds, \psi^k \right\rangle_1 &\longrightarrow \left\langle \int_0^t \tilde{u}_s ds, \phi \right\rangle_1 = \int_0^{t} \inner{\tilde{u}_s}{\phi}_1 ds\\
\int_0^{t} \inner{\mathcal{G}\tilde{u}_s}{\psi^k} d\tilde{\mathcal{W}}_s = \left\langle\int_0^{t}\mathcal{G}\tilde{u}_sd\tilde{\mathcal{W}}_s , \psi^k \right\rangle &\longrightarrow \left\langle\int_0^{t}\mathcal{G}\tilde{u}_sd\tilde{\mathcal{W}}_s , \phi \right\rangle = \int_0^{t} \inner{\mathcal{G}\tilde{u}_s}{\phi} d\tilde{\mathcal{W}}_s
\end{align*}
$\tilde{\mathbbm{P}}-a.s.$ are trivial, and for the remaining two integrals we may simply apply the Dominated Convergence Theorem with (\ref{assumpty 3}) in mind. This justifies the identity (\ref{newid1}). The final property to prove is that for $\tilde{\mathbbm{P}}-a.e.$ $\omega$, $\tilde{u}_{\cdot}(\omega) \in C_{w}\left([0,T];L^2_{\sigma}\right)$. By the identity just shown it is clear that $\inner{\tilde{u}_{\cdot}(\omega)}{\phi} \in C\left([0,T];\R\right)$ where $\phi \in W^{1,2}_{\sigma}$ was arbitrary, but to conclude the weak continuity we must instead show this for any $\eta \in L^2_{\sigma}$. Furthermore we fix such an $\omega$ and $\eta \in L^2_{\sigma}$, any $t \in [0,T]$ and sequence of times $(t_k)$ in $[0,T]$ such that $t_k \rightarrow t$. To demonstrate the continuity let's fix $\varepsilon >0$, and choose a $\phi \in W^{1,2}_{\sigma}$ such that $$\norm{\eta - \phi} < \frac{\varepsilon}{4}\sup_{s\in[0,T]}\norm{\tilde{u}_s(\omega)}$$ where the right hand side is of course finite from Proposition \ref{prop for regularity of limit}. Note that there exists a $K \in \N$ such that for all $k \geq K$, $$\left\vert\inner{\tilde{u}_{t_k}(\omega) -\tilde{u}_{t}(\omega)}{\phi}\right\vert < \frac{\varepsilon}{2}.$$
Then for all $k \geq K$ we have that
\begin{align*}
   \left\vert\inner{\tilde{u}_{t_k}(\omega) -\tilde{u}_{t}(\omega)}{\eta}\right\vert &\leq \left\vert\inner{\tilde{u}_{t_k}(\omega) -\tilde{u}_{t}(\omega)}{\eta - \phi}\right\vert + \left\vert\inner{\tilde{u}_{t_k}(\omega) -\tilde{u}_{t}(\omega)}{\phi}\right\vert\\
   &<2 \sup_{s\in[0,T]}\norm{\tilde{u}_s(\omega)}\norm{\eta - \phi} + \frac{\varepsilon}{2}\\
   &< \varepsilon
\end{align*}
demonstrating the weak continuity and finishing the proof.

\end{proof}

\subsection{Probabilistically Strong Solutions in 2D} \label{sub 2d}

This subsection is dedicated to proving Theorem \ref{theorem 2D}. We fix $N=2$ as well as an arbitrary $\mathcal{F}_0-$measurable $u_0\in L^\infty(\Omega;L^2_{\sigma})$, and consider a martingale weak solution $\tilde{u}$ known to exist from the now proven Theorem \ref{existence of weak}. To make progress we look to understand $\tilde{u}$ as satisfying an identity in $\left(W^{1,2}_{\sigma}\right)^*$.

\begin{lemma} \label{originally unlabelled}
 $\mathcal{P}\mathcal{L}_{\tilde{u}_{\cdot}}\tilde{u}_{\cdot}, A\tilde{u}_{\cdot}$ and $ \sum_{i=1}^\infty\mathcal{P}\mathcal{Q}_i^2\tilde{u}_{\cdot}$ all belong to $L^2\left( [0,T]; \left(W^{1,2}_{\sigma}\right)^*\right)$ $\tilde{\mathbbm{P}}-a.s.$. Moreover for every $t\in[0,T]$, $\tilde{u}$ satisfies the identity \begin{equation} \label{another another id}
    \tilde{u}_{t} = \tilde{u}_0 - \int_0^t\mathcal{P}\mathcal{L}_{\tilde{u}_s}\tilde{u}_s\ ds - \nu\int_0^t  A \tilde{u}_s\, ds + \frac{\nu}{2}\int_0^t\sum_{i=1}^\infty \mathcal{P}\mathcal{Q}_i^2\tilde{u}_s ds - \nu^{\frac{1}{2}}\int_0^t \mathcal{P}\mathcal{G}\tilde{u}_s d\tilde{\mathcal{W}}_s 
\end{equation}
$\tilde{\mathbb{P}}-a.s.$ in $\left(W^{1,2}_{\sigma}\right)^*$.
\end{lemma}

\begin{proof}
    We first address how these functions define elements of $\tilde{\mathbb{P}}-a.s.$ in $\left(W^{1,2}_{\sigma}\right)^*$. For every $s \in [0,T]$ and $\tilde{\mathbbm{P}}-a.e.$ $\omega \in \tilde{\Omega}$, $\mathcal{P}\mathcal{L}_{\tilde{u}_s(\omega)}\tilde{u}_s(\omega)$ defines such an element by the duality pairing $$\inner{\mathcal{L}_{\tilde{u}_s(\omega)}\tilde{u}_s(\omega)}{\phi}$$ for $\phi \in W^{1,2}_{\sigma}$ as verified in (\ref{why the star}), noting that if $\mathcal{P}\mathcal{L}_{\tilde{u}_s(\omega)}\tilde{u}_s(\omega) \in L^2_{\sigma}$ then $$\inner{\mathcal{P}\mathcal{L}_{\tilde{u}_s(\omega)}\tilde{u}_s(\omega)}{\phi} = \inner{\mathcal{L}_{\tilde{u}_s(\omega)}\tilde{u}_s(\omega)}{\phi} = \inner{\mathcal{L}_{\tilde{u}_s(\omega)}\tilde{u}_s(\omega)}{\phi}_{L^{6/5} \times L^6}$$ so the representation is consistent. Similarly $A\tilde{u}_{s}(\omega)$ is defined by $\inner{\tilde{u}_s(\omega)}{\phi}_1$ appreciating that if $A\tilde{u}_{s}(\omega)\in L^2_{\sigma}$ (which is only well defined for $\tilde{u}_{s}(\omega) \in W^{2,2}_{\sigma}$) then $$\inner{A\tilde{u}_{s}(\omega)}{\phi} = \inner{\tilde{u}_{s}(\omega)}{\phi}_1.$$ By the same process, $\sum_{i=1}^\infty \mathcal{P}\mathcal{Q}_i^2\tilde{u}_s(\omega)$ is consistently defined by $\sum_{i=1}^\infty\inner{\mathcal{Q}_i\tilde{u}_s(\omega)}{\mathcal{Q}_i^*\phi}$. It is in showing the $L^2\left( [0,T]; \left(W^{1,2}_{\sigma}\right)^*\right)$ regularity that we make use of the special case $N=2$. For the nonlinear term, note that
    \begin{align*}
\norm{\mathcal{P}\mathcal{L}_{\tilde{u}_s}\tilde{u}_s}_{\left(W^{1,2}_{\sigma}\right)^*} = \sup_{\norm{\phi}_1 =1}\left\vert\inner{\mathcal{L}_{\tilde{u}_s}\tilde{u}_s}{\phi}\right\vert = \sup_{\norm{\phi}_1 =1}\left\vert\inner{\tilde{u}_s}{\mathcal{L}_{\tilde{u}_s}\phi}\right\vert
    \end{align*}
   using (\ref{wloglhs}), and from two instances of H\"{o}lder's Inequality as well as Theorem \ref{gagliardonirenberginequality} with $p=4, q=2, \alpha = 1/2$ and $m=1$, 
    \begin{align} \nonumber
        \left\vert\inner{\tilde{u}_s}{\mathcal{L}_{\tilde{u}_s}\phi}\right\vert \leq \norm{\tilde{u}_s}_{L^4}\norm{\mathcal{L}_{\tilde{u}_s}\phi}_{L^{4/3}} &\leq c\sum_{k=1}^2\norm{\tilde{u}_s}_{L^4}\norm{\tilde{u}_s}_{L^4}\norm{\partial_k\phi}\\ &\leq c\norm{\tilde{u}_s}\norm{\tilde{u}_s}_1\norm{\phi}_1. \label{a bound in align}
    \end{align}
    Combining the two we have that
    \begin{align*}
\int_0^T\norm{\mathcal{P}\mathcal{L}_{\tilde{u}_s}\tilde{u}_s}_{\left(W^{1,2}_{\sigma}\right)^*}^2ds \leq c\int_0^T\norm{\tilde{u}_s}^{2}\norm{\tilde{u}_s}^{2}_1ds \leq c\norm{\tilde{u}}^{2}_{L^\infty([0,T];L^2_{\sigma})}\int_0^T\norm{\tilde{u}_s}^{2}_1ds < \infty.\end{align*}
 A simpler argument justifies that $A\tilde{u}, \sum_{i=1}^\infty\mathcal{P}\mathcal{Q}_i^2\tilde{u} \in L^2\left([0,T]; \left(W^{1,2}_{\sigma}\right)^*\right)$, and indeed all are progressively measurable inherited from the progressive measurability of $\tilde{u}$ in $W^{1,2}_{\sigma}$ and the measurablity of the mappings from $W^{1,2}_{\sigma}$ into $\left(W^{1,2}_{\sigma}\right)^*$. As for (\ref{another another id}), this now follows immediately from rewriting the inner products of (\ref{newid1}) in terms of the duality pairings given here, and taking $\phi$ in this pairing outside of the integral. We note that the stochastic integral is well defined in $L^2_{\sigma}$, and is then simply embedded into $(W^{1,2}_{\sigma})^*$.  
\end{proof}

\begin{corollary}
For $\tilde{\mathbbm{P}}-a.e.$ $\omega$, $\tilde{u}_{\cdot}(\omega) \in C\left([0,T];L^2_{\sigma}\right)$.
\end{corollary}

\begin{proof}
    This is now an immediate application of Proposition \ref{rockner prop}.
\end{proof}

It is the fact that the nonlinear term does not satisfy the regularity of Lemma \ref{originally unlabelled} in 3D that one cannot deduce the continuity in that setting. The same is true for the uniqueness, which we prove now.

\begin{proposition} \label{first uniqueness prop}
    Suppose that $\tilde{w}$ is another martingale weak solution of (\ref{projected Ito}) with respect to the same filtered probability space $\left(\tilde{\Omega},\tilde{\mathcal{F}},(\tilde{\mathcal{F}}_t), \tilde{\mathbbm{P}}\right)$, cylindrical Brownian Motion $\tilde{\mathcal{W}}$ and initial condition $\tilde{w}_0 = \tilde{u}_0$ $\tilde{\mathbbm{P}}-a.s.$. In addition assume that for $\tilde{\mathbbm{P}}-a.e.$ $\omega$, $\tilde{w}_{\cdot}(\omega) \in C\left([0,T];L^2_{\sigma}\right)$. Then
    $$ \tilde{\mathbbm{P}}\left(\left\{\omega \in \tilde{\Omega}: \tilde{u}_t(\omega) = \tilde{w}_t(\omega) \quad \forall t\in[0,T]\right\}\right) = 1.$$
\end{proposition}

\begin{proof}
    We make our argument by considering the expectation of the difference of the solutions $\tilde{u},\tilde{w}$, and to do so we need to manufacture an increased regularity through stopping times once more. To this end let's define the stopping times $(\alpha_R)$ by $$\alpha_R := T \wedge \inf \left\{r \geq 0: \sup_{s\in[0,r]}\norm{\tilde{u}_s}^2 + \int_0^r \norm{\tilde{u}_s}_1^2ds \geq R \right\} \wedge \inf \left\{r \geq 0: \sup_{s\in[0,r]}\norm{\tilde{w}_s}^2 + \int_0^r \norm{\tilde{w}_s}_1^2ds \geq R \right\}$$ and subsequent processes
    $$\tilde{u}^R_{\cdot}:=\tilde{u}_{\cdot}\mathbbm{1}_{\cdot \leq \alpha_R}, \qquad \tilde{w}^R_{\cdot}:=\tilde{w}_{\cdot}\mathbbm{1}_{\cdot \leq \alpha_R}, \qquad \sy_{\cdot} = \tilde{u}^R - \tilde{w}^R.$$ Moreover the difference process satisfies
\begin{align*}
    \tilde{u}_{t\wedge \alpha_R} - \tilde{w}_{t\wedge \alpha_R} &= -\int_0^t\mathcal{P}\mathcal{L}_{\tilde{u}^R_s}\tilde{u}^R_s - \mathcal{P}\mathcal{L}_{\tilde{w}^R_s}\tilde{w}^R_s\ ds - \nu\int_0^t  A\sy_s\, ds\\ &+ \frac{\nu}{2}\int_0^t\sum_{i=1}^\infty \mathcal{P}\mathcal{Q}_i^2\sy_s ds - \nu^{\frac{1}{2}}\int_0^t\mathbbm{1}_{s \leq \alpha_R} \left(\mathcal{P}\mathcal{G}\tilde{u}^R_s - \mathcal{P}\mathcal{G}\tilde{w}^R_s \right)d\tilde{\mathcal{W}}_s
\end{align*}
and we can apply the Energy Equality of Proposition \ref{rockner prop} to see that
\begin{align*}
    \norm{\tilde{u}_{t\wedge \alpha_R} - \tilde{w}_{t\wedge \alpha_R}}^2 &= -2\int_0^t\inner{\mathcal{P}\mathcal{L}_{\tilde{u}^R_s}\tilde{u}^R_s - \mathcal{P}\mathcal{L}_{\tilde{w}^R_s}\tilde{w}^R_s}{\sy_s}_{(W^{1,2}_{\sigma})^* \times W^{1,2}_{\sigma}} ds\\ &- 2\nu\int_0^t  \inner{A\sy_s}{\sy_s}_{(W^{1,2}_{\sigma})^* \times W^{1,2}_{\sigma}} ds + \nu\int_0^t\sum_{i=1}^\infty \inner{\mathcal{P}\mathcal{Q}_i^2\sy_s}{\sy_s}_{(W^{1,2}_{\sigma})^* \times W^{1,2}_{\sigma}} ds\\ &+\nu\int_0^t\mathbbm{1}_{s \leq \alpha_R}\sum_{i=1}^\infty\norm{\mathcal{P}\mathcal{G}_i\tilde{u}^R_s - \mathcal{P}\mathcal{G}_i\tilde{w}^R_s}^2ds - 2\nu^{\frac{1}{2}}\int_0^t \inner{\mathcal{P}\mathcal{G}\tilde{u}^R_s - \mathcal{P}\mathcal{G}\tilde{w}^R_s }{\sy_s}d\tilde{\mathcal{W}}_s.
\end{align*}
Motivated by the use of Lemma \ref{gronny}, we consider arbitrary stopping times $0 \leq \theta_j \leq \theta_k \leq T$ and substitute $\theta_j$ into the above, then subtract this from the identity for any $\theta_j \leq r \leq T$, to give that 
\begin{align*}
    \norm{\tilde{u}_{r\wedge \alpha_R} - \tilde{w}_{r\wedge \alpha_R}}^2 &= \norm{\tilde{u}_{\theta_j\wedge \alpha_R} - \tilde{w}_{\theta_j\wedge \alpha_R}}^2-2\int_{\theta_j}^r\inner{\mathcal{P}\mathcal{L}_{\tilde{u}^R_s}\tilde{u}^R_s - \mathcal{P}\mathcal{L}_{\tilde{w}^R_s}\tilde{w}^R_s}{\sy_s}_{(W^{1,2}_{\sigma})^* \times W^{1,2}_{\sigma}} ds\\ &- 2\nu\int_{\theta_j}^r  \inner{A\sy_s}{\sy_s}_{(W^{1,2}_{\sigma})^* \times W^{1,2}_{\sigma}} ds + \nu \int_{\theta_j}^r\sum_{i=1}^\infty \inner{\mathcal{P}\mathcal{Q}_i^2\sy_s}{\sy_s}_{(W^{1,2}_{\sigma})^* \times W^{1,2}_{\sigma}} ds\\ &+\nu\int_{\theta_j}^r\mathbbm{1}_{s \leq \alpha_R}\sum_{i=1}^\infty\norm{\mathcal{P}\mathcal{G}_i\tilde{u}^R_s - \mathcal{P}\mathcal{G}_i\tilde{w}^R_s}^2ds - 2\nu^{\frac{1}{2}}\int_{\theta_j}^r \inner{\mathcal{P}\mathcal{G}\tilde{u}^R_s - \mathcal{P}\mathcal{G}\tilde{w}^R_s }{\sy_s}d\tilde{\mathcal{W}}_s.
\end{align*}
We break this down term by term, starting with the nonlinear term which is by definition
$$\inner{\mathcal{P}\mathcal{L}_{\tilde{u}^R_s}\tilde{u}^R_s - \mathcal{P}\mathcal{L}_{\tilde{w}^R_s}\tilde{w}^R_s}{\sy_s}_{(W^{1,2}_{\sigma})^* \times W^{1,2}_{\sigma}} = \inner{\mathcal{L}_{\tilde{u}^R_s}\tilde{u}^R_s - \mathcal{L}_{\tilde{w}^R_s}\tilde{w}^R_s}{\sy_s}.$$
Using (\ref{cancellationproperty'}) and then (\ref{a bound in align}),
\begin{align}\nonumber \left\vert\inner{\mathcal{L}_{\tilde{u}^R_s}\tilde{u}^R_s - \mathcal{L}_{\tilde{w}^R_s}\tilde{w}^R_s}{\sy_s}\right\vert &=\left\vert\inner{\mathcal{L}_{\sy_s }\tilde{u}^R_s + \mathcal{L}_{\tilde{w}^R_s}\sy_s}{\sy_s}\right\vert\\ \nonumber &= \left\vert\inner{\mathcal{L}_{\sy_s }\tilde{u}^R_s}{\sy_s}\right\vert\\\nonumber &\leq c\norm{\sy_s}\norm{\sy_s}_1\norm{\tilde{u}^R_s}_1\\ \label{little star}
&\leq c\norm{\tilde{u}^R_s}_1^2\norm{\sy_s}^2 + \nu\norm{\sy_s}_1^2
\end{align} 
where $c$ now depends on $\nu$, which is not meaningful here. In addition observe that \begin{equation} \label{another forgot}
    \inner{A\sy_s}{\sy_s}_{(W^{1,2}_{\sigma})^* \times W^{1,2}_{\sigma}} = \norm{\sy_s}^2_1
\end{equation}
and we combine the next two integrals as
\begin{align*}&\nu\int_{\theta_j}^r\mathbbm{1}_{s \leq \alpha_R}\sum_{i=1}^\infty\left(\inner{\mathcal{P}\mathcal{Q}_i^2\sy_s}{\sy_s}_{(W^{1,2}_{\sigma})^* \times W^{1,2}_{\sigma}} + \norm{\mathcal{P}\mathcal{G}_i\tilde{u}^R_s - \mathcal{P}\mathcal{G}_i\tilde{w}^R_s}^2\right)ds\\ & \qquad \qquad \qquad \leq \nu\int_{\theta_j}^r\mathbbm{1}_{s \leq \alpha_R}\sum_{i=1}^\infty\left(\inner{\mathcal{Q}_i\sy_s}{\mathcal{Q}_i^*\sy_s} + \norm{\mathcal{G}_i\tilde{u}^R_s - \mathcal{G}_i\tilde{w}^R_s}^2\right)ds\end{align*}
using the definition of the duality pairing and that $\mathcal{P}$ is an orthogonal projection in $L^2(\mathscr{O};\R^2)$. We apply (\ref{assumpty 5}) to bound this again by \begin{equation} \label{i forgot}c\int_{\theta_j}^r\left(1 + \norm{\tilde{u}^R_s}_1^2 + \norm{\tilde{w}^R_s}_1^2\right)\norm{\sy_s}^2 + \nu\norm{\sy_s}_1^2 ds\end{equation}
where the constant $c$ now also depends on $R$ having used that $\norm{\tilde{u}^R_s}^2, \norm{\tilde{w}^R_s}^2 \leq R$ (which we remark is only true as $\tilde{u}, \tilde{w}$ are pathwise continuous in $L^2_{\sigma}$). Combining (\ref{little star}), (\ref{another forgot}) and (\ref{i forgot}) we see that 
\begin{align*}
    \norm{\tilde{u}_{r\wedge \alpha_R} - \tilde{w}_{r\wedge \alpha_R}}^2 &\leq \norm{\tilde{u}_{\theta_j\wedge \alpha_R} - \tilde{w}_{\theta_j\wedge \alpha_R}}^2 + c\int_{\theta_j}^r\left(1 + \norm{\tilde{u}^R_s}_1^2 + \norm{\tilde{w}^R_s}_1^2\right)\norm{\sy_s}^2ds\\ &- 2\nu^{\frac{1}{2}}\int_{\theta_j}^r \inner{\mathcal{G}\tilde{u}^R_s - \mathcal{G}\tilde{w}^R_s }{\sy_s}d\tilde{\mathcal{W}}_s.
\end{align*}
We now take the absolute value on the right hand side, followed by the supremum over $r \in [\theta_j,\theta_k]$, then the expectation and immediately apply the Burkholder-Davis-Gundy Inequality to achieve that
\begin{align*}
    &\tilde{\mathbbm{E}}\left(\sup_{r\in[\theta_j,\theta_k]}\norm{\tilde{u}_{r\wedge \alpha_R} - \tilde{w}_{r\wedge \alpha_R}}^2\right) \leq \tilde{\mathbbm{E}}\left(\norm{\tilde{u}_{\theta_j\wedge \alpha_R} - \tilde{w}_{\theta_j\wedge \alpha_R}}^2\right)\\ & \qquad \qquad \qquad + c\tilde{\mathbbm{E}}\int_{\theta_j}^{\theta_k}\left(1 + \norm{\tilde{u}^R_s}_1^2 + \norm{\tilde{w}^R_s}_1^2\right)\norm{\sy_s}^2ds + c\tilde{\mathbbm{E}}\left(\int_{\theta_j}^{\theta_k} \sum_{i=1}^\infty \inner{\mathcal{G}_i\tilde{u}^R_s - \mathcal{G}_i\tilde{w}^R_s }{\sy_s}^2ds\right)^{\frac{1}{2}}.
\end{align*}
We now use (\ref{assumpty 6}) and follow the same process as (\ref{the same process}) to obtain that
\begin{align*}
    &c\tilde{\mathbbm{E}}\left(\int_{\theta_j}^{\theta_k} \sum_{i=1}^\infty \inner{\mathcal{G}_i\tilde{u}^R_s - \mathcal{G}_i\tilde{w}^R_s }{\sy_s}^2ds\right)^{\frac{1}{2}}\\ & \qquad \qquad \qquad \leq c\tilde{\mathbbm{E}}\left(\int_{\theta_j}^{\theta_k} \left(1 + \norm{\tilde{u}^R_s}_1^2 +\norm{\tilde{w}^R_s}_1^2\right)\norm{\sy_s}^4 ds\right)^{\frac{1}{2}}\\&\qquad \qquad \qquad \leq \frac{1}{2}\tilde{\mathbbm{E}}\left(\sup_{r\in[\theta_j,\theta_k]}\norm{\sy_r}^2\right) + c\tilde{\mathbbm{E}}\int_{\theta_j}^{\theta_k} \left(1 + \norm{\tilde{u}^R_s}_1^2 +\norm{\tilde{w}^R_s}_1^2\right)\norm{\sy_s}^2ds.
\end{align*}
We now use that $\norm{\sy_r}^2 \leq \norm{\tilde{u}_{r\wedge \alpha_R} - \tilde{w}_{r\wedge \alpha_R}}^2$ and rearrange to give
\begin{align*}
    \tilde{\mathbbm{E}}\left(\sup_{r\in[\theta_j,\theta_k]}\norm{\tilde{u}_{r\wedge \alpha_R} - \tilde{w}_{r\wedge \alpha_R}}^2\right) &\leq 2\tilde{\mathbbm{E}}\left(\norm{\tilde{u}_{\theta_j\wedge \alpha_R} - \tilde{w}_{\theta_j\wedge \alpha_R}}^2\right)\\ &+ c\tilde{\mathbbm{E}}\int_{\theta_j}^{\theta_k}\left(1 + \norm{\tilde{u}^R_s}_1^2 + \norm{\tilde{w}^R_s}_1^2\right)\norm{\tilde{u}_{s\wedge \alpha_R} - \tilde{w}_{s\wedge \alpha_R}}^2ds.
\end{align*}
We can now apply Lemma \ref{gronny} to deduce that $$\tilde{\mathbbm{E}}\left(\sup_{r\in[0,T]}\norm{\tilde{u}_{r\wedge \alpha_R} - \tilde{w}_{r\wedge \alpha_R}}^2\right) = 0$$ as of course $\tilde{u}_0 = \tilde{w}_0$ $\tilde{\mathbbm{P}}-a.s.$. We note that $\left(\sup_{r\in[0,T]}\norm{\tilde{u}_{r\wedge \alpha_R} - \tilde{w}_{r\wedge \alpha_R}}^2\right)$ is a monotone increasing sequence in $R$, hence we take the limit as $R \rightarrow \infty$ and apply the Monotone Convergence Theorem to obtain $$\tilde{\mathbbm{E}}\left(\sup_{r\in[0,T]}\norm{\tilde{u}_{r} - \tilde{w}_{r}}^2\right) = 0$$ which gives the result.
\end{proof}

It is now immediate that Theorem \ref{theorem 2D} holds in this case of the bounded initial condition.

\begin{corollary}
    There exists a unique weak solution $u$ of the equation (\ref{projected Ito}) with the property that for $\mathbbm{P}-a.e.$ $\omega$, $u_{\cdot}(\omega) \in C\left([0,T];L^2_{\sigma}\right)$.
\end{corollary}

\begin{proof}
    This follows from a classical Yamada-Watanabe type result, proven rigorously in this setting in [\cite{rockner2008yamada}].
\end{proof}

To prove Theorem \ref{theorem 2D} it thus only remains to extend the result to an arbitrary $\mathcal{F}_0-$measurable $u_0: \Omega \rightarrow L^2_{\sigma}$, which we now fix.

\begin{proof}[Proof of Theorem \ref{theorem 2D}:]
    We first show the existence of such a solution. The idea is as in [\cite{goodair2022existence1}] Theorem 3.40 where we use the fact that for each $k \in \N \cup \{0\}$ there exists a weak solution $u^k$ of the equation (\ref{projected Ito}) for the initial condition $u_0\mathbbm{1}_{k \leq \norm{u_0} < k+1}$. We argue that the process $u$ defined by $$u_t(\omega):= \sum_{k=1}^\infty u^k_t(\omega)\mathbbm{1}_{k \leq \norm{u_0(\omega)} < k+1}$$ is a weak solution. Appreciating that the infinite sum is merely formal and that for each $\omega$ $u(\omega):=u^k(\omega)$ for some $k$, then clearly $u$ inherits the pathwise regularity of the weak solutions $(u^k)$. As for the identity (\ref{identityindefinitionofspatiallyweak}), we introduce the more compact notation $$A_k:=\left\{\omega \in \Omega: k \leq \norm{u_0(\omega)} < k+1 \right\}$$ and as the $(A_k)$ partition $\Omega$, it is sufficient to show that 
\begin{align} \nonumber
     \mathbbm{1}_{A_k}\inner{u_t}{\phi} = \mathbbm{1}_{A_k}\inner{u_0}{\phi} & - \mathbbm{1}_{A_k}\int_0^{t}\inner{\mathcal{L}_{u_s}u_s}{\phi}ds - \nu\mathbbm{1}_{A_k}\int_0^{t} \inner{u_s}{\phi}_1 ds\\ &+ \frac{1}{2}\mathbbm{1}_{A_k}\int_0^{t}\sum_{i=1}^\infty \inner{\mathcal{Q}_iu_s}{\mathcal{Q}_i^*\phi} ds - \mathbbm{1}_{A_k}\int_0^{t} \inner{\mathcal{G}u_s}{\phi } d\mathcal{W}_s\nonumber
\end{align}
or equivalently
  \begin{align} \nonumber
     \mathbbm{1}_{A_k}\inner{u^k_t}{\phi} = \mathbbm{1}_{A_k}\inner{u_0\mathbbm{1}_{A_k}}{\phi} & - \mathbbm{1}_{A_k}\int_0^{t}\inner{\mathcal{L}_{u^k_s}u^k_s}{\phi}ds - \nu\mathbbm{1}_{A_k}\int_0^{t} \inner{u^k_s}{\phi}_1 ds\\ &+ \frac{\nu}{2}\mathbbm{1}_{A_k}\int_0^{t}\sum_{i=1}^\infty \inner{\mathcal{Q}_iu^k_s}{\mathcal{Q}_i^*\phi} ds - \nu^{\frac{1}{2}}\mathbbm{1}_{A_k}\int_0^{t} \inner{\mathcal{G}u_s}{\phi } d\mathcal{W}_s.\nonumber
\end{align}
We are a little more precise for the stochastic integral as we cannot simply take any random function through the integral, however $A_k$ is $\mathcal{F}_0-$measurable so it is justified here (see e.g. [\cite{goodair2022stochastic}] Proposition 1.6.14) hence $$\mathbbm{1}_{A_k}\int_0^{t} \inner{\mathcal{G}u_s}{\phi } d\mathcal{W}_s = \int_0^{t} \mathbbm{1}_{A_k}\inner{\mathcal{G}u_s}{\phi } d\mathcal{W}_s = \int_0^{t} \mathbbm{1}_{A_k}\inner{\mathcal{G}u^k_s}{\phi } d\mathcal{W}_s =  \mathbbm{1}_{A_k}\int_0^{t}\inner{\mathcal{G}u^k_s}{\phi } d\mathcal{W}_s.$$ This identity is granted from $u^k$ being a weak solution for the initial condition $u_0\mathbbm{1}_{A_k}$. To conclude the existence we only need to verify the progressive measurability, for which we understand $u$ as the pointwise almost everywhere limit of the sequence $\left(\sum_{k=1}^n u^k\mathbbm{1}_{k \leq \norm{u_0} < k+1} \right)$ over the product space $\Omega \times [0,t]$ equipped with the product sigma algebra $\mathcal{F}_t \times \mathcal{B}([0,t])$ in $W^{1,2}_{\sigma}$. Each $u^k$ is progressively measurable hence so too is $u^k\mathbbm{1}_{k \leq \norm{u_0} < k+1}$, thus measurable with respect to $\mathcal{F}_t \times \mathcal{B}([0,t])$, and the pointwise almost everywhere limit preserves the measurability which provides the result. This concludes the proof that $u$ is a weak solution of (\ref{projected Ito}) with the property that for $\mathbbm{P}-a.e.$ $\omega$, $u_{\cdot}(\omega) \in C\left([0,T];L^2_{\sigma}\right)$, and one can show it is the unique such solution identically to Proposition \ref{first uniqueness prop}.
\end{proof}

\subsection{Energy Estimates for the Constructed Solution} \label{sub energy}

We prove Proposition \ref{ex theorem}.

\begin{proof}[Proof of Proposition \ref{ex theorem}:]
    Of course the existence of a martingale weak solution comes from Theorem \ref{existence of weak}, so it is only the estimates (\ref{hello1}) and (\ref{hello3}) which must be justified and can be done so for the $\tilde{u}$ constructed in the proof of Theorem \ref{existence of weak}. We start with (\ref{hello1}): recall in Proposition \ref{prop for regularity of limit} how the regularity of the limit process was obtained from the uniform bounds (\ref{second result}). In the same manner it is sufficient to show that 
    \begin{equation} \label{hello2}
\mathbbm{E}\left(\sup_{r\in[0,T]}\norm{u^{n}_{r}}^2 \right) \leq \left(1 + o_{\nu}\right) \norm{u^n_0}^2 + o_{\nu}
    \end{equation}
    for every $n\in \N$, where $u^n$ is the strong solution of
    \begin{equation} \nonumber
    u^n_t = u^n_0 - \int_0^t\mathcal{P}_n\mathcal{P}\mathcal{L}_{u^n_s}u^n_s\ ds - \nu\int_0^t \mathcal{P}_n A u^n_s\, ds + \frac{\nu}{2}\int_0^t\sum_{i=1}^\infty \mathcal{P}_n \mathcal{P}\mathcal{Q}_i^2u^n_s ds - \nu^{\frac{1}{2}}\int_0^t \mathcal{P}_n\mathcal{P}\mathcal{G}u^n_s d\mathcal{W}_s
\end{equation}
in analogy with (\ref{projected Ito galerkin}), and $o_{\nu}$ is independent of $n$. Identically to Proposition \ref{prop for first energy} but simply ignoring the contribution from the Stokes Operator in the inequality and using that the initial condition is deterministic, we arrive at 
  \begin{align*}
\mathbbm{E}\left(\sup_{r\in[0,T]}\norm{u^{n}_{r}}^2\right) \leq \norm{u^n_0}^2 + c\nu + c\nu\int_0^{T} \mathbbm{E}\left(\norm{u^{n}_s}^2\right) ds + c\nu^{\frac{1}{2}}\mathbbm{E}\left(\int_0^{T}\sum_{i=1}^\infty \inner{\mathcal{G}_i\hat{u}^{n}_s}{\hat{u}^{n}_s}^2 ds\right)^\frac{1}{2}.
\end{align*}
The final term is controlled similarly again, just a little more precisely:
\begin{align*}
        c\nu^{\frac{1}{2}}\mathbbm{E}\left(\int_0^{T}\sum_{i=1}^\infty \inner{\mathcal{G}_i\hat{u}^{n}_s}{\hat{u}^{n}_s}^2 ds\right)^\frac{1}{2}
        &\leq c\nu^{\frac{1}{2}} + c\nu^{\frac{1}{2}}\mathbbm{E}\left(\sup_{r\in[0,T]}\norm{\hat{u}^{n}_r}^2\int_0^{T} \norm{\hat{u}^{n}_s}^2ds\right)^\frac{1}{2}\\
        &\leq c\nu^{\frac{1}{2}} + \nu^{\frac{1}{2}}\mathbbm{E}\left(\sup_{r\in[0,T]}\norm{\hat{u}^{n}_r}^2\right) + c\nu^{\frac{1}{2}}\mathbbm{E}\int_0^{T} \norm{\hat{u}^{n}_s}^2ds
    \end{align*}
    so that 
 \begin{align*}
\left(1 - \nu^{\frac{1}{2}}\right)\mathbbm{E}\left(\sup_{r\in[0,T]}\norm{u^{n}_{r}}^2\right) \leq \norm{u^n_0}^2 + c\nu + c\left(\nu + \nu^{\frac{1}{2}}\right)\int_0^{T} \mathbbm{E}\left(\norm{u^{n}_s}^2\right) ds
\end{align*}
and furthermore through dividing by $\left(1 - \nu^{\frac{1}{2}}\right)$ and rewriting with the $o_{\nu}$ notation,
 \begin{align*}
\mathbbm{E}\left(\sup_{r\in[0,T]}\norm{u^{n}_{r}}^2\right)  \leq \left(1 + o_{\nu}\right) \norm{u^n_0}^2 + o_{\nu} + o_{\nu}\int_0^{T} \mathbbm{E}\left(\norm{u^{n}_s}^2\right) ds
\end{align*}
having rewritten $\frac{1}{1-\nu^{1/2}}$ as $1 + \frac{\nu^{1/2}}{1-\nu^{1/2}}$. We now apply the standard Gr\"{o}nwall Inequality to deduce that
 \begin{align*}
\mathbbm{E}\left(\sup_{r\in[0,T]}\norm{u^{n}_{r}}^2\right) \leq e^{o_{\nu}}\left[\left(1 + o_{\nu}\right) \norm{u^n_0}^2 + o_{\nu} \right] \leq \left(1 + o_{\nu}\right) \norm{u^n_0}^2 + o_{\nu} 
\end{align*}
using the convergence $\lim_{\nu \rightarrow 0}e^{o_{\nu}} = 1$, demonstrating (\ref{hello2}) to conclude the justification of (\ref{hello1}). Similarly for (\ref{hello3}) it is sufficient to demonstrate that \begin{equation} \label{hello4}
        \mathbbm{E}\left[\norm{u^n_{t}}^2 + \nu\int_0^{t} \norm{u^n_s}^2_1 ds\right] \leq \left(1 + o_{\nu}\right) \norm{u^n_0}^2 + o_{\nu}.
         \end{equation}
The proof is near identical to (\ref{hello2}), though as we do not take the supremum then we take the expectation of the stochastic integral directly so this term is null, and we maintain the term coming from the Stokes Operator. Thus we arrive directly at
    \begin{align*}
\mathbbm{E}\left(\norm{u^{n}_{t}}^2\right) + \nu\mathbbm{E}\int_0^{t} \norm{u^{n}_s}^2_1 ds \leq \norm{u^n_0}^2 + c\nu + c\nu\int_0^{T} \mathbbm{E}\left(\norm{u^{n}_s}^2\right) ds
\end{align*}
from which the remainder of the proof follows as above, noting that $$\int_0^{T} \mathbbm{E}\left(\norm{u^{n}_s}^2\right) ds \leq \int_0^{T}\left[ \mathbbm{E}\left(\norm{u^{n}_s}^2\right) + \nu\mathbbm{E}\int_0^s\norm{u^n_r}_1^2dr\right] ds$$
so we can apply the Gr\"{o}nwall Inequality bounding the entirety of the left hand side as is required.

\end{proof}

\section{Appendix} \label{section appendix}

We present some supplementary results used in the paper. Lemma \ref{Lemma 5.2} guides the proof of Proposition \ref{prop for tightness}, and similarly so for Lemma \ref{lemma for D tight} serving Proposition \ref{prop for tightness two}. Theorem \ref{gagliardonirenberginequality} is classical and facilitates estimates on the nonlinear term, for example in Lemma \ref{lemma application of gag}. Lemma \ref{gronny} and Proposition \ref{rockner prop} are fundamental techniques in SPDE theory used throughout the paper. 

\begin{lemma} \label{Lemma 5.2}
    Let $\mathcal{H}_1, \mathcal{H}_2$ be Hilbert Spaces such that $\mathcal{H}_1$ is compactly embedded into $\mathcal{H}_2$, and for some fixed $T>0$ let $\sy^n: \Omega \times [0,T] \rightarrow \mathcal{H}_1$ be a sequence of measurable processes such that \begin{equation} \label{first condition} \sup_{n\in \N}\mathbbm{E}\int_0^T\norm{\sy^n_s}^2_{\mathcal{H}_1}ds < \infty\end{equation} and for any $\varepsilon > 0$, 
    \begin{equation}\label{second condition} \lim_{\delta \rightarrow 0^+}\sup_{n \in \N}\mathbbm{P}\left(\left\{\omega \in \Omega:\int_0^{T-\delta}\norm{\sy^n_{s + \delta}(\omega) - \sy^n_s(\omega)}^2_{\mathcal{H}_2}ds > \varepsilon\right\} \right) =0.\end{equation}
    Then the sequence of the laws of $(\sy^n)$ is tight in the space of probability measures over $L^2\left([0,T];\mathcal{H}_2\right)$.
\end{lemma}

\begin{proof}
    See [\cite{rockner2022well}] Lemma 5.2.
\end{proof}

\begin{lemma} \label{lemma for D tight}
    Let $\mathcal{Y}$ be a reflexive separable Banach Space and $\mathcal{H}$ a separable Hilbert Space such that $\mathcal{Y}$ is compactly embedded into $\mathcal{\mathcal{H}}$, and consider the induced Gelfand Triple
    $$\mathcal{Y} \xhookrightarrow{} \mathcal{H} \xhookrightarrow{} \mathcal{Y}^*. $$ For some fixed $T>0$ let $\sy^n: \Omega \rightarrow C\left([0,T];\mathcal{H}\right)$ be a sequence of measurable processes such that for every $t\in[0,T]$, \begin{equation} \label{first condition primed}
        \sup_{n \in \N}\mathbbm{E}\left(\sup_{t\in[0,T]}\norm{\sy^n_t}_{\mathcal{H}}\right) < \infty
    \end{equation}
    and for any sequence of stopping times $(\gamma_n)$ with $\gamma_n: \Omega \rightarrow [0,T]$, and any $\varepsilon > 0$, $y \in \mathcal{Y}$,
    \begin{equation} \label{second condition primed}
        \lim_{\delta \rightarrow 0^+}\sup_{n \in \N}\mathbbm{P}\left(\left\{
    \omega \in \Omega: \left\vert \left\langle \sy^n_{(\gamma_n + \delta) \wedge T} -\sy^n_{\gamma_n }   , y     \right\rangle_{\mathcal{H}} \right\vert > \varepsilon \right\}\right)  = 0.
    \end{equation}
    Then the sequence of the laws of $(\sy^n)$ is tight in the space of probability measures over $\mathcal{D}\left([0,T];\mathcal{Y}^*\right)$.
\end{lemma}

\begin{proof}
    We essentially combine the tightness criteria of [\cite{jakubowski1986skorokhod}] Theorem 3.1 and [\cite{aldous1978stopping}] Theorem 1, in the specific case outlined here. Firstly in reference to [\cite{jakubowski1986skorokhod}] Theorem 3.1 we may take $E$ to be $\mathcal{Y}^*$ (which is separable from the reflexivity and separability of $\mathcal{Y})$ and $\mathbbm{F}$ to be $\left(\mathcal{Y}^*\right)^*$, which is well known to separate points in $\mathcal{Y}^*$ from a corollary of the Hahn-Banach Theorem which asserts that for every $\phi \in \mathcal{Y}^*$ there exists a $\psi \in \left(\mathcal{Y}^*\right)^*$ such that $\inner{\phi}{\psi}_{\mathcal{Y}^* \times \left(\mathcal{Y}^*\right)^*} = \norm{\phi}_{\mathcal{Y}^*}$. We also note that condition $(3.3)$ in [\cite{jakubowski1986skorokhod}] is satisfied for $(\mu_n)$ taken to be the sequence of laws of $(\sy^n)$ over $\mathcal{D}\left([0,T];\mathcal{Y}^*\right)$, owing to the property (\ref{first condition primed}). Indeed as $\mathcal{Y}$ is compactly embedded into $\mathcal{H}$ then $\mathcal{H}$ is compactly embedded into $\mathcal{Y}^*$, so one only needs to take a bounded subset of $\mathcal{H}$ for this property (3.3). Considering the closed ball of radius $M$ in $\mathcal{H}$, $\tilde{B}_M$, we have that 
\begin{align*}
    \mathbbm{P}\left(\left\{\omega \in \Omega: \sy^n(\omega) \notin D\left([0,T] ;\tilde{B}_M\right) \right\} \right) &\leq \mathbbm{P}\left(\left\{\omega \in \Omega: \sy^n(\omega) \notin C\left([0,T] ;\tilde{B}_M\right) \right\} \right)\\ &\leq 
    \mathbbm{P}\left(\left\{\omega \in \Omega: \sup_{t\in[0,T]}\norm{\sy^n_t(\omega)}_{\mathcal{H}} > M \right\} \right)\\
    &\leq \frac{1}{M}\mathbbm{E}(\sup_{t\in[0,T]}\norm{\sy^n_t}_{\mathcal{H}})\\
    &\leq  \frac{1}{M}\sup_{n \in \N}\mathbbm{E}(\sup_{t\in[0,T]}\norm{\sy^n_t}_{\mathcal{H}})
\end{align*}
from which we see an arbitrarily large choice of $M$ will justify (3.3). Therefore by Theorem 3.1 it only remains to show that for every $\psi \in \left(\mathcal{Y}^*\right)^*$ the sequence of the laws of $\inner{\sy^n}{\psi}_{\mathcal{Y}^* \times \left(\mathcal{Y}^*\right)^*}$ is tight in the space of probability measures over $\mathcal{D}\left([0,T];\R\right)$. By the reflexivity of $\mathcal{Y}$ for every $\psi \in \left(\mathcal{Y}^*\right)^*$ there exists a $y \in \mathcal{Y}$ such that $\inner{\sy^n}{\psi}_{\mathcal{Y}^* \times \left(\mathcal{Y}^*\right)^*} = \inner{\sy^n}{y}_{\mathcal{Y}^* \times \mathcal{Y}}$ and as $\sy^n_t \in \mathcal{H}$ $\mathbbm{P}-a.s.$, then this is furthermore just $\inner{\sy^n}{y}_{\mathcal{H}}$. The problem is now reduced to showing tightness in $\mathcal{D}\left([0,T];\R\right)$, which by Theorem 1 of [\cite{aldous1978stopping}] is satisfied if we can show that for for any sequence of stopping times $(\gamma_n)$, $\gamma_n: \Omega \rightarrow [0,T]$, and constants $(\delta_n)$, $\delta_n \geq 0$ and $\delta_n \rightarrow 0$ as $n \rightarrow \infty$:
    \begin{enumerate}
        \item For every $t \in [0,T]$, the sequence of the laws of $\inner{\sy^n_t}{y}_{\mathcal{H}}$ is tight in the space of probability measures over $\R$, \label{item 1}

        \item For every $\varepsilon > 0$, $\lim_{n \rightarrow \infty}\mathbbm{P}\left( \left\{ \omega \in \Omega: \left\vert \left\langle \sy^n_{(\gamma_n + \delta_n) \wedge T} -\sy^n_{\gamma_n }   , y     \right\rangle_{\mathcal{H}} \right\vert > \varepsilon \right\} \right) = 0.$ \label{item 2}
    \end{enumerate}
We address each item in turn: as for \ref{item 1}, we are required to show that for every $\varepsilon > 0$ and $t\in[0,T]$, there exists a compact $K_{\varepsilon} \subset \R$ such that for every $n \in \N$, $$\mathbbm{P}\left(\left\{\omega \in \Omega: \inner{\sy^n_t(\omega)}{y}_{\mathcal{H}} \notin K_{\varepsilon} \right\} \right) < \varepsilon.$$ To this end define $B_M$ as the closed ball of radius $M$ in $\R$, then
\begin{align*}
    \mathbbm{P}\left(\left\{\omega \in \Omega: \inner{\sy^n_t(\omega)}{y}_{\mathcal{H}} \notin B_M \right\} \right) &= \mathbbm{P}\left(\left\{\omega \in \Omega: \left\vert\inner{\sy^n_t(\omega)}{y}_{\mathcal{H}}\right\vert > M \right\} \right)\\
    &\leq \frac{1}{M}\mathbbm{E}(\left\vert\inner{\sy^n_t}{y}_{\mathcal{H}}\right\vert)\\
    &\leq \frac{\norm{y}_{\mathcal{H}}}{M}\sup_{n \in \N}\mathbbm{E}\left(\norm{\sy^n_t}_{\mathcal{H}}\right)
\end{align*}
so setting $$M:= \frac{\varepsilon}{2\norm{y}_{\mathcal{H}}\sup_{n \in \N}\mathbbm{E}\left(\norm{\sy^n_t}_{\mathcal{H}}\right)} $$ justifies item \ref{item 1}. As for \ref{item 2}, note that for each fixed $j \in \N$ we have that $$\left\vert \left\langle \sy^j_{(\gamma_j + \delta_j) \wedge T} -\sy^j_{\gamma_j }   , y     \right\rangle_{\mathcal{H}} \right\vert  \leq \sup_{n \in \N}\left\vert \left\langle \sy^n_{(\gamma_n + \delta_j) \wedge T} -\sy^n_{\gamma_n }   , y     \right\rangle_{\mathcal{H}} \right\vert $$ so in particular $$\lim_{j \rightarrow \infty}\mathbbm{P}\left(\left\{\left\vert \left\langle \sy^j_{(\gamma_j + \delta_j) \wedge T} -\sy^j_{\gamma_j }   , y     \right\rangle_{\mathcal{H}} \right\vert > \varepsilon\right\}\right)  \leq \lim_{j \rightarrow \infty}\sup_{n \in \N}\mathbbm{P}\left(\left\{\left\vert \left\langle \sy^n_{(\gamma_n + \delta_j) \wedge T} -\sy^n_{\gamma_n }   , y     \right\rangle_{\mathcal{H}} \right\vert > \varepsilon \right\}\right).$$ As $(\delta_j)$ was an arbitrary sequence of non-negative constants approaching zero, we can generically take $\delta \rightarrow 0^+$ and \ref{item 2} is implied by (\ref{second condition primed}). The proof is complete.

\end{proof}

\begin{theorem}[Gagliardo-Nirenberg Inequality] \label{gagliardonirenberginequality}
    Let $p,q,\alpha \in \R$, $m \in \N$ be such that $p > q \geq 1$, $m > N(\frac{1}{2} - \frac{1}{p})$ and $\frac{1}{p} = \frac{\alpha}{q} + (1-\alpha)(\frac{1}{2} - \frac{m}{N})$. Then there exists a constant $c$ (dependent on the given parameters) such that for any $f \in L^p(\mathscr{O};\R) \cap W^{m,2}(\mathscr{O};\R)$, we have \begin{equation}\label{gag bounded domain}\norm{f}_{L^p(\mathscr{O};\R)} \leq c\norm{f}^{\alpha}_{L^q(\mathscr{O};\R)}\norm{f}^{1-\alpha}_{W^{m,2}(\mathscr{O};\R)}.\end{equation}
\end{theorem}

\begin{proof}
See [\cite{nirenberg2011elliptic}] pp.125-126. 
\end{proof}

\begin{remark}
    In the original paper [\cite{nirenberg2011elliptic}], the inequality is stated for only the $m^{\textnormal{th}}$ order derivative and with an additional $\norm{f}_{L^r}$ term on the bounded domain, for any $r > 0$. By considering the full $W^{m,2}(\mathscr{O};\R^N)$ norm, one can remove this additional term through interpolation. 
\end{remark}

\begin{lemma}[Stochastic Gr\"{o}nwall] \label{gronny}
Fix $t>0$ and suppose that $\boldsymbol{\phi},\boldsymbol{\psi}, \boldsymbol{\eta}$ are real-valued, non-negative stochastic processes. Assume, moreover, that there exists constants $c',\hat{c}, \tilde{c}$ (allowed to depend on $t$) such that for $\mathbbm{P}-a.e.$ $\omega$, \begin{equation} \label{boundingronny} \int_0^t\boldsymbol{\eta}_s(\omega) ds \leq c'\end{equation} and for all stopping times $0 \leq \theta_j < \theta_k \leq t$,
$$\mathbbm{E}\left(\sup_{r \in [\theta_j,\theta_k]}\boldsymbol{\phi}_r\right) + \mathbbm{E}\int_{\theta_j}^{\theta_k}\boldsymbol{\psi}_sds \leq \hat{c}\mathbbm{E}\left(\left[\boldsymbol{\phi}_{\theta_j} + \tilde{c} \right] + \int_{\theta_j}^{\theta_k} \boldsymbol{\eta}_s\boldsymbol{\phi}_sds\right) < \infty. $$Then there exists a constant $C$ dependent only on $c',\hat{c},\tilde{c},t$ such that $$\mathbbm{E}\sup_{r \in [0,t]}\boldsymbol{\phi}_r + \mathbbm{E}\int_{0}^{t}\boldsymbol{\psi}_sds \leq C\left[\mathbbm{E}(\boldsymbol{\phi}_{0}) + \tilde{c}\right].$$
\end{lemma}

\begin{proof}
See [\cite{glatt2009strong}] Lemma 5.3.
\end{proof}

\begin{proposition} \label{rockner prop}
Let $\mathcal{H}_1 \subset \mathcal{H}_2 \subset \mathcal{H}_3$ be a triplet of embedded Hilbert Spaces where $\mathcal{H}_1$ is dense in $\mathcal{H}_2$, with the property that there exists a continuous nondegenerate bilinear form $\inner{\cdot}{\cdot}_{\mathcal{H}_3 \times \mathcal{H}_1}: \mathcal{H}_3 \times \mathcal{H}_1 \rightarrow \R$ such that for $\phi \in \mathcal{H}_2$ and $\psi \in \mathcal{H}_1$, $$\inner{\phi}{\psi}_{\mathcal{H}_3 \times \mathcal{H}_1} = \inner{\phi}{\psi}_{\mathcal{H}_2}.$$ Suppose that for some $T > 0$ and stopping time $\tau$,
\begin{enumerate}
        \item $\sy_0:\Omega \rightarrow \mathcal{H}_2$ is $\mathcal{F}_0-$measurable;
        \item $\eta:\Omega \times [0,T] \rightarrow \mathcal{H}_3$ is such that for $\mathbbm{P}-a.e.$ $\omega$, $\eta(\omega) \in L^2([0,T];\mathcal{H}_3)$;
        \item $B:\Omega \times [0,T] \rightarrow \mathscr{L}^2(\mathfrak{U};\mathcal{H}_2)$ is progressively measurable and such that for $\mathbbm{P}-a.e.$ $\omega$, $B(\omega) \in L^2\left([0,T];\mathscr{L}^2(\mathfrak{U};\mathcal{H}_2)\right)$;
        \item  \label{4*} $\sy:\Omega \times [0,T] \rightarrow \mathcal{H}_1$ is such that for $\mathbbm{P}-a.e.$ $\omega$, $\sy_{\cdot}(\omega)\mathbbm{1}_{\cdot \leq \tau(\omega)} \in L^2([0,T];\mathcal{H}_1)$ and $\sy_{\cdot}\mathbbm{1}_{\cdot \leq \tau}$ is progressively measurable in $\mathcal{H}_1$;
        \item \label{item 5 again*} The identity
        \begin{equation} \label{newest identity*}
            \sy_t = \sy_0 + \int_0^{t \wedge \tau}\eta_sds + \int_0^{t \wedge \tau}B_s d\mathcal{W}_s
        \end{equation}
        holds $\mathbbm{P}-a.s.$ in $\mathcal{H}_3$ for all $t \in [0,T]$.
    \end{enumerate}
The the equality 
  \begin{align} \label{ito big dog*}\norm{\sy_t}^2_{\mathcal{H}_2} = \norm{\sy_0}^2_{\mathcal{H}_2} + \int_0^{t\wedge \tau} \bigg( 2\inner{\eta_s}{\sy_s}_{\mathcal{H}_3 \times \mathcal{H}_1} + \norm{B_s}^2_{\mathscr{L}^2(\mathfrak{U};\mathcal{H}_2)}\bigg)ds + 2\int_0^{t \wedge \tau}\inner{B_s}{\sy_s}_{\mathcal{H}_2}d\mathcal{W}_s\end{align}
  holds for any $t \in [0,T]$, $\mathbbm{P}-a.s.$ in $\R$. Moreover for $\mathbbm{P}-a.e.$ $\omega$, $\sy_{\cdot}(\omega) \in C([0,T];\mathcal{H}_2)$. 
\end{proposition}

\begin{proof}
This is a minor extension of [\cite{prevot2007concise}] Theorem 4.2.5, which is stated and justified as Proposition 2.5.5. in [\cite{goodair2022stochastic}]. The extension here is necessary for our purposes to avoid the need for moment estimates. 
\end{proof}

\textbf{Acknowledgements:} Daniel Goodair was supported by the Engineering and Physical Sciences Research Council (EPSCR) Project 2478902. Dan Crisan was partially supported by the European Research Council (ERC) under the European Union's Horizon 2020 Research and Innovation Programme (ERC, Grant Agreement No 856408).

\bibliographystyle{spmpsci}
\bibliography{myBibliography2}

\begin{thebibliography}{10}
\providecommand{\url}[1]{{#1}}
\providecommand{\urlprefix}{URL }
\expandafter\ifx\csname urlstyle\endcsname\relax
  \providecommand{\doi}[1]{DOI~\discretionary{}{}{}#1}\else
  \providecommand{\doi}{DOI~\discretionary{}{}{}\begingroup
  \urlstyle{rm}\Url}\fi

\bibitem{aldous1978stopping}
Aldous, D.: Stopping times and tightness.
\newblock The Annals of Probability pp. 335--340 (1978)

\bibitem{alonso2020modelling}
Alonso-Or{\'a}n, D., Bethencourt~de Le{\'o}n, A., Holm, D.D., Takao, S.:
  Modelling the climate and weather of a 2d lagrangian-averaged
  euler--boussinesq equation with transport noise.
\newblock Journal of Statistical Physics \textbf{179}(5), 1267--1303 (2020)

\bibitem{arakeri2000ludwig}
Arakeri, J.H., Shankar, P.: Ludwig prandtl and boundary layers in fluid flow:
  How a small viscosity can cause large effects.
\newblock Resonance \textbf{5}(12), 48--63 (2000)

\bibitem{attanasio2011renormalized}
Attanasio, S., Flandoli, F.: Renormalized solutions for stochastic transport
  equations and the regularization by bilinear multiplicative noise.
\newblock Communications in Partial Differential Equations \textbf{36}(8),
  1455--1474 (2011)

\bibitem{bertozzi2002vorticity}
Bertozzi, A.L., Majda, A.J.: Vorticity and incompressible flow, \emph{Cambridge
  Texts in Applied Mathematics}, vol.~27.
\newblock Cambridge University Press, Cambridge (2002)

\bibitem{bessaih2013inviscid}
Bessaih, H., Ferrario, B.: Inviscid limit of stochastic damped 2d
  navier--stokes equations.
\newblock Nonlinearity \textbf{27}(1), 1 (2013)

\bibitem{billingsley2013convergence}
Billingsley, P.: Convergence of probability measures.
\newblock John Wiley \& Sons (2013)

\bibitem{bourguignon1974remarks}
Bourguignon, J.P., Brezis, H.: Remarks on the euler equation.
\newblock Journal of functional analysis \textbf{15}(4), 341--363 (1974)

\bibitem{brzezniak1992stochastic}
Brze{\'z}niak, Z., Capi{\'n}ski, M., Flandoli, F.: Stochastic navier-stokes
  equations with multiplicative noise.
\newblock Stochastic Analysis and Applications \textbf{10}(5), 523--532 (1992)

\bibitem{brzezniak2021well}
Brze{\'z}niak, Z., Slavik, J.: Well-posedness of the 3d stochastic primitive
  equations with multiplicative and transport noise.
\newblock Journal of Differential Equations \textbf{296}, 617--676 (2021)

\bibitem{cebeci1977momentum}
Cebeci, T., Bradshaw, P.: Momentum transfer in boundary layers.
\newblock Washington  (1977)

\bibitem{cipriano2015inviscid}
Cipriano, F., Torrecilla, I.: Inviscid limit for 2d stochastic navier--stokes
  equations.
\newblock Stochastic Processes and their Applications \textbf{125}(6),
  2405--2426 (2015)

\bibitem{clauser1956turbulent}
Clauser, F.H.: The turbulent boundary layer.
\newblock Advances in applied mechanics \textbf{4}, 1--51 (1956)

\bibitem{constantin1988navier}
Constantin, P., Foias, C.: Navier-Stokes Equations.
\newblock University of Chicago Press (1988)

\bibitem{cotter2020data}
Cotter, C., Crisan, D., Holm, D., Pan, W., Shevchenko, I.: Data assimilation
  for a quasi-geostrophic model with circulation-preserving stochastic
  transport noise.
\newblock Journal of Statistical Physics \textbf{179}(5), 1186--1221 (2020)

\bibitem{cotter2018modelling}
Cotter, C., Crisan, D., Holm, D.D., Pan, W., Shevchenko, I.: Modelling
  uncertainty using stochastic transport noise in a 2-layer quasi-geostrophic
  model.
\newblock arXiv preprint arXiv:1802.05711  (2018)

\bibitem{cotter2019numerically}
Cotter, C., Crisan, D., Holm, D.D., Pan, W., Shevchenko, I.: Numerically
  modeling stochastic lie transport in fluid dynamics.
\newblock Multiscale Modeling \& Simulation \textbf{17}(1), 192--232 (2019)

\bibitem{crisan2021theoretical}
Crisan, D., Holm, D.D., Luesink, E., Mensah, P.R., Pan, W.: Theoretical and
  computational analysis of the thermal quasi-geostrophic model.
\newblock arXiv preprint arXiv:2106.14850  (2021)

\bibitem{da2014stochastic}
Da~Prato, G., Zabczyk, J.: Stochastic equations in infinite dimensions, vol.
  152.
\newblock Cambridge university press (2014)

\bibitem{day1990no}
Day, M.A.: The no-slip condition of fluid dynamics.
\newblock Erkenntnis \textbf{33}(3), 285--296 (1990)

\bibitem{debussche2023consistent}
Debussche, A., Hug, B., M{\'e}min, E.: A consistent stochastic large-scale
  representation of the navier--stokes equations.
\newblock Journal of Mathematical Fluid Mechanics \textbf{25}(1), 19 (2023)

\bibitem{dovgal1994laminar}
Dovgal, A., Kozlov, V., Michalke, A.: Laminar boundary layer separation:
  instability and associated phenomena.
\newblock Progress in Aerospace Sciences \textbf{30}(1), 61--94 (1994)

\bibitem{dufee2022stochastic}
Duf{\'e}e, B., M{\'e}min, E., Crisan, D.: Stochastic parametrization: an
  alternative to inflation in ensemble kalman filters.
\newblock Quarterly Journal of the Royal Meteorological Society
  \textbf{148}(744), 1075--1091 (2022)

\bibitem{evans2010partial}
Evans, L.C.: Partial differential equations, vol.~19.
\newblock American Mathematical Soc. (2010)

\bibitem{flandoli2013topics}
Flandoli, F.: Topics on regularization by noise.
\newblock Lecture Notes, University of Pisa  (2013)

\bibitem{flandoli2015open}
Flandoli, F.: An open problem in the theory of regularization by noise for
  nonlinear pdes.
\newblock In: Workshop Classic and Stochastic Geometric Mechanics, pp. 13--29.
  Springer (2015)

\bibitem{flandoli2021delayed}
Flandoli, F., Galeati, L., Luo, D.: Delayed blow-up by transport noise.
\newblock Communications in Partial Differential Equations \textbf{46}(9),
  1757--1788 (2021)

\bibitem{flandoli2021high}
Flandoli, F., Luo, D.: High mode transport noise improves vorticity blow-up
  control in 3d navier--stokes equations.
\newblock Probability Theory and Related Fields \textbf{180}(1), 309--363
  (2021)

\bibitem{flandoli20212d}
Flandoli, F., Pappalettera, U.: 2d euler equations with stratonovich transport
  noise as a large-scale stochastic model reduction.
\newblock Journal of Nonlinear Science \textbf{31}(1), 1--38 (2021)

\bibitem{flandoli2022additive}
Flandoli, F., Pappalettera, U.: From additive to transport noise in 2d fluid
  dynamics.
\newblock Stochastics and Partial Differential Equations: Analysis and
  Computations pp. 1--41 (2022)

\bibitem{gao2019existence}
Gao, X., Gao, H.: Existence and uniqueness of weak solutions to stochastic 3d
  navier--stokes equations with delays.
\newblock Applied Mathematics Letters \textbf{95}, 158--164 (2019)

\bibitem{gie2012boundary}
Gie, G.M., Kelliher, J.P.: Boundary layer analysis of the navier--stokes
  equations with generalized navier boundary conditions.
\newblock Journal of Differential Equations \textbf{253}(6), 1862--1892 (2012)

\bibitem{glatt2015inviscid}
Glatt-Holtz, N., {\v{S}}ver{\'a}k, V., Vicol, V.: On inviscid limits for the
  stochastic navier--stokes equations and related models.
\newblock Archive for Rational Mechanics and Analysis \textbf{217}(2), 619--649
  (2015)

\bibitem{glatt2009strong}
Glatt-Holtz, N., Ziane, M., et~al.: Strong pathwise solutions of the stochastic
  navier-stokes system.
\newblock Advances in Differential Equations \textbf{14}(5/6), 567--600 (2009)

\bibitem{goodair2022stochastic}
Goodair, D.: Stochastic calculus in infinite dimensions and spdes.
\newblock arXiv preprint arXiv:2203.17206  (2022)

\bibitem{goodair2022navier}
Goodair, D., Crisan, D.: On the navier-stokes equations with stochastic lie
  transport.
\newblock arXiv preprint arXiv:2211.01265  (2022)

\bibitem{goodair2022existence1}
Goodair, D., Crisan, D., Lang, O.: Existence and uniqueness of maximal
  solutions to spdes with applications to viscous fluid equations.
\newblock arXiv preprint arXiv:2209.09137  (2022)

\bibitem{gorbushin2018asymptotic}
Gorbushin, A., Zametaev, V.: Asymptotic analysis of viscous fluctuations in
  turbulent boundary layers.
\newblock Fluid Dynamics \textbf{53}, 9--20 (2018)

\bibitem{holm2015variational}
Holm, D.D.: Variational principles for stochastic fluid dynamics.
\newblock Proceedings of the Royal Society A: Mathematical, Physical and
  Engineering Sciences \textbf{471}(2176), 20140,963 (2015)

\bibitem{holm2019stochastic}
Holm, D.D., Luesink, E.: Stochastic wave--current interaction in thermal
  shallow water dynamics.
\newblock Journal of Nonlinear Science \textbf{31}(2), 1--56 (2021)

\bibitem{holm2020stochastic}
Holm, D.D., Luesink, E., Pan, W.: Stochastic circulation dynamics in the ocean
  mixed layer.
\newblock arXiv preprint arXiv:2006.05707  (2020)

\bibitem{iftimie2011viscous}
Iftimie, D., Sueur, F.: Viscous boundary layers for the navier--stokes
  equations with the navier slip conditions.
\newblock Archive for rational mechanics and analysis \textbf{199}(1), 145--175
  (2011)

\bibitem{jakubowski1986skorokhod}
Jakubowski, A.: On the skorokhod topology.
\newblock In: Annales de l'IHP Probabilit{\'e}s et statistiques, 3, pp.
  263--285 (1986)

\bibitem{kato1984remarks}
Kato, T.: Remarks on zero viscosity limit for nonstationary navier-stokes flows
  with boundary.
\newblock In: Seminar on nonlinear partial differential equations, pp. 85--98.
  Springer (1984)

\bibitem{keller1978numerical}
Keller, H.B.: Numerical methods in boundary-layer theory.
\newblock Annual Review of Fluid Mechanics \textbf{10}(1), 417--433 (1978)

\bibitem{kelliher2006navier}
Kelliher, J.P.: Navier--stokes equations with navier boundary conditions for a
  bounded domain in the plane.
\newblock SIAM journal on mathematical analysis \textbf{38}(1), 210--232 (2006)

\bibitem{kelliher2007kato}
Kelliher, J.P.: On kato's conditions for vanishing viscosity.
\newblock Indiana University Mathematics Journal pp. 1711--1721 (2007)

\bibitem{kuksin2005family}
Kuksin, S., Penrose, O.: A family of balance relations for the two-dimensional
  navier--stokes equations with random forcing.
\newblock Journal of statistical physics \textbf{118}(3-4), 437--449 (2005)

\bibitem{kuksin2004eulerian}
Kuksin, S.B.: The eulerian limit for 2d statistical hydrodynamics.
\newblock Journal of statistical physics \textbf{115}, 469--492 (2004)

\bibitem{kuksin2006remarks}
Kuksin, S.B.: Remarks on the balance relations for the two-dimensional
  navier--stokes equation with random forcing.
\newblock Journal of statistical physics \textbf{122}(1), 101--114 (2006)

\bibitem{kuksin2008distribution}
Kuksin, S.B.: On distribution of energy and vorticity for solutions of 2d
  navier-stokes equation with small viscosity.
\newblock Communications in mathematical physics \textbf{284}, 407--424 (2008)

\bibitem{lang2022pathwise}
Lang, O., Pan, W.: A pathwise parameterisation for stochastic transport.
\newblock arXiv preprint arXiv:2202.10852  (2022)

\bibitem{lee2013effect}
Lee, J., Jung, S.Y., Sung, H.J., Zaki, T.A.: Effect of wall heating on
  turbulent boundary layers with temperature-dependent viscosity.
\newblock Journal of Fluid Mechanics \textbf{726}, 196--225 (2013)

\bibitem{van2021bayesian}
van Leeuwen, P.J., Crisan, D., Lang, O., Potthast, R.: Bayesian inference for
  fluid dynamics: A case study for the stochastic rotating shallow water model.
\newblock arXiv preprint arXiv:2112.15216  (2021)

\bibitem{bethencourt2022transport}
Bethencourt-de Le{\'o}n, A., Takao, S.: Transport noise restores uniqueness and
  prevents blow-up in geometric transport equations.
\newblock arXiv preprint arXiv:2211.14695  (2022)

\bibitem{lototsky2017stochastic}
Lototsky, S.V., Rozovsky, B.L., et~al.: Stochastic partial differential
  equations.
\newblock Springer (2017)

\bibitem{luongo2021inviscid}
Luongo, E.: Inviscid limit for stochastic navier-stokes equations under general
  initial conditions.
\newblock arXiv preprint arXiv:2111.14189  (2021)

\bibitem{lyman1990vorticity}
Lyman, F.: Vorticity production at a solid boundary.
\newblock Appl. Mech. Rev \textbf{43}(8), 157--158 (1990)

\bibitem{maekawa2016inviscid}
Maekawa, Y., Mazzucato, A.: The inviscid limit and boundary layers for
  navier-stokes flows.
\newblock arXiv preprint arXiv:1610.05372  (2016)

\bibitem{malik1990numerical}
Malik, M.R.: Numerical methods for hypersonic boundary layer stability.
\newblock Journal of computational physics \textbf{86}(2), 376--413 (1990)

\bibitem{memin2014fluid}
M{\'e}min, E.: Fluid flow dynamics under location uncertainty.
\newblock Geophysical \& Astrophysical Fluid Dynamics \textbf{108}(2), 119--146
  (2014)

\bibitem{metivier2004small}
M{\'e}tivier, G.: Small viscosity and boundary layer methods: theory, stability
  analysis, and applications.
\newblock Springer Science \& Business Media (2004)

\bibitem{mikulevicius2005global}
Mikulevicius, R., Rozovskii, B.L.: Global {$L_2$}-solutions of stochastic
  {N}avier--{S}tokes equations.
\newblock Ann. Probab. \textbf{33}(1), 137--176 (2005).
\newblock \doi{10.1214/009117904000000630}.
\newblock \urlprefix\url{http://dx.doi.org/10.1214/009117904000000630}

\bibitem{moore1963boundary}
Moore, D.: The boundary layer on a spherical gas bubble.
\newblock Journal of Fluid Mechanics \textbf{16}(2), 161--176 (1963)

\bibitem{morris1984boundary}
Morris, S., Canright, D.: A boundary-layer analysis of benard convection in a
  fluid of strongly temperature-dependent viscosity.
\newblock Physics of the Earth and planetary interiors \textbf{36}(3-4),
  355--373 (1984)

\bibitem{morton1984generation}
Morton, B.R.: The generation and decay of vorticity.
\newblock Geophysical \& Astrophysical Fluid Dynamics \textbf{28}(3-4),
  277--308 (1984)

\bibitem{nguyen2021nonlinear}
Nguyen, P., Tawri, K., Temam, R.: Nonlinear stochastic parabolic partial
  differential equations with a monotone operator of the
  ladyzenskaya-smagorinsky type, driven by a l{\'e}vy noise.
\newblock Journal of Functional Analysis \textbf{281}(8), 109,157 (2021)

\bibitem{nirenberg2011elliptic}
Nirenberg, L.: On elliptic partial differential equations.
\newblock In: Il principio di minimo e sue applicazioni alle equazioni
  funzionali, pp. 1--48. Springer (2011)

\bibitem{prandtl1904flussigkeitsbewegung}
Prandtl, L.: {\"U}ber flussigkeitsbewegung bei sehr kleiner reibung.
\newblock Verhandl. III, Internat. Math.-Kong., Heidelberg, Teubner, Leipzig,
  1904 pp. 484--491 (1904)

\bibitem{prevot2007concise}
Pr{\'e}v{\^o}t, C., R{\"o}ckner, M.: A concise course on stochastic partial
  differential equations, \emph{Lecture Notes in Mathematics}, vol. 1905.
\newblock Springer, Berlin (2007)

\bibitem{richardson1973no}
Richardson, S.: On the no-slip boundary condition.
\newblock Journal of Fluid Mechanics \textbf{59}(4), 707--719 (1973)

\bibitem{robinson2016three}
Robinson, J.C., Rodrigo, J.L., Sadowski, W.: The three-dimensional
  Navier--Stokes equations: Classical theory, vol. 157.
\newblock Cambridge university press (2016)

\bibitem{rockner2008yamada}
R{\"o}ckner, M., Schmuland, B., Zhang, X.: Yamada-watanabe theorem for
  stochastic evolution equations in infinite dimensions.
\newblock Condensed Matter Physics  (2008)

\bibitem{rockner2022well}
R{\"o}ckner, M., Shang, S., Zhang, T.: Well-posedness of stochastic partial
  differential equations with fully local monotone coefficients.
\newblock arXiv preprint arXiv:2206.01107  (2022)

\bibitem{ruckenstein1983no}
Ruckenstein, E., Rajora, P.: On the no-slip boundary condition of
  hydrodynamics.
\newblock Journal of colloid and interface science \textbf{96}(2), 488--491
  (1983)

\bibitem{sanjose2019modal}
Sanjose, M., Towne, A., Jaiswal, P., Moreau, S., Lele, S., Mann, A.: Modal
  analysis of the laminar boundary layer instability and tonal noise of an
  airfoil at reynolds number 150,000.
\newblock International Journal of Aeroacoustics \textbf{18}(2-3), 317--350
  (2019)

\bibitem{schetz2011boundary}
Schetz, J.A., Bowersox, R.D.: Boundary layer analysis.
\newblock American Institute of Aeronautics and Astronautics (2011)

\bibitem{serpelloni2013vertical}
Serpelloni, E., Faccenna, C., Spada, G., Dong, D., Williams, S.D.: Vertical gps
  ground motion rates in the euro-mediterranean region: New evidence of
  velocity gradients at different spatial scales along the nubia-eurasia plate
  boundary.
\newblock Journal of Geophysical Research: Solid Earth \textbf{118}(11),
  6003--6024 (2013)

\bibitem{shaing2008collisional}
Shaing, K.C., Cahyna, P., Becoulet, M., Park, J.K., Sabbagh, S., Chu, M.:
  Collisional boundary layer analysis for neoclassical toroidal plasma
  viscosity in tokamaks.
\newblock Physics of Plasmas \textbf{15}(8), 082,506 (2008)

\bibitem{simpson1989turbulent}
Simpson, R.L.: Turbulent boundary-layer separation.
\newblock Annual Review of Fluid Mechanics \textbf{21}(1), 205--232 (1989)

\bibitem{stevenson2002incipient}
Stevenson, P., Thorpe, R., Davidson, J.: Incipient motion of a small particle
  in the viscous boundary layer at a pipe wall.
\newblock Chemical Engineering Science \textbf{57}(21), 4505--4520 (2002)

\bibitem{street2021semi}
Street, O.D., Crisan, D.: Semi-martingale driven variational principles.
\newblock Proceedings of the Royal Society A \textbf{477}(2247), 20200,957
  (2021)

\bibitem{tani1962production}
Tani, I.: Production of longitudinal vortices in the boundary layer along a
  concave wall.
\newblock Journal of Geophysical Research \textbf{67}(8), 3075--3080 (1962)

\bibitem{taniguchi2014global}
Taniguchi, T.: Global existence of a weak solution to 3d stochastic
  navier--stokes equations in an exterior domain.
\newblock Nonlinear Differential Equations and Applications NoDEA \textbf{21},
  813--840 (2014)

\bibitem{temam2001navier}
Temam, R.: Navier-Stokes equations: theory and numerical analysis, vol. 343.
\newblock American Mathematical Soc. (2001)

\bibitem{temam2002boundary}
Temam, R., Wang, X.: Boundary layers associated with incompressible
  navier--stokes equations: the noncharacteristic boundary case.
\newblock Journal of Differential Equations \textbf{179}(2), 647--686 (2002)

\bibitem{timoshin1997instabilities}
Timoshin, S.: Instabilities in a high-reynolds-number boundary layer on a
  film-coated surface.
\newblock Journal of Fluid Mechanics \textbf{353}, 163--195 (1997)

\bibitem{wallace2010measurement}
Wallace, J.M., Vukoslav{\v{c}}evi{\'c}, P.V.: Measurement of the velocity
  gradient tensor in turbulent flows.
\newblock Annual review of fluid mechanics \textbf{42}, 157--181 (2010)

\bibitem{wang2001kato}
Wang, X.: A kato type theorem on zero viscosity limit of navier-stokes flows.
\newblock Indiana University Mathematics Journal pp. 223--241 (2001)

\bibitem{wang2023kato}
Wang, Y.G., Zhao, M.: On kato's conditions for the inviscid limit of the
  two-dimensional stochastic navier-stokes equation.
\newblock arXiv preprint arXiv:2303.06578  (2023)

\bibitem{weinan2000boundary}
Weinan, E.: Boundary layer theory and the zero-viscosity limit of the
  navier-stokes equation.
\newblock Acta Mathematica Sinica \textbf{16}, 207--218 (2000)

\end{thebibliography}

\end{document}